\documentclass{article}
\usepackage{amsmath}
\usepackage{amsthm}
\usepackage{amssymb}
\usepackage{esint}
\usepackage{amsthm}
\usepackage{booktabs}
\usepackage{bbm}
\usepackage{cite}
\usepackage{stmaryrd}
\usepackage{xcolor}
\SetSymbolFont{stmry}{bold}{U}{stmry}{m}{n}
\usepackage{euscript}
\usepackage[arrow,curve,matrix,arc,2cell]{xy}
\UseAllTwocells
\usepackage[utf8]{inputenc}
\usepackage[unicode]{hyperref}
\usepackage{slashed}
\DeclareFontFamily{U}{rsfs}{} 
\DeclareFontShape{U}{rsfs}{n}{it}{<->
rsfs10}{} \DeclareSymbolFont{mscr}{U}{rsfs}{n}{it}
\DeclareSymbolFontAlphabet{\scr}{mscr}
\def\mathscr{\scr}
\usepackage{cleveref}
\usepackage{svg}
\def\Ker{\mathop{\rm Ker}}

\newcommand{\Ad}{\operatorname{Ad}}
\newcommand{\id}{\operatorname{id}}

\newcommand{\fix}{\operatorname{fix}}

\newcommand{\vol}{\operatorname{vol}}
\newcommand{\hol}{\operatorname{hol}}
\renewcommand{\d}{\operatorname{d}}
\newcommand{\tensor}{\otimes}

\def\C{\mathbb{C}}
\def\R{\mathbb{R}}
\def\Z{\mathbb{Z}}

\renewcommand{\|}[1]{\left| \left| #1 \right| \right|}

\newcommand{\kuro}{\color{black}}

\theoremstyle{plain}
\newtheorem{theorem}{Theorem}[section]
\crefname{theorem}{Theorem}{Theorems}
\newtheorem{proposition}[theorem]{Proposition}
\crefname{proposition}{Proposition}{Propositions}
\newtheorem{lemma}[theorem]{Lemma}
\crefname{lemma}{Lemma}{Lemmas}
\newtheorem{corollary}[theorem]{Corollary}
\crefname{corollary}{Corollary}{Corollaries}

\crefname{assertion}{Assertion}{Assertions}

\crefname{property}{Property}{Properties}

\crefname{assumption}{Assumption}{Assumptions}

\theoremstyle{definition}
\newtheorem{definition}[theorem]{Definition}
\crefname{definition}{Definition}{Definitions}
\newtheorem{remark}[theorem]{Remark}
\crefname{remark}{Remark}{Remarks}

\crefname{problem}{Problem}{Problems}

\crefname{example}{Example}{Examples}

\crefname{examples}{Examples}{Examples}

\crefname{open}{Open question}{Open questions}

\crefname{question}{Question}{Questions}

\crefname{exercise}{Exercise}{Exercises}
\newtheorem{convention}[theorem]{Convention}
\crefname{convention}{Convention}{Conventions}
\makeatletter
\@addtoreset{equation}{section}
\makeatother
\usepackage{enumitem}
\usepackage{tikz-cd}
\begin{document}

%%%%%%%%%%%%%%%%%%%%%%%%%%%%%%%%%%%%%%%%%%%%%%%%%%%%%%%%%%%%%%%%%%%%%%%%%%%%%%%%%%%%%%%%%%%%%%%%%%%%%
%%%%%%%%%%%%%%%%%%%%%%%%%%%%%%%%%%%%%%%%%%%%%%%%%%%%%%%%%%%%%%%%%%%%%%%%%%%%%%%%%%%%%%%%%%%%%%%%%%%%%

\title{Spin(7)-instantons on Joyce's first examples of compact Spin(7)-manifolds} 
%%%%%%%%%%%%%%%%%%%%%%%%%%%%%%%%%%%%%%%%%%%%%%%%%%%%%%%%%%%%%%%%%%%%%%%%%%%%%%%%%%%%%%%%%%%%%%%%%%%%%
\author{  
Mateo Galdeano, Daniel Platt, Yuuji Tanaka, and Luya Wang}
\date{}

%%%%%%%%%%%%%%%%%%%%%%%%%%%%%%%%%%%%%%%%%%%%%%%%%%%%%%%%%%%%%%%%%%%%%%%%%%%%%%%%%%%%%%%%%%%%%%%%%%%%%%

\maketitle

%%%%%%%%%%%%%%%%%%%%%%%%%%%%%%%%%%%%%%%%%%%%%%%%%%%%%%%%%%%%%%%%%%%%%%%%%%%%%%%%%%%%%%%%%%%%%%%%%%%%%
\begin{abstract}
We construct $Spin(7)$-instantons on one of Joyce's compact $Spin(7)$-manifolds.
The underlying compact $Spin(7)$-manifold given by Joyce is the same as in Lewis' construction of $Spin(7)$-instantons.  However, our construction method and the resulting instantons are new.
The compact $Spin(7)$-manifold is constructed by gluing a $Spin(7)$-orbifold and certain local model spaces around the orbifold singularities. We construct our instantons by gluing non-flat connections on the local model spaces to a flat connection on the $Spin(7)$-orbifold. 
We deliver more than $20,000$ new four-parameter families of examples of $Spin(7)$-instantons within the structure groups $SO(3), SO(4), SO(5), SO(7)$, and $SO(8)$. 
\end{abstract}
%%%%%%%%%%%%%%%%%%%%%%%%%%%%%%%%%%%%%%%%%%%%%%%%%%%%%%%%%%%%%%%%%%%%%%%%%%%%%%%%%%%%%%%%%%%%%%%%%%%%%%

\setcounter{tocdepth}{3}
\tableofcontents

\section{Introduction}
\label{q1}

$Spin(7)$-instantons are Yang--Mills connections on a $Spin(7)$-manifold, which minimise the Yang--Mills functional and are solutions to a system of non-linear first-order elliptic equations. 
They are among the main objects in Gauge theory in higher dimensions as proposed by Donaldson--Thomas \cite{DoTh} and Donaldson--Segal \cite{DoSe}. 
$Spin(7)$-instantons have similar deformation properties to anti-self-dual instantons on 4-manifolds.
In fact, pursuing the algebro-geometric counterpart of them has recently become one of the hot topics in the research area of virtual enumerative invariants on Calabi--Yau four-folds (see e.g. Borisov--Joyce \cite{BoJo} and Oh--Thomas \cite{OhTh, OhTh2} for foundational works).

However, producing examples of these $Spin(7)$-instantons is known to be a very difficult problem, especially, if the underlying manifolds are compact, 
because of the extensive analysis involved, already apparent in the construction of the underlying compact $Spin(7)$-manifolds, as well as the limited flexibility in achieving surjectivity of the linearised operator in contrast to the situation in lower-dimensional gauge theory. 
In this article, we construct them on an example of compact $Spin(7)$-manifolds by Joyce \cite{Joyc3}. 
We exploit the explicit geometric nature of this example of the underlying compact $Spin(7)$-manifold, which enables us to simultaneously glue rigid Hermitian-Yang--Mills connections on Eguchi--Hanson spaces, one of the building blocks of this example, with rigid background flat connections on the manifold. This is carried out with refined analysis using weighted H\"older norms.

\paragraph{Joyce's examples of compact Spin(7)-manifolds.} The example of compact $Spin(7)$-manifold on which we construct $Spin(7)$-instantons was constructed by Joyce in \cite{Joyc3} (see also \cite{Joyc5}). 
These are the first examples of compact $Spin(7)$-manifolds in the literature decades after Berger's prediction \cite{Berg} of the existence of such manifolds. 
These manifolds are obtained by resolving singularities of $T^8 / \Gamma$, where $T^8$ is an eight-dimensional torus, and $\Gamma$ is a finite subgroup of $\text{Aut} (T^8)$. 
Differential-geometrically, this resolution is given by gluing in Eguchi--Hanson spaces. 
Then Joyce's elaborate analysis gives a family of torsion-free $Spin(7)$-structures on the resolution. 
See Section \ref{subsection:the-manifolds-by-Joyce} for a detailed description of a particular type of them, on which we construct $Spin(7)$-instantons in this paper.

\paragraph{Lewis' construction of Spin(7)-instantons.}

Lewis \cite{Lewi} constructed $Spin(7)$-instantons on one of Joyce's examples of compact $Spin(7)$-manifolds.  
His construction is to place Hermitian-Yang--Mills connections on a smooth K3 surface in a partial resolution of $T^8 / \Gamma$, which becomes a Cayley submanifold after the resolution. He glues then them to the trivial connection, and perturbs the resulting connection alongside the underlying $Spin(7)$-structure of Joyce.  
Here, a Cayley submanifold in a $Spin(7)$-manifold is a real four-dimensional submanifold calibrated by the four-form determining the $Spin(7)$-structure. 
This can be thought of as a Taubes style construction of $Spin(7)$-instantons, since, as Lewis and Tian \cite{Tian} find out, a sequence of $Spin(7)$-instantons may bubble out along a Cayley submanifold, after taking a subsequence and gauge transformations. 
Later, Walpuski \cite{Walp2} beautifully worked out the Taubes construction of $Spin(7)$-instantons, and Lewis's construction can be seen as an example of the Walpuski construction.

\paragraph{Our results.}

We work on the same examples as Lewis did. 
By resolving the orbifold singularity of $T^8/ \Gamma$, Joyce obtains a one-parameter family of smooth $Spin(7)$-manifolds, $M_t$, for $0<t \ll 1$. 
Here, the manifold $M_t$ tends to $T^8 / \Gamma$ as $t \rightarrow 0$.
Some of the singularities in the orbifold $T^8/\Gamma$ have a resolution whose local model is given by $\R^4 \times X$, where $X$ is the Eguchi--Hanson space (see Section \ref{subsection:the-manifolds-by-Joyce} for more details).  While Lewis glues in pullbacks of ASD connections which are concentrated away from the singular sets of the orbifold, we glue in pullbacks of ASD connections over $X$ near the singularities of $T^8/\Gamma$ and a flat connection on the orbifold $T^8/\Gamma$ together with the underlying $Spin(7)$-structures on $M^8$ by Joyce.  (See Figure \ref{fig:ours}). 

The following is our main theorem in this paper, which will be proved as Theorem \ref{th:analytic_construction} at the end of Section \ref{subsection:perturbation-step}: 

%\vspace{0.3cm}
\begin{figure}[hbtp]
    \centering
%    \includesvg{}
    \begingroup%
  \makeatletter%
  \providecommand\color[2][]{%
    \errmessage{(Inkscape) Color is used for the text in Inkscape, but the package 'color.sty' is not loaded}%
    \renewcommand\color[2][]{}%
  }%
  \providecommand\transparent[1]{%
    \errmessage{(Inkscape) Transparency is used (non-zero) for the text in Inkscape, but the package 'transparent.sty' is not loaded}%
    \renewcommand\transparent[1]{}%
  }%
  \providecommand\rotatebox[2]{#2}%
  \newcommand*\fsize{\dimexpr\f@size pt\relax}%
  \newcommand*\lineheight[1]{\fontsize{\fsize}{#1\fsize}\selectfont}%
  \ifx\svgwidth\undefined%
    \setlength{\unitlength}{327.27542024bp}%
    \ifx\svgscale\undefined%
      \relax%
    \else%
      \setlength{\unitlength}{\unitlength * \real{\svgscale}}%
    \fi%
  \else%
    \setlength{\unitlength}{\svgwidth}%
  \fi%
  \global\let\svgwidth\undefined%
  \global\let\svgscale\undefined%
  \makeatother%
  \begin{picture}(1,0.6000766)%
    \lineheight{1}%
    \setlength\tabcolsep{0pt}%
    \put(0,0){\includegraphics[width=\unitlength,page=1]{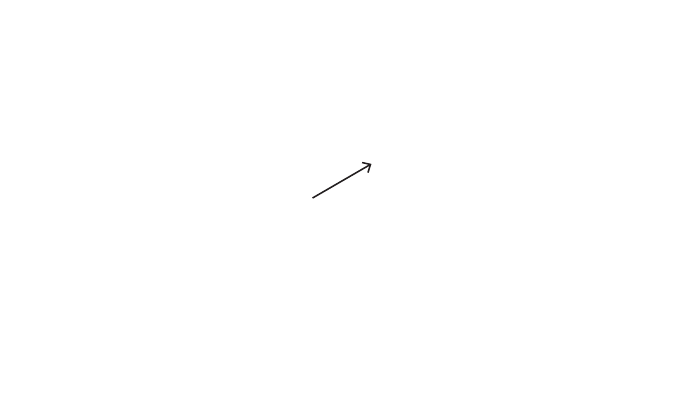}}%
    \put(0.10229916,0.14674262){\color[rgb]{0.1372549,0.12156863,0.1254902}\makebox(0,0)[lt]{\lineheight{1.25}\smash{\begin{tabular}[t]{l}$T^8/\Gamma$\end{tabular}}}}%
    \put(0,0){\includegraphics[width=\unitlength,page=2]{ours_svg-tex.pdf}}%
    \put(0.64461761,0.32317659){\color[rgb]{0.1372549,0.12156863,0.1254902}\makebox(0,0)[lt]{\lineheight{1.25}\smash{\begin{tabular}[t]{l}Lewis’ instanton on $M^8$\end{tabular}}}}%
    \put(0.56757241,0.00404171){\color[rgb]{0.1372549,0.12156863,0.1254902}\makebox(0,0)[lt]{\lineheight{1.25}\smash{\begin{tabular}[t]{l}Our instanton on $M^8$\end{tabular}}}}%
    \put(0,0){\includegraphics[width=\unitlength,page=3]{ours_svg-tex.pdf}}%
  \end{picture}%
\endgroup%
    \caption{Comparison of the instanton construction by Lewis \cite{Lewi} (top right) and our instanton construction (bottom right). }
    \label{fig:ours}
\end{figure}

\begin{theorem}
Let $M_t$ be Joyce's compact $Spin(7)$-manifold described in Section \ref{subsection:the-manifolds-by-Joyce}. 
Let $G$ be a compact Lie group. 
Suppose we are given compatible gluing data in the sense of Definition \ref{definition:compatible-gluing-data} and assume that the flat connection in the compatible gluing data is unobstructed. 
Then, there exist a vector bundle $E_t$ over the manifold $M_t$ with structure group $G$ and a $Spin(7)$-instanton $\tilde{A}_t$ on $E_t$ for sufficiently small $t$. 
\label{th:intro}
\end{theorem}

We provide lots of examples (more than 20,000 four-parameter families in total) of these $Spin(7)$-instantons in the above theorem for $G= SO(3)$, $SO(4)$, $SO(5)$, $SO(7)$, $SO(8)$ in Section \ref{section:examples}.

Prior to the construction of compact $Spin(7)$-manifolds \cite{Joyc3}, Joyce performed a similar construction of examples of compact $G_2$-manifolds \cite{Joyc1, Joyc2}.  
They are resolutions of $T^7 / \Gamma$, where $T^7$ is a seven-dimensional torus, and $\Gamma$ is a finite subgroup of $\text{Aut} (T^7)$. 
Walpuski constructed $G_2$-instantons on examples of those compact $G_2$-manifolds in \cite{Walp1} (see \cite{Walp3} as well). 
The work of Walpuski solves lots of technical issues which also arise in our case, hence, we thankfully make use of his estimates in this article. 
However, there are additional difficulties present in the $Spin(7)$ case that were not present in the $G_2$ case.
One problem comes from the construction of the underlying compact $Spin(7)$-manifold:
Sobolev embeddings get worse as the dimension increases, cf. \cite[p.295, Theorem G1]{Joyc5} and \cite[p.360, Theorem S1]{Joyc5}, and because of this the obvious glued $G_2$-structure can be perturbed to a torsion-free one, but the obvious glued $Spin(7)$-structure cannot.
To overcome this, Joyce constructed a correction term for the $Spin(7)$ case, which requires careful treatment when constructing instantons on this space.

\paragraph{Spin(7)-instantons in Physics.}
\label{sec:physics}

$Spin(7)$-instanton equations appeared early on in the physics literature as an 8-dimensional generalisation of the self-duality conditions in 4 dimensions \cite{CoDeFaNu, Ward} (see also \cite{DuGuTz}). Explicit examples of $Spin(7)$-instantons over $\mathbb{R}^8$ were constructed in \cite{FaNu, FuNi, CoGoKe, IvPo} (see \cite{LoMa} for a review). When super Yang--Mills theory is reduced on a $Spin(7)$-manifold, it produces a cohomological Yang--Mills theory \cite{AcFiSpOl} whose observables are topological invariants of the moduli space of $Spin(7)$-instantons \cite{BaKaSi, AcOlSp}, in the spirit of \cite{Witt}.

In string theory, $Spin(7)$-instantons have appeared in both type II and heterotic theories. 
In type II, one can compactify on a $Spin(7)$-manifold \cite{ShVa, Acha2, FiGa}, and wrap $n$ euclidean $D$-branes \cite{BeBeMoOoOzYi} around it. The resulting worldvolume theory is $SU(n)$ super Yang--Mills (the gauge groups $SO(n)$ and $Sp(n)$ appear if orientifold planes are included) \cite{Polc}, and $Spin(7)$-instantons describe particular background configurations in this setting \cite{MaMiMoSt}. 
On the other hand, heterotic supergravity includes an $\mathcal{N}=1$ super Yang--Mills multiplet and the background gauge fields must satisfy an instanton condition (see \cite{CaHoStWi,Stro, Hull} and the review \cite{Garc}). Compactifications on 8-dimensional manifolds require a conformally balanced $Spin(7)$-structure \cite{sIvan1, GaMaWaKi}---or holonomy $Spin(7)$ when the NS gauge field is zero---and both the gauge connection \cite{GaMaWa} as well as an additional tangent bundle connection \cite{sIvan2} must be $Spin(7)$-instantons. 
Such backgrounds have been obtained in \cite{HaSt, tIvan, AcOl, IvIv, HaNo}, yet compact solutions are extremely scarce \cite{FeIvUgVi} and it would be interesting to use the instantons we obtain in this paper to construct examples of this type.

\paragraph{Structure of the paper.} 
Section \ref{section:joyces-construction-of-spin7-mf} is a review of Joyce's construction of compact $Spin(7)$-manifolds.  After recalling the definition of and some facts on $Spin(7)$-manifolds, we describe Joyce's first construction of compact $Spin(7)$-manifolds, especially focusing on one example of them, on which we construct $Spin(7)$-instantons. 

Section \ref{sec:spin7instantons} addresses basic facts on $Spin(7)$-instantons.

We start our construction in Section \ref{sec:6}. We introduce the relevant weighted H\"{o}lder norm and prove an improved estimate for the torsion-free $Spin(7)$-structure, which we need in the later sections. We then give the definition of the compatible gluing data, and construct our approximate solution in Section \ref{subsection:approximate-solutions}. We prove an estimate for the approximate solution in Section \ref{subsection:pregluing-estimate}. 

Section \ref{sec:linear} studies the linearised operators of the problem on $\R^4 \times X$ and $X \times X$, where $X$ is the Eguchi--Hanson space. These give local models for our gluing construction. 
In Section \ref{section:construction}, we prove our analytic gluing theorem of $Spin(7)$-instantons, using the estimates in Sections \ref{sec:6} and \ref{sec:linear}. 
We give infinitely many examples of our construction for the structure groups $SO(3), SO(4), SO(5)$, $SO(7)$, $SO(8)$ by producing suitable flat connections on the orbifold $T^8 / \Gamma$ in Section \ref{section:examples}.

\vspace{0.3cm}
\noindent{\it Acknowledgements.}  
This project was started when the authors attended the British Isles Graduate Workshop III in Jersey in 2019, and we are very grateful to the organisers and participants for providing a stimulative research environment. We thank Panagiotis Angelinos, Shih-Kai Chiu, Yoshihiro Fukumoto, Xenia de la Ossa, Lorenzo Foscolo, Dominic Joyce, Fabian Lehmann, Jason Lotay, Holly Mandel, Simon Salamon, Jesus Sanchez, and Thomas Walpuski for helpful correspondence and discussions. 

M.G. was supported by a scholarship from the Mathematical Institute, University of Oxford, as well as the Deutsche Forschungsgemeinschaft (DFG, German Research Foundation) under Germany’s Excellence Strategy – EXC 2121 ``Quantum Universe'' – 390833306.  D.P. was supported by the Eric and Wendy Schmidt AI in Science Postdoctoral Fellowship, a Schmidt Futures program. D.P. and Y.T. were partially supported by the Simons Collaboration on Special Holonomy in Geometry, Analysis and Physics. Y.T. was partially supported by JSPS Grant-in-Aid for Scientific Research numbers 20H00114, 20K03582, 21H00973, 21K03246, 23H01073. L.W. was supported by NSF GRFP and Award No. 2303437.

\section{The first examples of compact Spin(7)-manifolds by Joyce} 
\label{section:joyces-construction-of-spin7-mf}
\label{subsection:approximate-spin7-structures}

\subsection{Spin(7)-manifolds}

We define the group $Spin(7)$ as the double cover of the group $SO(7)$, which is $21$-dimensional, compact, connected, simply-connected, and semisimple Lie group. 
However, it will be more relevant for our purposes to use an alternative definition, namely, we describe $Spin(7)$ as a subgroup of $GL(8,\R)$, motivated by the fact that the Spin representation of the group $Spin(7)$ is eight-dimensional (see Joyce \cite{Joyc5}).

Let $\R^8$ have coordinates $(x_1,\dots,x_8)$. 
Write $d x_{ijkl}=d x_i\wedge d x_j\wedge d x_l\wedge d x_k$ and define the following four-form:
\begin{align*}
\Omega_0=d x_{1234}+d x_{1256}+d x_{1278}+d x_{1357}-d x_{1368}-d x_{1458}-d x_{1467}-\\
-d x_{2358}-d x_{2367}-d x_{2457}+d x_{2468}+d x_{3456}+d x_{3478}+d x_{5678}.
\end{align*}
Then, the group $Spin(7)$ is the subgroup of $GL(8,\R)$ preserving $\Omega_0$.
We call $\Omega_0$ the {\it standard $Spin(7)$-structure} on $\R^8$. 

Note that given the Euclidean metric $g_0=d x_1^2+\cdots+d x_8^2$ and an orientation, $\Omega_0$ is a self-dual four-form, that is, $*\Omega_0=\Omega_0$.

Now let $M$ be an eight-dimensional manifold.  
We consider a pair $(\Omega , g)$ consisting of a Riemannian metric $g$ and a four-form $\Omega$ on $M$, and call it {\it $Spin(7)$-structure} if there exists an isometry between $T_pM$ and $\R^8$ identifying $\Omega|_p$ and $\Omega_0$ for all $p \in M$.
We say $\nabla \Omega$ is the {\it torsion} of the $Spin(7)$-structure, where $\nabla$ is the Levi-Civita connection of $g$, and $(\Omega , g)$ is {\it torsion-free} if $\nabla \Omega =0$.

\begin{theorem}[Section 12 in \cite{Sala} and Proposition 10.5.3 in \cite{Joyc5}]
Let $M$ be an eight-dimensional manifold with $Spin(7)$-structure $(\Omega, g)$. Then the following are equivalent:
\begin{enumerate}
\item $(\Omega,g)$ is torsion-free; 
\item $\text{Hol}(g)\subseteq Spin(7)$;  and 
\item $\d \Omega=0$ on $M$.
\end{enumerate}
\end{theorem}

\begin{definition}
A \emph{$Spin(7)$-manifold} is a triple $(M,\Omega,g)$ where $M$ is an 8-dimensional manifold and $(\Omega,g)$ is a torsion-free $Spin(7)$-structure.
\end{definition}

\subsection{Examples of compact Spin(7)-manifolds by Joyce}
\label{subsection:the-manifolds-by-Joyce}

The first construction of compact manifolds with holonomy $Spin(7)$ was obtained by Joyce \cite{Joyc3} (see also \cite[\S 13--14]{Joyc5}). 
The underlying topological space for this construction is a resolution $M$ of $T^8/\Gamma$, where $T^8$ is an eight-dimensional torus, and $\Gamma$ is a finite subgroup of $\operatorname{Aut}(T^8)$. 

In this paper, we are interested in one example of compact $Spin(7)$-manifolds in \cite{Joyc3}, which we describe in more detail later in this section. In that example, the resolution of each component of the singular set of $T^8/\Gamma$ is locally an open subset of the product of the Eguchi--Hanson space $X$ with itself, namely, $X \times X$,  or $ \R^4 \times X$.

For carefully chosen $\Gamma$, one has $\hat{A}(M) = 1$, which guarantees that the holonomy group of $M$ equals to $Spin(7)$, provided that there exists a torsion-free $Spin(7)$-structure on $M$. 
Here, $\hat{A}(M)$ is the $A$-hat genus of $M$. 
Joyce proves the existence of such a structure in the following form:  

\begin{theorem}[Theorem 13.6.1 and Proposition 13.7.1 in \cite{Joyc5}]
Let $\lambda$, $\mu$, $\nu >0$ be constants. Then there exist 
constants $\kappa, K >0$ such that for $0 < t \leq \kappa$ the 
following holds. 
Let $M$ be a compact 8-manifold, and $(\Omega_{t} , g_{t})$ a family of 
 $Spin(7)$-structures 
on $M$. 
Suppose that $\phi_{t}$ is a family of four-forms on $X$ 
with $\d \Omega_{t} + \d \phi_{t} = 0$, and 
\begin{enumerate}
\item[$(i)$] $ || \phi_{t} ||_{L^2} \leq \lambda t^{{13}/{3}}$ 
   and $|| \d \phi_{t} ||_{L^{10}} \leq \lambda t^{{7}/{5}}$;  
\item[$(ii)$] the injectivity radius 
$\delta (g_{t})$ satisfies $ \delta (g_{t}) \geq \mu t$; 
and 
\item[$(iii)$] the Riemannian curvature $R(g_{t})$ 
satisfies $|| R(g_{t}) ||_{C^{0}} \leq \nu t^{-2}. 
$\end{enumerate} 
Then there exists a smooth torsion-free $Spin(7)$-structure 
$( \tilde{\Omega}_{t} , \tilde{g}_{t})$ on $M$ such that 
$|| \tilde{\Omega}_{t} -\Omega_{t} ||_{C^{0}} \leq K t^{{1}/{3}}$ and 
$|| \nabla ( \tilde{\Omega}_{t} - \Omega_{t} ) ||_{L^{10}} \leq K t^{{2}/{15}}$.  
\label{th:Joyce_tfsp7}
\end{theorem}

Here $g_t$ is the metric induced by $\Omega_t$, that is, the metric with respect to which $Spin(7)$ is a subgroup of $SO(8)$.

To construct the data of the assumption in Theorem \ref{th:Joyce_tfsp7}, we interpolate between the torsion-free $Spin(7)$-structures on the smooth locus of $T^8/\Gamma$ and on the resolutions of the components of the singular set. In the example we describe below, the latter carries $Spin(7)$ structures because the Eguchi--Hanson space $X$ carries a Ricci-flat ALE metric, and therefore has holonomy $SU(2)$.
Hence, the local resolutions have holonomy $SU(2) \times SU(2)$,  or $\lbrace 1 \rbrace \times SU(2)$. 
Since each of these is a subgroup of $Spin(7)$, this implies that each local resolution carries a $Spin(7)$-structure. 
The $Spin(7)$-structure asymptotes to the $Spin(7)$-structure on the smooth locus of $T^8/\Gamma$ on the overlap region.

The interpolation between these two structures requires significant additional work. 
Unlike the $G_2$ case, a straightforward use of partitions of unity is not sufficient, because one would only achieve $\Vert \phi_t \Vert_{L^2} = O (t^4)$, whereas we require that $\Vert \phi_t \Vert_{L^2} = o(t^4)$. 
A more detailed construction allowed Joyce to cancel the fourth-order terms and apply Theorem \ref{th:Joyce_tfsp7}. 

\paragraph{An example of the construction by Joyce.}
Here, we describe in detail one type of compact $Spin(7)$-manifolds constructed in \cite{Joyc3}. 
Similar arguments for other choices of $\Gamma$ lead to topologically distinct $Spin(7)$-manifolds. 

Consider $T^8 := (\mathbb{R}/\mathbb{Z})^8$ with coordinates $(x_1,...,x_8)$, and carrying the flat $Spin(7)$-structure $\Omega_0$.
The group $\Gamma$, generated by the following involutions, acts on $T^8$:
\begin{align}
\label{equation:alpha-beta-gamma-delta-def}
 \begin{split}
 \alpha : 
 (x_1, \dots, x_8) 
 & \mapsto 
 (-x_1, -x_2, -x_3, -x_4, 
 x_5, x_6, x_7, x_8);
 \\
 \beta : 
 (x_1, \dots, x_8) 
 & \mapsto 
 (x_1, x_2, x_3, x_4, 
 -x_5, -x_6, -x_7, -x_8); 
 \\
 \gamma :
 (x_1, \dots, x_8) 
 & \mapsto 
 \left( \frac{1}{2}-x_1, \frac{1}{2}-x_2, x_3, x_4, 
 \frac{1}{2}-x_5, \frac{1}{2}-x_6, x_7, x_8 \right); \text{ and}
 \\
 \delta :
 (x_1, \dots, x_8) 
 & \mapsto 
 \left( -x_1, x_2, \frac{1}{2}-x_3, x_4, 
 \frac{1}{2}-x_5, x_6, \frac{1}{2}-x_7, x_8 \right).
 \end{split}
\end{align}
Note that $\Gamma$ preserves $\Omega_0$ and $\Gamma \cong \mathbb{Z}_2^4$. 

To understand how $T^8/\Gamma$ is resolved, we first determine the fixed locus of each element of $\Gamma$. 
Then we investigate how $\Gamma$ acts on a neighborhood of the fixed locus. 
%Different choices of $\Gamma$ lead to different singular loci in the quotient and require topologically distinct resolutions. 
Our choice of $\Gamma$ is made to achieve $\hat{A}(M) = 1$ so that the holonomy group of the resolved manifold $M$ coincides with the group $Spin(7)$.

\begin{itemize}
\item 
The fixed locus of $\alpha$ is equal to $\{p_i\} \times T^4$ for 16 $p_i \in {T}^4$.
For each $p_i$, the action of $\alpha$ on the first factor is the quotient of $T^4$ by $\mathbb{Z}_2$, while in a neighborhood of $p_i$ on the second factor the action of $\beta$ is modelled by the action of $\mathbb{Z}_2$ on $\mathbb{R}^4$ on the second component. 
On the other hand, $\gamma$ and $\delta$ act freely on the fixed set of $\alpha$, dividing its 16 components into $4$ sets of $4$.
The resulting singularities in $T^8/\Gamma$, are called \emph{singularities of type (ii)} in \cite{Joyc3}, and we can resolve each component of the fixed locus of $\alpha$ by ${T}^4/\{\pm 1\}  \times U$, where $U$ is an open subset of the Eguchi--Hanson space. 
We leave the first factor unresolved because the singular point is included in the singular set of $\alpha \beta$, so a different resolution will be glued to a neighborhood of this point.

\item 
The discussion for the fixed locus of $\beta$ is identical to the case for $\alpha$ described above, up to the order of the factors and the resulting singularities in $T^8/\Gamma$. 

\item 
The fixed locus of $\alpha \beta$ is 256 points. 
Also, $\gamma$ and $\delta$ act freely on these points, identifying them into 64 sets of 4. 
Locally, the singularities at these points are modelled by neighborhoods of $0$ in $\mathbb{R}^4/\{\pm 1\} \times \mathbb{R}^4/\{\pm 1\}$, are called \emph{singularities of type (iii)} in \cite{Joyc3} and are resolved by $U \times U$.

\item 
The fixed locus of $\gamma$ is 16 copies of ${T}^4$, and $\alpha$, $\beta$, and $\delta$ act freely on these components, identifying them into 2 sets of 8.
These singularities are called \emph{singularities of type (i)} in \cite{Joyc3}.
A neighborhood of each component is isomorphic to $T^4 \times B^4/ \{\pm 1\} $, so the resolution is ${T}^4 \times U$. 

\item 
The discussion for $\delta$ is identical to the discussion for $\gamma$.  
\end{itemize}

We label the fixed loci of elements of $\Gamma$ by $S_j$ for $j \in \{1, \dots, 76 \}$ as displayed in \cref{table:singular-set-in-t8} and also write $S=\cup_{j \in \{1,\dots,76\}} S_j$ for the singular set of $T^8/\Gamma$.
For some small $\zeta>0$ ($\zeta=1/9$ will do) we have that each of the $S_j$ then has a neighbourhood isomorphic to $T^4/\{\pm 1\} \times B^4_{\zeta}/ \{ \pm 1 \}$, $B^4_{\zeta}/ \{ \pm 1 \} \times B^4_{\zeta}/ \{ \pm 1 \}$, or $T^4 \times B^4_{\zeta}/ \{ \pm 1 \}$ considered as Spin(7)-orbifolds, and those neighbourhoods will be denoted by $T_j$ for $j \in \{1, \dots, 76 \}$.
Note that $T:=\bigcup_{j \in \{1, \dots, 76\}} T_j$ consists of five connected components:
\begin{align}
\label{equation:definition-C1-C2-C3-C4-C5}
 C_1:=T_{73},C_2:=T_{74},C_3:=T_{75},C_4:=T_{76}, \text{ and } C_5:=T_1 \cup \dots \cup T_8, 
\end{align}
where $C_5$ also contains $T_j$ for $j \in \{9, \dots, 72\}$. 

\begin{table}[h]
\begin{center}
\begin{tabular}{@{}cccc@{}}
\toprule
                               & Type  & Neighbourhood                                & Label                   \\ \midrule
$\fix(\alpha)$, $\fix(\beta)$  & (ii)  & $T^4/\{\pm 1\} \times B^4_{\zeta}/ \{ \pm 1 \}$     & $S_1, \dots, S_8$       \\
$\fix(\alpha \beta)$           & (iii) & $B^4_{\zeta}/ \{ \pm 1 \} \times B^4_{\zeta}/ \{ \pm 1 \}$ & $S_9, \dots, S_{72}$    \\
$\fix(\gamma)$, $\fix(\delta)$ & (i)   & $T^4 \times B^4_{\zeta}/ \{ \pm 1 \}$               & $S_{73}, \dots, S_{76}$ \\ \bottomrule
\end{tabular}
\caption{Labels for the singular loci in $T^8/\Gamma$.}
\label{table:singular-set-in-t8}
\end{center}
\end{table}

Denote by $\rho: X \rightarrow \C^2/\{ \pm 1\}$ Eguchi--Hanson space from \cref{proposition:eguchi-hanson}.
For each $t \in (0,1)$ and $j \in \{1, \dots, 76\}$ we define a resolution $\pi_{j,t}:\tilde{T}_{j,t} \rightarrow T_j$ as follows:
\begin{itemize}
\item
For $j \in \{1, \dots, 8\}$; $\pi_{j,t}:\tilde{T}_{j,t} = T^4/\{\pm 1\} \times \rho^{-1}(B^4_{t^{-1}\zeta}/ \{ \pm 1 \}) \rightarrow T_j$ is given by $\pi_{j,t}(x,y)=(x, t \cdot \rho(y))$. 

\item
For $j \in \{9, \dots, 72\}$; $\pi_{j,t}:\tilde{T}_{j,t} = \rho^{-1}(B^4_{t^{-1}\zeta}/ \{ \pm 1 \}) \times \rho^{-1}(B^4_{t^{-1}\zeta}/ \{ \pm 1 \}) \rightarrow T_j$ is given by $\pi_{j,t}(x,y)=(t \cdot \rho(x), t \cdot \rho(y))$. 

\item
For $j \in \{73, \dots, 76\}$; $\pi_{j,t}:\tilde{T}_{j,t} = T^4 \times \rho^{-1}(B^4_{t^{-1}\zeta}/ \{ \pm 1 \}) \rightarrow T_j$ is given by $\pi_{j,t}(x,y)=(x, t \cdot \rho(y))$.
\end{itemize}

For later use we also define the \emph{regular part of $\tilde{T}_{j,t}$}:
\begin{itemize}
\item
For $j \in \{1, \dots, 8\}$; 
$\tilde{T}_{j,t}^{\text{reg}}
:=
(T^4/\{\pm 1\} \setminus \fix(-1)) \times \rho^{-1}(B^4_{t^{-1}\zeta}/ \{ \pm 1 \})$,

\item
For $j \in \{9, \dots, 76\}$; 
$\tilde{T}_{j,t}^{\text{reg}} := \tilde{T}_{j,t}$.
\end{itemize}

By appropriately identifying points in the resolution of the overlapping areas, these spaces glue to a resolution $\pi: M \rightarrow {T}^8/\Gamma$. 
The usefulness of this construction is that all resolutions are modelled on products of the Eguchi--Hanson space with itself or flat space. 
The explicit formula for the Eguchi--Hanson metric allows us to describe the geometry as we rescale the metric on the resolutions by $t^2$.

%On an eight-manifold carrying a $Spin(7)$-structure, the $\hat{A}$ genus is computed as a linear combination of the Betti numbers (see \cite[\S 10.6]{Joyc5} and \cite[Theorem 7.3]{Sala2}). 
%Therefore, one can check $\hat{A}(M) = 1$ by computing the Betti numbers of each of the spaces comprising the resolutions described above and then use Mayer--Vietoris arguments to compute the Betti numbers of $M$ (see \cite[\S 14.2]{Joyc5}). 

The above eight-manifold $M$ carries a $Spin(7)$-structure obtained as follows:
the sets $\tilde{T}_{j,t}$ carry product $Spin(7)$-structures $\Omega_{j,t}$.
In \cite[p. 535]{Joyc3}, the smooth $Spin(7)$-structure $\Omega_t \in \Omega^4(M)$ is defined, which satisfies the following estimate:

\begin{proposition}
\label{proposition:spin-7-difference-to-product-structure}
 There exists a constant $c>0$ such that the following is true:
 on $M \setminus \pi^{-1} \left( \bigcup_{j=1}^{76} T_j \right)$ we have
 \[
  \|{ \Omega_t-\Omega_0 }_{C^{0,\alpha}}
  \leq
  ct^{7/2}.
 \]
 For $j \in \{1, \dots, 76\}$ we have on $\pi^{-1}(T_j)$ the estimate
 \[
  \|{ \Omega_t-\Omega_j }_{C^{0,\alpha}}
  \leq
  ct^{7/2}.
 \]
\end{proposition}

\begin{proof}
    We write out the proof for the first estimate, the proof for the second is analogous.
    In \cite[Equation 34]{Joyc3}, the four-form $X_t$ is defined and $\Omega_t$ is a projection of $X_t$ onto the bundle of forms with stabiliser $Spin(7)$.
    In the proof of \cite[Proposition 4.2.3]{Joyc3} we have the estimates $|\nabla(X_t-\Omega_0)|=O(t^{7/2})$ and $|\nabla(X_t-\Omega_t)|=O(t^{7/2})$, which implies $|\nabla(\Omega_t-\Omega_0)|=O(t^{7/2})$.
    From \cite[Proposition 4.2.2]{Joyc3} we have $|\Omega_t-\Omega_0|=O(t^4)$, which proves the claim.
\end{proof}
We denote the resulting manifold by $M_{t}$, when we emphasize the $t$-dependency.  Otherwise, it is denoted by $M$.

\section{Spin(7)-instantons on Spin(7)-manifolds}

\subsection{Spin(7)-instantons}
\label{sec:spin7instantons}

Let $M$ be a compact $Spin(7)$-manifold. 
Then the bundle $\Lambda^2: = \Lambda^2 T^* M$ on  a $Spin(7)$-manifold orthogonally  splits as:
\begin{equation}
\Lambda^2=\Lambda^2_{21}\oplus\Lambda^2_7,
 \label{equation:2forms-orthogonal-splitting}
\end{equation}
where $\Lambda^2_{21}$ is a vector bundle of rank 21 corresponding to $\mathfrak{spin}(7) \cong \mathfrak{so}(7)\subset\mathfrak{so}(8)$ and $\Lambda^2_7$ is a vector bundle of rank 7 orthogonal to $\Lambda^2_{21}$  
(see e.g. \cite[\S 9]{SaWa}, \cite[\S 3.4]{Reye} for more details). 
We denote by $\pi^2_7$ the projection $\Lambda^2 \to \Lambda^2_7$.

\begin{remark}
There is an alternative way of finding the decomposition \eqref{equation:2forms-orthogonal-splitting}. 
Using the $Spin(7)$-structure, we define the operator 
$ S :\Lambda^2 \longrightarrow \Lambda^2$ by  
$$
\alpha \longmapsto *(\Omega\wedge\alpha).
$$ 
It turns out that $S$ preserves the splitting and the subspaces are eigenspaces of the operator. We find that $\Lambda^2_{21}$ has eigenvalue $-1$; and $\Lambda^2_7$ has eigenvalue $3$. Namely, the orthogonal projections onto $\Lambda^2_7$ and $\Lambda^2_{21}$ are given by
 \begin{align}
 \label{eq:spin7-projections}
 \begin{split}
  \pi^2_{21}: \Lambda^2 & \rightarrow \Lambda^2_{21}
  \\
  \omega & \mapsto \frac{1}{4} (* (\Omega \wedge \omega)-3\omega), \\
  \pi^2_7: \Lambda^2 & \rightarrow \Lambda^2_7
  \\
  \omega & \mapsto \frac{1}{4} (* (\Omega \wedge \omega)+\omega).
  \\
 \end{split}
 \end{align}
 See e.g. again  \cite[\S 9]{SaWa} for more details. 
\end{remark}

\begin{definition}
 A connection $A$ of a principal $G$-bundle $P$ over $M$ is said to be a \emph{$Spin(7)$-instanton} if its curvature $F_A$ satisfies
 \begin{align}
 \label{equation:instanton-equation}
  \pi^2_7(F_A) = 0.
 \end{align}
\end{definition}

\paragraph{Examples of Spin(7)-instantons.} 

Examples of $Spin(7)$-instantons were first constructed in Physics such as by Fairlie--Nuyts \cite{FaNu}, Fubini--Nicolai \cite{FuNi} on $\mathbb{R}^8$.  
Then Kanno--Yasui \cite{KaYa} constructed examples of them on Bryant--Salamon's $Spin(7)$-manifolds with structure group $Spin(7)$.  
Later, Clarke \cite{Clar} (see also Clarke--Oliveira \cite{ClOl}) constructed them on Bryant--Salamon's manifolds with structure group $SU(2)$.  

As mentioned in the Introduction, the first examples of $Spin(7)$-instantons on compact $Spin(7)$-manifolds were given by  
Lewis \cite{Lewi}. 
Subsequently, 
Tanaka \cite{Tana} constructed examples of $Spin(7)$-instantons on Joyce's second examples of compact $Spin(7)$-manifolds. Walpuski \cite{Walp2} worked out a Taubes' construction for $Spin(7)$-instantons, as mentioned before.

\subsection{Local deformation}

As in the case of anti-self-dual instantons in four dimensions, the action of the gauge group $\mathscr{G}(P)$ preserves the decomposition of the curvature $F_A$ of a connection $A$, 
that is, if $A$ is a $Spin(7)$-instanton, then so is $g (A)$ for all $g \in \mathscr{G}(P)$. 
Therefore, we consider\footnote{with an appropriate topology.}  
$$\mathscr{M}(P, \Omega)
  :=
  \{
   A \in \mathscr{A}(P)
   :
   \pi^2_7(F_A)=0
  \}
  /\mathscr{G}(P), 
  $$ and call it the {\it moduli space of $Spin(7)$-instantons on $P$}.

The local description around $A \in \mathscr{M}(P, \Omega)$ is described by means of the linearisation of the $Spin(7)$-instanton equation \eqref{equation:instanton-equation}, namely, the linear operator 
$\d^7_A := \pi^2_7 \circ \d_A$, with customarily the {\it Coulomb gauge condition $\d^*_A a =0$} (cf. \cite[\S 2.3.1]{DoKr}). 
In other words, the following  elliptic complex governs the infinitesimal deformation of a $Spin(7)$-instanton $A$: 
\begin{equation}
0 \longrightarrow 
\Gamma ( \Lambda^0 \otimes \Ad P) 
\xrightarrow{\d_A} \Gamma ( \Lambda^1 \otimes \Ad P) \xrightarrow{ \d_A^7} 
\Gamma (\Lambda^2_7 \otimes \Ad P ) \longrightarrow 0, 
\label{eq:complex}
\end{equation}
or, equivalently, so does the following linear elliptic operator: 
\begin{align} 
\label{equation:linearised-operator}
L_A= \left(\d_A^*, \d^7_A \right): \Gamma ( \Lambda^1 \otimes \Ad P ) 
\longrightarrow 
\Gamma ( \Lambda^0 \otimes \Ad P) \oplus \Gamma (\Lambda^2_7 \otimes \Ad P). 
\end{align}
We denote by $\mathcal{H}^i_A$ the $i$-th cohomology of the complex \eqref{eq:complex} for $i=0,1,2$.  

\begin{remark}
The ellipticity of the $Spin(7)$-instanton equations was first addressed in Bilge \cite{Bilg} and the existence of an elliptic complex was described in Reyes Carri\'on \cite{Reye}. 
From the viewpoint of physics, the (infinitesimal) moduli space of heterotic backgrounds is given by the cohomology of a certain bundle-valued complex over the compactifying manifold. The instanton deformation complex of Reyes Carri{\'o}n \cite{Reye} appears naturally in this description, encoding the infinitesimal moduli of the gauge fields. See \cite{AnGrLuOv, GaRuTi} for the $SU(3)$ case and \cite{ClGaTi, delaOsLaSv} for the $G_2$ case.
\end{remark} 
\kuro

\begin{definition}
\label{definition:rigid-irreducible-unobstructed}
Let $A$ be a $Spin(7)$-instanton on $P$. 
We say $A$ is \emph{irreducible} if $\mathcal{H}^0_A=0$, it is \emph{infinitesimally rigid} if $\mathcal{H}^1_A=0$, and it is \emph{unobstructed} if $\mathcal{H}^2_A=0$.
\end{definition}

When the underlying manifold is a Calabi--Yau four-fold, $Spin(7)$-instantons on a Hermitian vector bundle on it become Hermitian-Yang--Mills connections (see \cite{Lewi}).  
Therefore, the moduli space of $Spin(7)$-instantons can be identified with that of semistable bundles via the Hitchin--Kobayashi correspondence. 
Then one may consider an analogue of the theory of Donaldson invariants by means of the moduli spaces of semistable sheaves on a Calabi--Yau four-fold with the help of virtual techniques. As mentioned in Introduction, this direction of research has been pursued by Borisov--Joyce \cite{BoJo}, Oh--Thomas \cite{OhTh}, \cite{OhTh2}, and other people.

\section{Approximate solutions and the estimate}
\label{sec:6}

The smooth manifold $M_t$ carries the $Spin(7)$-structure with small torsion $\Omega_t$ from Section \ref{subsection:the-manifolds-by-Joyce}.
In \cite{Joyc3} the torsion-free $Spin(7)$-structure $\tilde{\Omega}_t$ on $M_t$ was constructed.
The difference $\tilde{\Omega}_t-\Omega_t$ is small and in Section \ref{subsection:improvement-of-spin7-existence-result} we will sharpen an estimate from the literature and quantify this smallness explicitly.

The $Spin(7)$-structure $\Omega_t$ was obtained by gluing together different local $Spin(7)$-structures.
For each local $Spin(7)$-structure there exist certain obvious $Spin(7)$-instantons.
In Section \ref{subsection:approximate-solutions}, we will glue these obvious $Spin(7)$-instantons together to a connection $A_t$ that is globally defined on all of $M_t$.

Because of the gluing process, $A_t$ is not exactly a $Spin(7)$-instanton---neither with respect to $\Omega_t$ nor with respect to $\tilde{\Omega}_t$.
However, it is close to being a $Spin(7)$-instanton.
In Section \ref{subsection:pregluing-estimate} we estimate the error of $A_t$ of being a $Spin(7)$-instanton with respect to $\tilde{\Omega}_t$, making use of the estimate from Section \ref{subsection:improvement-of-spin7-existence-result}.

In Section \ref{subsection:perturbation-step}, the connection $A_t$ will be perturbed to an exact solution of the $Spin(7)$-instanton equation.

\begin{convention}
 In what follows, we will prove estimates on the one-parameter family of manifolds $(M, \tilde{\Omega}_t)$ and it will be important to understand exactly how these estimates depend on $t$.
 Throughout, we will use $c$ to denote a constant that can be different from line to line, but is always independent of $t$.
\end{convention}

In Section \ref{subsection:the-manifolds-by-Joyce} (cf. Table \ref{table:singular-set-in-t8}) we introduced the notation $S_1, \dots, S_8, S_{73},\dots,S_{76}$ for singularities looking locally like $\R^4 \times (\R^4 /\{\pm 1\})$, and we used $S_9,\dots,S_{72}$ for singularities looking locally like $(\R^4 /\{\pm 1\}) \times (\R^4 /\{\pm 1\})$.
In the following norm, we will use the minimum distance to \emph{any} of these singular sets:

\begin{definition}[Weighted H\"{o}lder norm]
Set $\displaystyle{S:= \bigcup_{j=1}^{76} S_j}$. \label{definition:weighted-hoelder-norm-on-Mt}
 On $M_t$ we define $r_t : M_t \rightarrow [0, \zeta]$ by $r_t (x) := 
  \min \{ \zeta, d_{T^8/\Gamma} (\pi(x), S) \}$ for $x \in M_t$
 and the weight functions by 
 \[
  w_t(x)
  =
  t+ r_t(x)
  \text{ and }
  w_t(x,y)
  =
  \min
  \{
   w_t(x), w_t(y)
  \}.
 \]
 For a H\"{o}lder exponent $\alpha \in (0,1)$ and weight parameter $\beta \in \R$ define
 \begin{align*}
  [f]
  _{C^{0,\alpha}_{\beta,t}(U)}
  &=
  \sup_{d(x,y) \leq w_t(x,y)}
  w_t(x,y)^{\alpha-\beta}
  \frac{|f(x)-f(y)|}{d(x,y)^\alpha},
  \\
  \|{f}_{L^\infty_{\beta,t}(U)}
  &=
  \|{w_t^{-\beta} f}_{L^\infty(U)},
  \text{ and}
  \\
  \|{f}_{C^{k,\alpha}_{\beta,t}(U)}
  &=
  \sum_{j=0}^k
  \|{\nabla^j f}_{L^\infty_{\beta-j}(U)}
  +
  \left[
   \nabla^k f
  \right]_{C^{0,\alpha}_{\beta-j}(U)}.
 \end{align*}
 Here, $f$ is a section of a vector bundle over $U \subset M_t$ equipped with a norm $| \cdot |$ and a connection $\nabla$.
 We use parallel transport to compare values of $f$ at different points.
 If $U$ is not specified, then we take $U = M_t$.
\end{definition}

\subsection{Improved estimate for torsion-free Spin(7)-structures}
\label{subsection:improvement-of-spin7-existence-result}

In \cite{Joyc3} the torsion-free $Spin(7)$-structure $\tilde{\Omega}_t$ on $M_t$ was constructed, and an estimate for $\tilde{\Omega}_t-\Omega_t$ in the $C^0$-norm was given.
For later analysis, we require an estimate in the weighted H\"{o}lder norms from \cref{definition:weighted-hoelder-norm-on-Mt}, and it is the purpose of this section to prove this estimate.

We begin with the following technical lemma which is an adaptation from Proposition 2.25 in \cite{Walp3}. 

\begin{proposition}
\label{proposition:near-euclidean-coordinates}
 For each $\mu > 0$ and $K \in \mathbb{N}_0$ there exists a constant $\epsilon > 0$ such that the following holds for all $t \in (0,T)$ and $p \in M_t$:
 the number $R:= \epsilon(t+r_t(p))$ is less than the injectivity radius of $M_t$ at $p$ and if we identify $T_p M_t$ isometrically with $\R^8$ and denote by $s_R: B_1 \rightarrow B_R(p)$ the map obtained by multiplication with $R$ followed by the exponential map, then
 \begin{align}
  \left|
   \partial^k(R^{-2} s^*_R g_t-g_{\R^8})
  \right|
  \leq
  \mu
 \end{align}
 for all $k \in \{0, \dots, K\}$.
 Here $g_{\R^8}$ denotes the standard metric on $\R^8$.
\end{proposition}

This allows to transport a standard elliptic regularity theorem from $\R^8$ to small neighbourhoods of $M_t$ as in Equation (2.29) in \cite{Walp3}.
For a Riemannian $8$-manifold $N$ denote by $\Lambda^4_-(N)$ the bundle of anti-self-dual $4$-forms, denote the smooth anti-self-dual $4$-forms (i.e. smooth sections of $\Lambda^4_-(N)$) by $\Omega^4_-(N)$, and denote $C^{k,\alpha}$-sections of $\Lambda^4_-(N)$ by $C^{k,\alpha}(\Lambda^4_-(N))$.

\begin{lemma}
\label{lemma:holder-seminorm-elliptic-regularity}
 For $\alpha \in (0,1)$ small enough, there exist $\epsilon, c>0$ independent of $t$ such that for all $p \in M_t$ and $R$ as in \cref{proposition:near-euclidean-coordinates} we have
 \begin{align}
  R^\alpha
  [\sigma]_{C^{0,\alpha}(B_{\frac{1}{2}R})}
  \leq
  c
  \left(
   R^{1/5}
   \|{ \d \sigma }_{L^{10}(B_R)}
   +
   \|{ \sigma }_{L^\infty(B_R)}
  \right)
 \end{align}
 for any $\sigma \in \Omega^4_-(B_1)$.
\end{lemma}

\begin{proof}
 The operator $\d+\d^*$ is an elliptic operator acting on $\Omega^4_-(\R^8)$.
 Because $|\d^* \sigma|=|\d \sigma|$ for $\sigma \in \Omega^4_-(\R^8)$, we get from standard elliptic theory together with the Sobolev embedding $L^{10}_1 \hookrightarrow C^{0,\alpha}$ in dimension $8$ for small $\alpha$ the following estimate:
 \begin{align}
  [\sigma]_{C^{0,\alpha}(B_{1/2})}
  \leq
  c
  \left(
   \|{ \d \sigma }_{L^{10}(B_1)}
   +
   \|{ \sigma }_{L^\infty(B_1)}
  \right).
 \end{align}
 Applying the rescaling map $s_R: B_1 \rightarrow B_R(p)$ and using \cref{proposition:near-euclidean-coordinates} gives the claim.
\end{proof}

We are now ready to prove the main result of this section, namely the estimate for the difference $\tilde{\Omega}_t-\Omega_t$ in the weighted H\"{o}lder norm:

\begin{theorem}
\label{proposition:spin-7-structure-holder-estimate}
 There exists $c>0$ independent of $t$ such that for small $t$ there exists a torsion-free $Spin(7)$-structure $\tilde{\Omega}_t$ satisfying
 \begin{align}
 \label{equation:torsion-free-spin-7-structure-improved-estimate}
  \|{\tilde{\Omega}_t-\Omega_t}_{C^{0,\alpha}_{0;t}}
  \leq
  ct^{3/10}.
 \end{align}
\end{theorem}

The two theorems \cite[Theorem A and Theorem 4.2.4]{Joyc3} give \cref{proposition:spin-7-structure-holder-estimate} with $\|{\tilde{\Omega}_t-\Omega_t}_{C^0} \leq ct^{1/2}$, and we now claim a $C^{0,\alpha}_{0;t}$-estimate.
We will give a proof of this improved theorem in the following.

\begin{proof}[Proof of \cref{proposition:spin-7-structure-holder-estimate}]
 In the proof of \cite[Theorem A]{Joyc3}, an element $\eta_t \in \Omega^4_-(M_t)$ solving the equation $\d \eta_t=\d \phi_t +\d F(\eta_t)$ and obeying the estimates
 \begin{align}
 \label{equation:torsion-free-spin-7-construction-original-estimates}
 \|{ \eta_t }_{L^\infty} \leq ct^{1/2}
 \text{ and }
 \|{ \nabla \eta_t }_{L^{10}} \leq ct^{3/10}
 \end{align} 
 was constructed.
 From this, the torsion-free $Spin(7)$-structure 
 \begin{align}
 \label{equation:omega-tilde-eta-equation}
 \tilde{\Omega}_t:=\Omega_t+\eta_t-\Theta(\eta_t)
 \end{align}
 was defined, where $\Theta$ is the smooth function from \cite[Lemma 5.1.1]{Joyc3}.

 We first prove
 \begin{align}
     \label{equation:eta-c0alpha-estimate}
     \|{\eta}_{C^{0,\alpha}_{0;t}}
     \leq
     c t^{3/10}.
 \end{align}
 The estimate $\|{\eta}_{C^0} \leq c t^{1/2}$ is contained in \eqref{equation:torsion-free-spin-7-construction-original-estimates},
 thus, it remains to show $[\eta]_{C^{0,\alpha}_{0;t}} \leq ct^{3/10}$.
 Let $\epsilon>0$ be the constant from \cref{lemma:holder-seminorm-elliptic-regularity}.
 For $x,y \in M_t$ with $d(x,y) \leq \epsilon w_t(x,y)$ we have
 \begin{align}
 \label{equation:eta-holder-estimate}
 \begin{split}
     w_t^\alpha(x,y)
     \frac{|\eta(x)-\eta(y)|}{d(x,y)^\alpha}
     &
     \leq
     cR^\alpha 
     \frac{|\eta(x)-\eta(y)|}{d(x,y)^\alpha}
     \\
     &
     \leq
     c
     \left(
     R^{1/5}
     \|{ \d \eta }_{L^{10}(B_R)}
     +
     \|{ \eta }_{L^\infty(B_R)}
     \right)
     \\
     &
     \leq
     c
     \left( t^{3/10}+t^{1/2} \right),
 \end{split}
 \end{align}
 where we used $R=\epsilon w_t(x,y)$ in the first step,
 we used \cref{lemma:holder-seminorm-elliptic-regularity} in the second step (for this we used the assumption $d(x,y) \leq \epsilon w_t(x,y)$),
 and we used $R \leq c$ together with the bounds in \eqref{equation:torsion-free-spin-7-construction-original-estimates} in the last step.
 
 The estimate \eqref{equation:eta-holder-estimate} for $x,y \in M_t$ with $d(x,y) \leq w_t(x,y)$ rather than $d(x,y) \leq \epsilon w_t(x,y)$ is proved as follows.
 We claim that there is $L>0$ depending on $\epsilon$ but not on $t$ such that for any $x,y \in M_t$ with $d(x,y) \leq w_t(x,y)$ we have points $x_0=x, x_1, \dots, x_{L-1}, x_L=y \in M_t$ such that $d(x_{i-1},x_i)\leq \epsilon w_t(x_{i-1},x_i)$ for all $i \in \{1, \dots, L\}$.
 This sequence is constructed by taking $L$ equally spaced points that are no closer to $\pi^{-1}(S_1 \cup \dots \cup S_{76})$ that $x$ or $y$.
 It satisfies $w_{t_i}(x_i) \geq w_t(x,y)$ and therefore
  \begin{align}
 \label{equation:eta-holder-estimate-telescoped}
     w_t^\alpha(x,y)
     \frac{|\eta(x)-\eta(y)|}{d(x,y)^\alpha}
     &
     \leq
     \sum_{i=1}^L
     w_t^\alpha(x_{i-1},x_i)
     \frac{|\eta(x_{i-1})-\eta(x_i)|}{d(x_{i-1},x_i)^\alpha}.
 \end{align}
 For large enough $L$, independent of $t$, we have that $d(x_{i-1},x_i) \leq \epsilon d(x,y) \leq \epsilon w_t(x,y)$, so we can plug \cref{equation:eta-holder-estimate} into \cref{equation:eta-holder-estimate-telescoped} and obtain
 \[
     w_t^\alpha(x,y)
     \frac{|\eta(x)-\eta(y)|}{d(x,y)^\alpha}
     \leq
     c (t^{3/10}+t^{1/2}), 
 \]
 this time for $d(x,y) \leq w_t(x,y)$.

 follows because one can choose at most $1/\epsilon$ points, which is a fixed number of points, between $x$ and $y$ with pairwise distance $\epsilon w_t(x,y)$ or less and apply \eqref{equation:eta-holder-estimate}.
 This shows the bound \eqref{equation:eta-c0alpha-estimate}.

 Next, we check that this implies the estimate \eqref{equation:torsion-free-spin-7-structure-improved-estimate}.
 We have
 \begin{align}
    \label{equation:omega-tilde-minus-omega-triangle-inequality-result}
     \begin{split}
         \|{\tilde{\Omega}_t-\Omega_t}_{C^{0,\alpha}_{0;t}}
         &\leq
         \|{\eta_t}_{C^{0,\alpha}_{0;t}} + \|{\Theta(\eta_t)}_{C^{0,\alpha}_{0;t}}
     \end{split}
 \end{align}
 by \eqref{equation:omega-tilde-eta-equation}.
 We estimate the last summand as follows:
 \begin{align}
    \label{equation:theta(eta)-estimate}
     \begin{split}
         \|{\Theta(\eta_t)}_{C^{0,\alpha}_{0;t}}
         &\leq
         \|{\Theta(\eta_t)}_{C^0}
         +
         \sup _{d(x,y) \leq w_t(x,y)} w_t(x,y)^\alpha
         \frac{|\Theta(\eta_t)(x)-\Theta(\eta_t)(y)|}{d(x,y)^\alpha}
         \\
         &\leq
         c (t^{3/10})^2
         \\
         &\quad \quad
         +
         c
         \underbrace{
         \sup _{d(x,y) \leq w_t(x,y)} w_t(x,y)^\alpha
         \frac{|\eta_t(x)-\eta_t(y)|}{d(x,y)^\alpha}
         }_{=[\eta]_{C^{0,\alpha}_{0;t}}}
         \cdot
         (|\eta(x)|+|\eta(y)|)
         \\
         &\leq
         c t^{3/5},
     \end{split}
 \end{align}
 where the first step is the definition of our weighted H\"{o}lder norm,
 the second step is from \eqref{equation:torsion-free-spin-7-construction-original-estimates} together with \cite[Lemma 5.1.1 (iii)]{Joyc3},
 and the last step is \eqref{equation:eta-c0alpha-estimate}.

 Now, combining \eqref{equation:eta-c0alpha-estimate}, \eqref{equation:omega-tilde-minus-omega-triangle-inequality-result}, and \eqref{equation:theta(eta)-estimate} gives the claim.
\end{proof}

\subsection{Approximate solutions}
\label{subsection:approximate-solutions}

In the following, let $G$ be a compact Lie group.
We construct a bundle $E_t$ on the compact $Spin(7)$-manifold $M_t$ and a family of approximate $Spin(7)$-instantons $A_t$ on it. 

\begin{definition}
\label{definition:compatible-gluing-data}
 We consider the following data, which we call \emph{compatible gluing data}, consisting of: 
 \begin{enumerate}[label=(\roman*)]
  \item
  an orbifold $G$-bundle $(E_0, \theta) \rightarrow T^8/\Gamma$ with flat connection $\theta$ that is infinitesimally rigid as a $Spin(7)$-instanton (see \cref{definition:rigid-irreducible-unobstructed});
  
  \item
  one point
  $x_k \in C_k \setminus S \, \, (k= 1, \dots 5)$ together with $G$-equivariant maps $\phi_k : (E_0)_{x_k} \rightarrow G$ for each connected component $C_1,C_2,C_3,C_4,C_5$ of $T:= \bigcup_{j \in \{1, \dots, 76\}} T_j$ (see \cref{equation:definition-C1-C2-C3-C4-C5}); and 
  
  \item 
  a bundle $E_j$ over Eguchi--Hanson space $X$ with a finite energy infinitesimally rigid ASD instanton $A_j$ on it that is asymptotic to the flat connection at infinity defined by the representation $\hol_\theta: \pi_1(T_j \setminus S_j) \rightarrow G$ for $j \in \{ 73, \dots, 76\}$
 \end{enumerate}
 with the property that
 \begin{enumerate}[label=(\roman*)]  
  \item[(iv)]
  The monodromy representation $\hol_{\theta}:\pi_1(T_j \setminus S_j) \rightarrow G$ is trivial for $j \in \{1, \dots, 72 \}$. 
 \end{enumerate}
\end{definition}

\vspace{0.1cm}
Given compatible gluing data, we construct a bundle $E_t$ over the compact $Spin(7)$-manifold $M_{t}$ and a connection $A_t$ on $E_t$ as follows:
for $j \in \{1,\dots,76\}$ define the $G$-bundle $\tilde{E}_j \rightarrow \tilde{T}_{j,t}^{\text{reg}}$ as 
\begin{align*}
    \tilde{E}_j :=
    \begin{cases}
        (G \times (\R^4 \times X))/(\mathbb{Z}^4 \rtimes \mathbb{Z}_2) & \text{ for } j \in \{1,\dots,8\}
        \\
        G \times (X \times X) & \text{ for } j \in \{9,\dots,72\}
        \\
        p_2^*E_j/\mathbb{Z}^4 & \text{ for } j \in \{73,\dots,76\}.
    \end{cases}
\end{align*}
Here, $p_2: \R^4 \times X \rightarrow X$ denotes projection onto the second coordinate.
The action of $\mathbb{Z}_2$ on $\R^4 \times X$ corresponds to the action of $\alpha$ near $\fix(\beta)$ and vice versa.
To be precise:
for $a \in \mathbb{Z}^4$, $b \in \mathbb{Z}_2$, and $(v,x) \in \R^4 \times X$ we have $a \cdot (v,x)=(v+a,x)$ and $b \cdot (v,x)=(-v,x)$.
The action $\mathbb{Z}^4 \curvearrowright \R^4 \times X$ is translation in the $\R^4$ coordinate.
The lifts of the two actions are induced by $\theta$.
Then set
\begin{align}
\label{equation:def-of-E-t}
 E_t :=
 \left(
 E_0|_{(T^8/\Gamma) \setminus S}
 \cup
 \bigcup_{j \in \{1, \dots, 76 \}}
 \tilde{E}_j
 \right)/\sim,
\end{align}
where $\sim$ denotes the following equivalence relation:
recall that
\[
 M_t =
 \left(
 (T^8/\Gamma) \setminus S
 \bigcup_{j \in \{1, \dots, 76 \}}
 \tilde{T}_{j,t}^{\text{reg}}
 \right)/\approx,
\]
where $\approx$ identifies points in the resolution of the overlapping areas.
We have bundles $E_0$ over $(T^8/\Gamma) \setminus S$ and $\tilde{E}_j$ over $\tilde{T}_{j,t}^{\text{reg}}$ and can lift the equivalence relation $\approx$ to an equivalence relation $\sim$ on the bundles by the compatibility conditions from Definition \ref{definition:compatible-gluing-data}.

For $j \in \{73, \dots, 76\}$, under the identification of $E_0$ and $p_2^*E_j$ on $T_j$ we can write $p_2^*A_j=\theta+a_j$.
Fix a smooth cut-off function $\overline{\chi}:[0,\zeta] \rightarrow [0,1]$ such that $\overline{\chi}|_{[0,\zeta/4]} \equiv 1$ and $\overline{\chi}|_{[\zeta/2, \zeta]} \equiv 0$.
Let $\chi=\overline{\chi} \circ d_{T^8/\Gamma}(\pi(\cdot), \fix \Gamma)$ and define
the \emph{gluing region} by 
\begin{align}
\label{equation:gluing-region}
 O := \left\{ x \in M_t:
 \frac{\zeta}{4} \leq d_{T^8/\Gamma}(\pi(x) , \fix \Gamma) \leq \frac{\zeta}{2}
 \right\}.
\end{align}
We then interpolate between $p_2^*A_j$ and $\theta$ on $O$ as follows:
\begin{align}
\label{equation:approximate-spin7-instanton}
 A_t(x)
 :=
 \begin{cases}
  p_2^*A_j(x)
  & \text{ if }
  d_{T^8/\Gamma}(\pi(x) , \fix \Gamma) \leq 
  \frac{\zeta}{4}, 
  \\
  \theta(x)+\chi(x)a_j(x)
  & \text{ if }
  x \in O, \text{ and }
  \\
  \theta(x), 
  & \text{ otherwise.}
 \end{cases}
\end{align}

\begin{remark}
\label{remark:pontrjagin-classes}
Denote by $[S_j]$ a homology class in $H_4 (M_t, \Z)$, which is defined as $(i_{j,t})_{*} (T^4 \times \{ x \} )$, where $i_{j,t} : \tilde{T}_{j,t} \to M_t$ is the inclusion, and $x \in \pi^{-1}_{j, t} (B^4_{\zeta} / \{ \pm 1 \} )$.  
Write the adjoint bundle of $E_t$ by $\Ad E_t$. 
Then, from the above construction, one sees the first Pontryagin class of 
$\Ad  E_t$ is given by $\displaystyle{p_1 (\Ad E_t) = - \sum_{j=73}^{76}  k_j \text{P.D.}  [S_j]}$, where 
 $k_j$ is $\displaystyle{\frac{1}{8 \pi^2} \int_X | F_{A_j} |^2 d\text{vol}}$, $\text{P.D.}$ stands for Poincar\'e dual, and we used $p_2 (\Ad E_t) = 0$. 
We omit the proof of these claims here, since they are almost the same as ones given by Walpuski for the $G_2$-instanton case (see \cite[Proposition 2.60]{Walp3}). 
\end{remark}

\subsection{Pregluing estimate}
\label{subsection:pregluing-estimate}

We prove an estimate on the approximate solution $A_t$. 

\begin{proposition}
\label{proposition:pregluing-estimate}
The approximate solution $A_t$ satisfies:
 \begin{align}
  \|{
   \left( \pi^2_7 \right)^{\tilde{\Omega}^t}
   (F_{A_t})
  }_{C^{0,\alpha}_{-2; t}}
  \leq
  ct^{3/10},
 \end{align}
 where $\left( \pi^2_7 \right)^{\tilde{\Omega}^t}: \Omega^2(M_t) \rightarrow \Omega^2_7(M_t)$ is defined using the torsion-free $Spin(7)$-structure $\tilde{\Omega}^t$.
\end{proposition}

\begin{proof}
 First note that $A_t$ is flat everywhere except on $\pi_{j,t}^{-1}(T_j)$ for $j \in \{73, \dots, 76\}$, so it suffices to estimate $\pi^2_7(F_{A_t})$ there.
 Denote by $\left( \pi^2_7 \right)^{\Omega_j}: \Omega^2(M_t) \rightarrow \Omega^2_7(M_t)$ the projection defined by the $Spin(7)$-structure $\Omega_j$ on $\pi_{j,t}^{-1}(T_j)$.
 Then we have: 
 \begin{align}
 \label{equation:pregluing-pi-2-7-omega-tilde-line}
 \begin{split}
  \|{
   \left( \pi^2_7 \right)^{\tilde{\Omega}^t}
   (F_{A_t})
  }_{C^{0,\alpha}_{-2; t}}
  &\leq
  \|{
   \left( 
   \left( \pi^2_7 \right)^{\tilde{\Omega}^t}
   -
   \left( \pi^2_7 \right)^{\Omega_j}
   \right)
   (F_{A_t})
  }_{C^{0,\alpha}_{-2; t}}
  \\
  &\quad
  +
  \|{
   \left( \pi^2_7 \right)^{\Omega_j}
   (F_{A_t})
  }_{C^{0,\alpha}_{-2; t}}.
 \end{split}
 \end{align}
 We now estimate the first summand.
 As before, $\rho:X \rightarrow \C^2/\{\pm 1\}$ denotes the blowup map of the Eguchi--Hanson space.
 Observe that
 \begin{align*}
  | (t+r_t)^{2+k} \nabla^k F_{A_t} |_{g_t}
  &\leq
  t^{-2} | (t+r_t)^{4+k} \nabla^k F_{A_t} |_{g_t}
  \\
  &\leq
  c
  t^{2+k} | (1+|\rho|)^{4+k} \nabla^k F_{A_j} |_{g_t} 
  \\
  &=
  c
  | (1+|\rho|)^{4+k} \nabla^k F_{A_j} |_{g_X}
  \\
  &\leq
  c,
 \end{align*}% fixed 2024-05-14
 where we used $r_t=|\rho \circ p_2|$ on $\pi_{j,t}^{-1}(T_j)$, where $p_2: T^4 \times X \rightarrow X$ denotes projection onto the second component, in the second step;
 we used $g_t|_{\{p\} \times X}=t^2 g_X$ for any $p \in T^4$ on $\pi^{-1}(T_j)$ in the third step;
 and we used \cref{proposition:framed-asd-decay} in the last step.
 Thus, 
 \begin{align}
 \label{equation:approx-instanton-curvature-estimate}
  \|{
   F_{A_t}
  }_{C^{0,\alpha}_{-2; t}}
  &
  \leq
  c.
 \end{align}
 We can turn to the right-hand side of \eqref{equation:pregluing-pi-2-7-omega-tilde-line}.
 Because $\tilde{\Omega}^t$ and $\Omega_j$ are close one expects that projections with respect to them are close, too.
 That is indeed the case, and we obtain:
 \begin{align}
 \label{equation:pregluing-pi-2-7-omega-difference-line}
 \begin{split}
  \|{
   \left( 
   \left( \pi^2_7 \right)^{\tilde{\Omega}^t}
   -
   \left( \pi^2_7 \right)^{\Omega_j}
   \right)
   (F_{A_t})
  }_{C^{0,\alpha}_{-2; t}}
  &\leq
  c
  \|{ 
%   \left( \pi^2_7 \right)^{\tilde{\Omega}^t}
%   -
%   \left( \pi^2_7 \right)^{\Omega_j}
   \tilde{\Omega}^t
   -
   \Omega_j
  }_{C^{0,\alpha}_{0; t}}
  %\\
  %&\quad
  \cdot
  \|{
   F_{A_t}
  }_{C^{0,\alpha}_{-2; t}}
  \\
  &\leq
  ct^{3/10}.
 \end{split}
 \end{align}
 Here we used \eqref{eq:spin7-projections} in the first step, and \cref{proposition:spin-7-structure-holder-estimate,proposition:spin-7-difference-to-product-structure} together with the bound \eqref{equation:approx-instanton-curvature-estimate} in the second step.
 For the second summand, recall that on $T_j$ for $j \in \{ 73, \dots, 76\}$ we have $A_t=\theta+\chi a_j$, therefore
 \begin{align}
 \notag
  F_{p_2^*A_j}
  &=
   \underbrace{F_\theta}_{=0}+\d_\theta a_j + \frac{1}{2} [a_j,a_j]
  ,
  \text{ which implies}
  \\
 \label{equation:pregluing-curvature-computation}
 \begin{split}
  F_{A_t}
  &=
  \underbrace{F_\theta}_{=0}
  +\d _{\theta}( \chi a_j)
  +\frac{1}{2} \chi^2 [a_j,a_j]
  \\
  &=
  (\d \chi)\wedge a_j + \chi F_{p_2^*A_j}
  +
  \frac{\chi^2-\chi}{2}[a_j,a_j]
 \end{split}
 \end{align}
 The connection $A_j$ is an ASD instanton, so $p_2^*A_j$ is a $Spin(7)$-instanton with respect to $\Omega_j$.
 Furthermore, $A_j$ is asymptotic to some flat connection $A_j^\infty$ with rate $-3$ by \cref{proposition:framed-asd-decay}, and therefore
 \begin{align}
 \label{equation:pregluing-difference-decay}
  \nabla^k a_j
  =
  t^{2+k} \mathcal{O}(r_t^{-3-k}).
 \end{align}
 Combining \cref{equation:pregluing-curvature-computation,equation:pregluing-difference-decay}, we find that
 \begin{align}
 \label{equation:pregluing-pi-2-7-omega-line}
 \begin{split}
  \|{
   \left( \pi^2_7 \right)^{\Omega_j}
   (F_{A_t})
  }_{C^{0,\alpha}_{-2; t}}
  &=
  \|{
   \left( \pi^2_7 \right)^{\Omega_j}
   (F_{A_t})
  }_{C^{0,\alpha}_{-2; t}(O)}
  \\
  &\leq
  c
  \|{
   \left( \pi^2_7 \right)^{\Omega_j}
   (F_{A_t})
  }_{C^{0,\alpha}(O)}
  \\  
  &\leq
  c t^2.
 \end{split}
 \end{align}
 For the first equality we used that the factors $\d \chi$ and $\frac{\chi^2-\chi}{2}$ from \cref{equation:pregluing-curvature-computation} are supported in $O$.
 In the second step, we used that, in $O$, the weight function $(t+r_t)$ is uniformly bounded by a constant.
 Estimating the terms in the right-hand side of \eqref{equation:pregluing-pi-2-7-omega-tilde-line} by the bounds \eqref{equation:pregluing-pi-2-7-omega-difference-line} and \eqref{equation:pregluing-pi-2-7-omega-line}, we obtain the claim.
\end{proof}

\section{Linear problems on local models}
\label{sec:linear}

We will now make the first step towards perturbing the approximate $Spin(7)$-instantons from the previous sections to genuine $Spin(7)$-instantons.
Like in other gluing problems, this makes use of a fixed point theorem and elliptic estimates for the linearisation of the partial differential equation to be solved, for us the $Spin(7)$-instanton equation.
The proof of the estimate for the linearisation of the instanton equation on $M_t$ will be given in Section \ref{section:construction}.
It will crucially use local analysis on the gluing regions of $M_t$.
Recall the singular sets $S_1, S_2, \dots, S_{76}$ in $T^8/\Gamma$ from \cref{table:singular-set-in-t8}.
$S_1,\dots, S_8$ are isometric to $T^4/\{ \pm 1 \}$; $S_{73}, \dots, S_{76}$ are isometric to $T^4$; and analysis around the resolution of these singularities will reduce to analysis on $\R^4 \times X$, which is explained in Section \ref{subsection:operator-on-r4-times-ale}. 
The singular sets $S_9, \dots, S_{72}$ are points, and analysis around the resolution of these singularities will reduce to analysis on $X \times X$, which is explained in Section \ref{subsection:operator-on-x-x}.
Here, $X$ denotes the Eguchi--Hanson space as usual and it comes with the blowup map $\rho: X \rightarrow \C^2/\{ \pm 1 \}$.

\subsection{Linear problem on $\R^4 \times X$}
\label{subsection:operator-on-r4-times-ale}

The estimates stated in this section are mostly adaptations of similar ones in Walpuski \cite{Walp3} to our $Spin(7)$-instanton setting.

In \cref{definition:weighted-hoelder-norm-on-Mt} we defined weighted H\"{o}lder norms on $M_t$.
In the following definition, we define weighted H\"{o}lder norms on the gluing piece $\R^4 \times X$.
The norms are related by a rescaling, which is made precise at the end of this section in \cref{proposition:s-beta-estimates}.

\begin{definition}[Weighted H\"{o}lder norm]
\label{definition:weighted-hoelder-norm-on-Eguchi-times-R4}
 We denote by $p_2$ the projection onto the second factor of $\R^4 \times X$ and define the weight function by 
 \[
  w(x)
  =
  1+
  |\rho \circ p_2(x)| . 
 \]
 We define $w(x,y)$,
  $[\cdot]
  _{C^{0,\alpha}_\beta(U)}$,
  $\|{\cdot}_{L^\infty_\beta(U)}$,
 and $\|{\cdot}_{C^{k,\alpha}_\beta(U)}$
 in the same manner as in \cref{definition:weighted-hoelder-norm-on-Mt}.
\end{definition}

The Eguchi--Hanson space $X$ is a hyper-K\"{a}hler manifold which implies that $\R^4 \times X$ is a $Spin(7)$-manifold.
This is essentially a linear algebra computation which is explained in the following lemma from Walpuski \cite{Walp3}:

\begin{lemma}[Examples 1.11, and 1.44 in \cite{Walp3}]
\label{lemma:g2-spin7-linear-algebra}
Let $\omega = (\omega_1, \omega_2, \omega_3)$ and $\mu = ( \mu_1,\mu_2,\mu_3) $ be two hyper-K\"{a}hler triples on $\R^4$. 
Consider $\R^8 = (\R^4, \mu ) \times (\R^4 , \omega )$. 
Let $(x^0,x^1,x^2,x^3)$ be coordinates on $(\R^4 , \mu )$ such that
\begin{align*}
 \mu_1=\d x^{01}+\d x^{23}, \;
 \mu_2=\d x^{02}-\d x^{13}, \;
 \mu_3=\d x^{03}+\d x^{12}.
\end{align*}
Then we have that four-form $\Omega_0 :=
  \frac{1}{2} \omega_1 \wedge \omega_1
  +
  \frac{1}{2} \mu_1 \wedge \mu_1
  -
  \sum_{i=1}^3
  \omega_i \wedge \mu_i$ on $(\R^4 , \mu ) \times ( \R^4 , \omega )$ 
  has stabiliser $Spin(7)$.
\end{lemma}

The following is about how the linearised operator of the $Spin(7)$-instanton equation on $\R^4 \times X$ and that of the ASD instanton equation on $X$ are related:

\begin{corollary}[Equation 5.23 in \cite{Walp3}]
 Let $(E,A)$ be a bundle with a finite energy ASD instanton over a hyper-K\"{a}hler four-manifold $X$.
 Denote by $p_X:\R^4 \times X \rightarrow X$ the projection onto the second factor.
 Then $(p_X^* E, p_X^*A)$ is a $Spin(7)$-instanton with respect to the $Spin(7)$-structure from Lemma \ref{lemma:g2-spin7-linear-algebra}.
 
 Then we have identifications
 \begin{align*}
  \Lambda^1(\R^4 \times X)
  &\simeq
  \Lambda^0(\R^4 \times X)
  \oplus
  p_X^*\Lambda^1(X)
  \oplus
  p_X^*\Lambda^2_+(X),
  \\
  \Lambda^0(\R^4 \times X) \oplus \Lambda^2_7(\R^4 \times X)
  &\simeq
  p_X^* \Lambda^1(X)
  \oplus
  p_X^* \Lambda^2_+(X)
  \oplus
  p_X^* \Lambda^0(X), 
 \end{align*}
where $\Lambda^2_+ (X)$ is the bundle of self-dual two-forms on $X$, 
 under which we obtain 
 \begin{align}
 \label{equation:LL*-L*L-on-R4-times-X}
  L_{p_X^*A} L_{p_X^*A}^*
  =
  L_{p_X^*A}^* L_{p_X^*A}
  =
  \Delta_{\R^4}
  +
  \begin{pmatrix}
   \delta_A \delta_A^* & 0 \\
   0 & \delta_A^* \delta_A
  \end{pmatrix}, 
 \end{align}
where $L_{p_X^*A}$ is the linearised operator of the $Spin(7)$-instanton equation at $p_X^*A$ from \cref{equation:linearised-operator}, and 
 \begin{align*}
  \delta_A : \Omega^1(X, \Ad( E)) & \rightarrow \Omega^0(X, \Ad (E)) \oplus \Omega^2_+(X, \Ad (E))
  \\
  a & \mapsto (\d^* _A a, \d^+ _A a) 
 \end{align*}
is the linearised operator of the ASD instanton equation on $X$.
The operator $\Delta$ acts on $C^\infty(\R^4 \times X,\Lambda^k(\R^4 \times X))$ by identifying $C^\infty(\R^4 \times X,\Lambda^k(\R^4 \times X)) \simeq C^\infty(\R^4, C^\infty(X, \Lambda^k(\R^4 \times X)))$.
\end{corollary}

We state the following auxiliary lemma that will be used at different points throughout the article:

\begin{lemma}[Lemma 2.76 in \cite{Walp3}]
\label{lemma:constant-in-R-n-direction}
 Let $E$ be a vector bundle of bounded geometry over a Riemannian manifold $X$ of bounded geometry and with subexponential volume growth, and suppose that $D: C^\infty(X,E) \rightarrow C^\infty(X,E)$ is a uniformly elliptic operator of second order whose coefficients and their first derivatives are uniformly bounded, that is non-negative, i.e. $\langle Da, a \rangle \geq 0$ for all $a \in L^2_2(X,E)$, and formally self-adjoint.
 If $a \in C^\infty(\R^n \times X, E)$ satisfies
 \[
  (\Delta_{\R^n} + D)a=0
 \]
 and $\|{a}_{L^\infty}$ is finite, then $a$ is constant in the $\R^n$-direction, that is $a(x,y)=a(y)$.
 Here, by slight abuse of notation, let $E$ denote the pullback of $E$ to $\R^n \times X$ as well.
\end{lemma}

By similar arguments as in the proof of \cite[Proposition 2.74]{Walp3}, we obtain the following \cref{lemma:kernel-of-pullback-instanton} as a consequence of \cref{lemma:constant-in-R-n-direction}:

\begin{lemma}
\label{lemma:kernel-of-pullback-instanton}
 Let $\beta < 0$.
 Then $a \in C^{1,\alpha}_{\beta}$ is in the kernel of $L_{p^*A}: C^{1,\alpha}_\beta \rightarrow C^{0,\alpha}_{\beta-1}$ if and only if it is given by the pullback of an element of the $L^2$-kernel of $\delta_A$ to $\R^4 \times X$.
\end{lemma}

This concludes the analysis of the $Spin(7)$-equation on the model space $\R^4\times X$, but there is one thing left to do before ending the section.
In Section \ref{section:construction}, we will prove estimates for the linearised $Spin(7)$-instanton operator $L_{A_t}$ on the manifold $M_t$.
We will do this by rescaling around the glued-in parts, thereby comparing sections on $\R^4 \times X$ and on $M_t$.
This rescaling is made precise by the map $s_{j,\beta,t}$ defined below.

Fix $j \in \{1, \dots, 76\}$ such that the singularity $S_j$ is isometric to $T^4$ (that is $j \in \{ 73, \dots, 76 \}$) or $T^4/\{ \pm 1 \}$ (that is $j \in \{1, \dots, 8\}$).

Part of the compatible gluing data (cf. \cref{definition:compatible-gluing-data}) is then a bundle $E_j$ over 
the Eguchi--Hanson space $X$ carrying an infinitesimally rigid ASD connection $A_j$.
We define 
\begin{align*}
 \tilde{\iota}_{j,t}: \R^4 \times \rho^{-1}(B^4_{t^{-1}\zeta} / \{ \pm 1\}) & \rightarrow T^4 \times \rho^{-1}(B^4_{t^{-1}\zeta} / \{ \pm 1\}) \subset M_t
 \text{ for } j \in \{73,\dots,76\}
 \\
 (x,y) & \mapsto [tx,y],
 \\
 \tilde{\iota}_{j,t}: \R^4 \times \rho^{-1}(B^4_{t^{-1}\zeta} / \{ \pm 1\}) & \rightarrow T^4/\{ \pm 1\} \times \rho^{-1}(B^4_{t^{-1}\zeta} / \{ \pm 1\}) \\ & \quad \subset M_t
 \text{ for } j \in \{1,\dots,8\}
 \\
 (x,y) & \mapsto [tx,y]. 
\end{align*}
Then, for $\beta \in \mathbb{R}$ we consider
\begin{align}
\label{equation:def-s-beta-t-R4-times-X}
\begin{split}
 s_{j,\beta,t}:
 \Omega^k(M_t, \Ad E_t) &\rightarrow
 \Omega^k(\R^4 \times \rho^{-1}(B^4_{t^{-1}\zeta}/ \{ \pm 1\}), p_2^*\Ad E_j)
 \\
 a & \mapsto
 t^{\beta-k}
 (\tilde{\iota}_{j,t})^* a.
\end{split}
\end{align}
Here, $p_2:\R^4 \times \rho^{-1}(B^4_{t^{-1}\zeta}/ \{ \pm 1\}) \rightarrow X$ denotes the projection onto the second factor so that $p_2^* A_j$ is a $Spin(7)$-instanton on the bundle $p_2^* E_j$.
The map $s_{j,\beta,t}$ has the property that it is close to an isometry between the weighted H\"{o}lder norm $C^{k,\alpha}_\beta$ on $\R^4 \times X$ and the weighted H\"{o}lder norm $C^{k,\alpha}_{t ; \beta}$ on $M_t$.

This is made precise in Proposition \ref{proposition:s-beta-estimates} below.
The map is not precisely an isometry because the model $Spin(7)$-structure on $\R^4 \times X$ and the torsion-free $Spin(7)$-structure on $M_t$ are not the same, but their difference is roughly of size $t^{3/10}$, which was shown in \cref{proposition:spin-7-structure-holder-estimate}.
For the $G_2$-case, the statement can be found as \cite[Proposition 7.7]{Walp1}, and we omit the proof here.

\begin{proposition}
\label{proposition:s-beta-estimates}
 There is a constant $c > 0$ such that
 \begin{align*}
  \frac{1}{c}
  \|{a}
  _{C^{k,\alpha}_{\beta,t}(\tilde{T}_{j,t})}
  \leq
  \|{s_{j,\beta,t}a}
  _{C^{k,\alpha}_{\beta}(\R^4 \times \rho^{-1}(B^4_{t^{-1}\zeta}/ \{ \pm 1\}))}
  \leq
  c
  \|{a}
  _{C^{k,\alpha}_{\beta,t}(\tilde{T}_{j,t})},
  \\
  \|{
   L_t a - s_{j,\beta-1,t}^{-1} L_{A_i} s_{j,\beta,t} a
  }
  _{C^{0,\alpha}_{\beta-1,t}(\tilde{T}_{j,t})}
  \leq
  c t^{3/10}
  \|{a}
  _{C^{0,\alpha}_{\beta,t}(\tilde{T}_{j,t})},
  \\
  \|{
   L_t^* \underline{b} - s_{j,\beta-1,t}^{-1} L_{A_i}^* s_{j,\beta,t} \underline{b}
  }
  _{C^{0,\alpha}_{\beta-1,t}(\tilde{T}_{j,t})}
  \leq
  c t^{3/10}
  \|{\underline{b}}
  _{C^{0,\alpha}_{\beta,t}(\tilde{T}_{j,t})}.
 \end{align*}
\end{proposition}

\subsection{Linear problem on $X \times X$}
\label{subsection:operator-on-x-x}

Around point singularities on $T^8/\Gamma$ we glue in copies of $X \times X$ for the manifold construction.
For our $Spin(7)$-instanton construction we glue in the trivial connection over $X \times X$.
Because of this, we study the deformation theory of the trivial connection over $X \times X$ in this section.

\begin{definition}[Weighted H\"{o}lder norm]
\label{definition:weighted-hoelder-norm-on-product-of-Eguchi-Hansons}
 We denote by $p_i \, (i=1,2)$  the projection onto the $i$-th factor of $X \times X$, or $(\C^2 /\{ \pm 1 \}) \times (\C^2 /\{ \pm 1 \})$, according to context.
 Define the weight function by 
 \begin{align*}
  &w(x)
  =
  1+
  \min \{ |\rho \circ p_1(x)|, |\rho \circ p_2(x)| \}
  &&
  \text{ on }
  X \times X,
  \\
  &w(x)
  =
  1+
  \min \{ | p_1(x)|, | p_2(x)| \}
  &&
  \text{ on }
  (\C^2 /\{ \pm 1 \}) \times (\C^2 /\{ \pm 1 \}).
 \end{align*}
 We then define $w(x,y)$,
  $[\cdot]
  _{C^{0,\alpha}_\beta(U)}$,
  $\|{\cdot}_{L^\infty_\beta(U)}$, and 
  $\|{\cdot}_{C^{k,\alpha}_\beta(U)}$ in the same way 
 as in \cref{definition:weighted-hoelder-norm-on-Mt}.
\end{definition}

The manifold $X \times X$ is a special case of a QALE manifold, defined in Joyce \cite{Joyc5}.
Our norms are equivalent to the norms from \cite[Definition 9.5.5]{Joyc5}, if one chooses $\beta=0$ in that definition.
In turn, QALE manifolds are special cases of the QAC manifolds from Degeratu--Mazzeo \cite{Degeratu} and how weighted norms on these spaces relate to the weighted norms in \cite{Joyc5} is explained in \cite[Section 8.1]{Degeratu}.

The goal of this section is to prove the following:

\begin{proposition}
    \label{proposition:trivial-connectionon-X-X-unobstructed-rigid}
    Let $L:=(\d^*, \d_7): \Omega^1(X \times X) \rightarrow (\Omega^0 \oplus \Omega^2_7)(X \times X)$ and let $L^*$ be its formal adjoint.
    Then 
    \[
        \Ker L=0,
        \Ker L^*=0,
    \]
    if acting on $C^{1,\alpha}_\beta$-sections for any $\beta < 0$.    
\end{proposition}

We note that $X \times X$ is a QALE manifold asymptotic to $(\C^2 /\{ \pm 1 \}) \times (\C^2 /\{ \pm 1 \})$, and first check that the corresponding statement on the asymptotic limit is true:

\begin{proposition}
    \label{poposition:trivial-connection-rigid-on-c2-times-c2}
    Let $L:=(\d^*, \d_7): \Omega^1((\C^2 /\{ \pm 1 \}) \times (\C^2 /\{ \pm 1 \})) \rightarrow (\Omega^0 \oplus \Omega^2_7)((\C^2 /\{ \pm 1 \}) \times (\C^2 /\{ \pm 1 \}))$ and let $L^*$ be its formal adjoint.
    Then 
    \[
        \Ker L=0,
        \Ker L^*=0,
    \]
    if acting on $C^{1,\alpha} \cap C^0_\beta$-sections for any $\beta < 0$.
\end{proposition}

\begin{proof}
    Let $a \in \Omega^1((\C^2 /\{ \pm 1 \}) \times (\C^2 /\{ \pm 1 \}))$ such that $La=0$.
    We pull back $a$ and $L$ under the quotient map
    $\C^2 \times \C^2 \rightarrow (\C^2 /\{ \pm 1 \}) \times (\C^2 /\{ \pm 1 \})$
    and will carry on with analysis on $\C^2 \times \C^2$, without changing notation.
    
    Elliptic regularity implies that $a$ is smooth, so $La=0$ implies $L^* L a=0$.
    By \cref{equation:LL*-L*L-on-R4-times-X}:
    \[
    \left(
    \Delta_{\C^2}
  +
  \begin{pmatrix}
   \delta \delta^* & 0 \\
   0 & \delta^* \delta
  \end{pmatrix}
  \right)
  a=0
    \]
    where $\delta$ denotes the ASD deformation operator of the trivial connection on $\C^2$.
    By \cref{lemma:constant-in-R-n-direction} we have that $a$ is a pullback from $\C^2$.
    In particular its norm in the $(\C^2 \times \{0\})$-direction is constant, i.e. $|a|(x,y)=|a|(x',y)$ for all $x,x',y \in \C^2$.
    By swapping the two $\C^2$-factors in the argument we see that $|a|$ is constant on all of $\C^2 \times \C^2$.
    That is a contradiction to $a \in C^0_\beta$.

    The proof for $\Ker L^*=0$ is analogous.
\end{proof}

The idea for the following proof of \cref{proposition:trivial-connectionon-X-X-unobstructed-rigid} is taken from \cite[Section 4.2.3]{Chiu}.

\begin{proof}[Proof of \cref{proposition:trivial-connectionon-X-X-unobstructed-rigid}]
    We only prove $\Ker L^*=0$, the proof for $\Ker L=0$ is analogous.
    Assume there is $0 \neq a \in \Ker L^*$ with
    \begin{align}
        \label{equation:c2alpha-bound-on-x-times-x}
        \|{a}_{C^{1,\alpha}_\beta}<c_1.
    \end{align}
    By elliptic regularity, $a$ is smooth.
    By \cite[Section 3.3]{Lewi}, $L$ can be identified with the Dirac operator.
    By the Lichnerowicz formula \cite[Theorem II.8.8]{LaMi} we thus have $\nabla^*\nabla a=LL^*a=0$ and therefore $|a|^2$ is a subharmonic function, i.e.
    % \begin{align*} % old
    %     \label{equation:a-harmonic}
    %     \Delta |a|^2
    %     =
    %     |\nabla a|^2+
    %     \langle \nabla^* \nabla a, a \rangle
    %     =
    %     |\nabla a|^2
    %     \geq 0. 
    % \end{align*}
    \begin{align} % new
        \label{equation:a-harmonic}
        \Delta |a|^2
        =
        -2|\nabla a|^2
        +
        2
        \langle \nabla^* \nabla a, a \rangle
        =
        -2|\nabla a|^2
        \leq 0.
    \end{align}
    Here, $\Delta$ denotes Laplace operator with the sign convention that $\Delta$ be positive definite.
    The Weitzenböck formula \cite[Equation 6, Corollary 2.3]{Semm} states that $\nabla^* \nabla=\Delta$ on $\Lambda^0 \oplus \Lambda^2_7$, and hence that $\Delta a=0$.
    The plan of the proof is to define a sequence of rescalings $a_{\epsilon_i}$ converging to a harmonic form on the singular space $\C^2/\{\pm 1\} \times \C^2/\{\pm 1\}$, which will lead to a contradiction to Proposition \ref{poposition:trivial-connection-rigid-on-c2-times-c2}.

    On $X \times X$ we have a natural one-parameter family of metrics $g_\epsilon$ (see \cite[p.160]{Joyc5}), where $\epsilon$ is the size of a minimal $S^2$ in each of the $X$ factors.
    These are conformally equivalent:
    there exists a diffeomorphism $\phi_\epsilon: X \times X \rightarrow X \times X$ satisfying $(\phi_\epsilon)^* g_\epsilon=\sqrt{\epsilon} g_1$.
    This follows abstractly from Kronheimer's classification of four-dimensional ALE manifolds, see \cite[Theorem 7.2.3]{Joyc5} for a summary.
    Explicitly, $\phi_\epsilon$ is given by rescaling the fibre direction when viewing $X$ as a fibration over $\mathbb{CP}^1$.
    
    Fix $x_0 \in X \times X$ with $a(x_0) \neq 0$ and define
    \[
        a_\epsilon
        :=
        \frac{(\phi_\epsilon)_* a}{\|{(\phi_\epsilon)_* a}_{L^2(B(\phi_\epsilon(x_0),1)),g_\epsilon}},
    \]
    where for $\delta>0$ we write $B(\phi_\epsilon(x_0),\delta)$ for the ball around $\phi_\epsilon(x_0)$ with radius $\delta$, 
    and $\|{\cdot }_{L^2(B(\phi_\epsilon(x_0),1)),g_\epsilon}$ denotes the $L^2$-norm with respect to $g_\epsilon$.
    The definition of $B(\phi_\epsilon(x_0),\delta)$ depends on $\epsilon$, which is missing from our notation by slight abuse of notation.
    The function $|a|^2$ is subharmonic by \eqref{equation:a-harmonic} and therefore satisfies the mean-value inequality \cite[Theorem 1.2]{LiSc}:
    \begin{align}
        \label{equation:mean-value-inequality}
        |a(x_0)|^2
        \leq 
        c_3 \fint_{B(x_0,\epsilon^{-1})} |a|^2
    \end{align}
    for some constant $c_3>0$, independent of $\epsilon$.
    Here, $\fint_S f dV$ denotes the average value of $f$ on the set $S$.
    Thus
    \begin{align*}
        \|{a_\epsilon}^2_{L^2(B(\phi_\epsilon(x_0),2)),g_\epsilon}
        &=
        \frac{\|{(\phi_\epsilon)_* a}_{L^2(B(\phi_\epsilon(x_0),2)),g_\epsilon}^2}{\|{(\phi_\epsilon)_* a}_{L^2(B(\phi_\epsilon(x_0),1)),g_\epsilon}^2}
        \\
        &=
        \frac{\|{a}_{L^2(B(x_0,2\epsilon^{-1})),g_\epsilon}^2}{\|{a}_{L^2(B(x_0,\epsilon^{-1})),g_\epsilon}^2}
        \\
        &=
        \frac{\fint _{B(x_0,2\epsilon^{-1})} |a|^2}{\fint _{B(x_0,\epsilon^{-1})} |a|^2}
        \cdot
        \frac{\vol(B(x_0,2\epsilon^{-1}))}{\vol(B(x_0,\epsilon^{-1})) }       
        \\
        &\leq
        \frac{c_1}{|a(x_0)|^2 \cdot (1/c_3)} \cdot c_2
        \\
        &\leq
        c_4,
    \end{align*}
    where in the second to last step we used \eqref{equation:mean-value-inequality} to estimate the denominator.
    The numerator is the average value of $|a|^2$, which is bounded by the unweighted $C^0$-norm of $a$, which is bounded by \eqref{equation:c2alpha-bound-on-x-times-x}, because $\beta<0$.

    We therefore have for $y \in B(\phi_\epsilon(x_0),1)$ that
    \begin{align}
        \label{equation:a-epsilon-uniform-bound-on-ball}
        |a_\epsilon(y)|^2
        \leq
        c_5 \fint_{B(y,1)} |a_\epsilon|^2
        \leq
        c_6 \|{a_\epsilon}^2_{L^2(B(\phi_\epsilon(x_0),2)),g_\epsilon}
        \leq
        c_4 \cdot c_6,
    \end{align}
    where in the first step we used the mean value inequality around $y$.
    Furthermore, by the inequality \eqref{equation:c2alpha-bound-on-x-times-x} we have $c>0$ such that
    \[
        \frac{|a(x)|w^{-\beta}(x)}{\|{a}_{C^0(B(x_0,1))}}
        \leq
        c
    \]
    for all $x \in X \times X$.
    Therefore
    \begin{align}
        \label{equation:a-epsilon-global-uniform-bound}
        \begin{split}
        \frac{|a_\epsilon(y)|w^{-\beta}(y)}{\|{a_\epsilon}_{C^0(B(\phi_\epsilon(x_0),1))}}
        &=
        \frac{|a(\phi_\epsilon^{-1}(y))|}{\|{a}_{C^0(\phi_\epsilon^{-1}(B(\phi_\epsilon(x_0),1)))}}
        \cdot w^{-\beta}(y)
        \\
        &
        \leq
        \frac{|a(\phi_\epsilon^{-1}(y))|}{\|{a}_{C^0(B(x_0,1))}}
        \cdot
        w^{-\beta}(\phi_\epsilon^{-1}(y))
        \\
        &\leq
        c
        \end{split}
    \end{align}
    for all $y \in X \times X$, where in the second step we used $w^{-\beta}(y)\leq w^{-\beta}(\phi^{-1}_\epsilon(y))$ which holds if $\beta \leq 0$ and $\epsilon \leq 1$, and $B(x_0,1) \subset \phi^{-1}_\epsilon(B(\phi_\epsilon(x_0),1))$.
    The estimate \eqref{equation:a-epsilon-uniform-bound-on-ball} gives a uniform bound for $a_\epsilon$ on $B(\phi_\epsilon(x_0),1)$.
    Together with the bound \eqref{equation:a-epsilon-global-uniform-bound} this implies a uniform $C^0_{\beta}$-bound for $a_\epsilon$ on $X \times X$, independent of $\epsilon$:
    \begin{align}
        \label{equation:a-epsilon-global-uniform-bound-short-form}
        \|{a_\epsilon}_{C^0_\beta}
        \leq
        c.
    \end{align}
    Let $K \subset U \subset X \times X$, where $K$ is compact and $U$ is open.
    By general elliptic theory, we have the following Schauder estimate for the Laplacian
    \[
        \|{u}_{C^{2,\alpha}(K),g_\epsilon}
        \leq
        c_{K,U,\epsilon}
        \left(
        \|{\Delta u}_{C^{0,\alpha}(U),g_\epsilon}+\|{u}_{C^0(U),g_\epsilon}
        \right),
    \]
    for all $u \in C^{2,\alpha}(\Lambda^2(U))$
    where $c_{K,U,\epsilon}$ is a constant depending on $K,U,$ and $\epsilon$.
    If $U$ has a positive distance from $(S^2 \times X) \cup (X \times S^2)$, then $c_{K,U,\epsilon}$ can be chosen to be independent of $\epsilon$, because the metric on $U$ converges to the flat metric.
    Applying the Schauder estimate to the harmonic element $a_\epsilon$ and using that $|a_\epsilon|$ is bounded on every bounded set, we obtain a bound $\|{a_\epsilon}_{C^{2,\alpha}(K)}<c$ independent of $\epsilon$.
    By the Arzela--Ascoli theorem we can therefore extract a limit, up to passing to a subsequence, $a_\epsilon \rightarrow a^*$ in $K$.
    By exhausting all of $X \times X \setminus (S^2 \times X \cup X \times S^2)$ with compact sets, we can use a diagonal argument to extract a limit $a^* \in (\Omega^0 \oplus \Omega^2_7)((\C^2/\{ \pm 1\} \setminus 0) \times (\C^2/\{ \pm 1\} \setminus 0))$, i.e. $a_\epsilon \rightarrow a^*$ in $C^{2,\alpha/2}$ on compact sets.

    Let $y_0 \in \{0\} \times (\C^2/\{ \pm 1\} \setminus 0) \cup (\C^2/\{ \pm 1\} \setminus 0) \times \{0\}$, i.e., one of the points where $a^*$ is not defined.
    Let $U$ be a bounded open neighbourhood of $y_0$.
    As $a_\epsilon$ is uniformly bounded on $U$, we have that there exists a constant $c>0$ such that $|a^*(y)|<c$ for $y \in U$, where defined.
    Thus $a^* \in L^2((\Lambda^0 \oplus \Lambda^2_7)(U))$, and is therefore smooth on $U$ by elliptic regularity.
    Repeating this argument on all of $\{0\} \times (\C^2/\{ \pm 1\} \setminus 0) \cup (\C^2/\{ \pm 1\} \setminus 0) \times \{0\}$, we find that $a^*$ is smooth on all of $\C^2/\{\pm 1\} \times \C^2/\{\pm 1\}$.

    Denote the blowup map $\pi: X \times X \rightarrow \C^2/\{\pm 1\} \times \C^2/\{\pm 1\}$.
    Using this notation, we find that
    \begin{align}
        \label{equation:fatou-application}
        \begin{split}
        \int_{B(0,1)} |a^*|_{g_0}^2
        &=
        \int_{B(0,1)\setminus (\{0\}\times \C^2/\{\pm 1\} ) \times (\C^2/\{\pm 1\} \times \{0\})} |a^*|_{g_0}^2
        \\
        &=
        \int_{\pi^{-1}(B(0,1)\setminus (\{0\}\times \C^2/\{\pm 1\} ) \times (\C^2/\{\pm 1\} \times \{0\}))} \lim_{\epsilon \rightarrow 0} |a_\epsilon|_{g_\epsilon}^2
        \\
        &=
        \int_{\pi^{-1}(B(0,1)\setminus (\{0\}\times \C^2/\{\pm 1\} ) \times (\C^2/\{\pm 1\} \times \{0\}))} \lim _{\epsilon \rightarrow 0} \sup |a_\epsilon|_{g_\epsilon}^2
        \\
        &\geq
        \lim _{\epsilon \rightarrow 0} \sup
        \int_{\pi^{-1}(B(0,1)\setminus (\{0\}\times \C^2/\{\pm 1\} ) \times (\C^2/\{\pm 1\} \times \{0\}))}
        |a_\epsilon|_{g_\epsilon}^2
        \\
        &=
        1,
        \end{split}
    \end{align}
    where we used the Reverse Fatou lemma in the second to last step.
    Thus $a^* \in C^{2,\alpha}((\Lambda^0 \oplus \Lambda^2_7)(\C^2/\{\pm 1\} \times \C^2/\{\pm 1\}))$ with $\Delta a^*=0$, $\|{a^*}_{C^0_\beta}<c$ by the bound \eqref{equation:a-epsilon-global-uniform-bound-short-form}, and $a^* \neq 0$ by the estimate \eqref{equation:fatou-application}.
    But this is a contradiction to \cref{poposition:trivial-connection-rigid-on-c2-times-c2}.
\end{proof}

In \eqref{equation:def-s-beta-t-R4-times-X} we defined a map $s_{j,\beta,t}$ that takes elements in $\Omega^1(M_t, \Ad E_t)$ to one-forms over (an open set of) the model space $\R^4 \times X$ for $j \in \{1, \dots, 8\}$ and $j \in \{73, \dots, 76\}$.
It remains to define the analogue of $s_{j,\beta,t}$ in the case that the the model space is $X \times X$, i.e. for $j \in \{9, \dots, 72\}$.
To this end, fix $j \in \{9, \dots, 72\}$.
Recall that we glue together trivial bundles around the point singularities in our definition of $E_t$, cf. \eqref{equation:def-of-E-t}, so $E_t|_{\tilde{T}_{j,t}}$ is the trivial bundle and $s_{j,\beta,t}$ will produce Lie algebra valued $1$-forms and not $1$-forms taking values in a potentially non-trivial adjoint bundle as in the case of \eqref{equation:def-s-beta-t-R4-times-X}.

Denote the inclusion of 
$\tilde{T}_{i,t}=
\rho^{-1}(B^4_{t^{-1}\zeta} / \{ \pm 1\}) 
\times 
\rho^{-1}(B^4_{t^{-1}\zeta} / \{ \pm 1\})$ in $M_t$ by $\tilde{\iota}_{i,t}$.
Then, for $\beta \in \mathbb{R}$ define
\begin{align}
\label{equation:definition-s-beta-t}
\begin{split}
 s_{j,\beta,t}:
 \Omega^k(M_t, \Ad E_t) &\rightarrow
 \Omega^k(\rho^{-1}(B^4_{t^{-1}\zeta}/ \{ \pm 1\}) \times \rho^{-1}(B^4_{t^{-1}\zeta}/ \{ \pm 1\}), \mathfrak{g})
 \\
 a & \mapsto
 t^{\beta-k}
 (\tilde{\iota}_{i,t})^* a.
 \end{split}
\end{align}
The following is an analogue of \cref{proposition:s-beta-estimates} and we again omit the proof:

\begin{proposition}
\label{proposition:s-beta-estimates-X-times-X}
 There is a constant $c > 0$ such that
 \begin{align*}
  \frac{1}{c}
  \|{a}
  _{C^{k,\alpha}_{\beta,t}(\tilde{T}_{i,t})}
  \leq
  \|{s_{\beta,t}a}
  _{C^{k,\alpha}_{\beta}(\rho^{-1}(B^4_{t^{-1}\zeta}/ \{ \pm 1\}) \times \rho^{-1}(B^4_{t^{-1}\zeta}/ \{ \pm 1\}))}
  \leq
  c
  \|{a}
  _{C^{k,\alpha}_{\beta,t}(\tilde{T}_{i,t})},
  \\
  \|{
   L_t a - s_{\beta-1,t}^{-1} L s_{\beta,t} a
  }
  _{C^{0,\alpha}_{\beta-1,t}(\tilde{T}_{i,t})}
  \leq
  c t^{3/10}
  \|{a}
  _{C^{0,\alpha}_{\beta,t}(\tilde{T}_{i,t})},
  \\
  \|{
   L_t^* \underline{b} - s_{\beta-1,t}^{-1} L^* s_{\beta,t} \underline{b}
  }
  _{C^{0,\alpha}_{\beta-1,t}(\tilde{T}_{i,t})}
  \leq
  c t^{3/10}
  \|{\underline{b}}
  _{C^{0,\alpha}_{\beta,t}(\tilde{T}_{i,t})},
 \end{align*}
 where $L$ denotes the linearisation of the $Spin(7)$-instanton equation at the product connection over the trivial $G$-bundle over $X \times X$.
\end{proposition}

\section{Deforming to genuine solutions}
\label{section:construction}

So far, we have constructed an approximate solution $A_t$ to the $Spin(7)$-instanton equation on $M_t$ in Section \ref{sec:6}, and we have studied the linearisation of the instanton equation on model spaces in Section \ref{sec:linear}.
In this section, we will complete our construction by deforming the approximate solution $A_t$ to a family of genuine $Spin(7)$-instantons on $M_t$ for sufficiently small $t$.

To implement this, we prove an estimate for the linearisation on $M_t$ in Section \ref{subsection:linear-problem-on-Mt}. We then use this estimate to prove the existence of the genuine $Spin(7)$-instantons in Section \ref{subsection:perturbation-step}.

\subsection{Linear problem on $M_t$}
\label{subsection:linear-problem-on-Mt}

We recall the following Schauder estimate:

\begin{proposition}[Proposition 2.86 in \cite{Walp3}]
\label{proposition:schauder-estimate-spin7-linearisation}
 For $\beta \in \mathbb{R}$ there is a constant $c > 0$ such that
 \[
  \|{
   a
  }_{C^{1,\alpha}_{\beta,t}}
  \leq
  c
  \left(
   \|{
    L_{A_t} a
   }_{C^{0,\alpha}_{\beta-1,t}}
   +
   \|{
    a
   }_{L^\infty_{\beta,t}}
  \right)
 \]
 for all $a \in \Omega^1(M_t, \Ad E_t)$.
\end{proposition}

Using this, we prove the key ingredient of our existence theorem, namely, the following uniform estimate for the linearisation of the instanton operator $L_{A_t}$.
Apart from Case 2.B., the proof is a straightforward adaptation of \cite[Proposition 2.83]{Walp3}.

\begin{proposition}
\label{proposition:estimate-for-inverse-of-spin7-linearisation}
\label{proposition:proposition:estimate-for-inverse-of-spin7-linearisation-augmented} % because augmented section was deleted, use double label here
 Assume that the flat connection $\theta$ on $T^8/\Gamma$ and the $Spin(7)$-intantons $A_j$'s over $X$ from \cref{definition:compatible-gluing-data} are infinitesimally rigid.
 Then, given $\beta \in (-2,0)$, there exists a constant $c > 0$ such that for small $t$ we have
 \begin{align}
  \|{
   a
  }_{C^{1,\alpha}_{\beta,t}}
  \leq
  c
   \|{
    L_{A_t} a
   }_{C^{0,\alpha}_{\beta-1,t}}
 \end{align}
 for all $a \in \Omega^1(M_t, \Ad E_t)$.
\end{proposition}

\begin{proof}
 Suppose for a contradiction that the claim is not true.
 We then have a sequence $\{ t_i \} \subset \R$ with $t_i \rightarrow 0$ and an element $a_i \in \Omega^1(M_{t_i}, E_{t_i})$ for each $i \in \mathbb{N}$ so that
 \begin{align}
 \label{equation:L-a-i-to-zero}
  \|{ a_i }_{L^\infty_{\beta,t_i}} \equiv 1
  \text{ and }
  \|{ L_{t_i} a_i}_{C^{0,\alpha}_{\beta-1,t_i}} \rightarrow 0.
 \end{align}
 By \cref{proposition:schauder-estimate-spin7-linearisation} this implies
 \begin{align}
 \label{equation:contradiction-proof-C1alpha-bounded-sequence}
  \|{
   a_i
  }_{C^{1,\alpha}_{\beta,t_i}}
  \leq
  2c.
 \end{align}
 Let $x_i \in M_{t_i}$ such that 
 \begin{align}
 \label{equation:a-i-nonzero-condition}
  w_{t_i}(x_i)^{-\beta} | a_i(x_i) | = 1.
 \end{align}
 Without loss of generality, we can assume that we are in one of the following three cases, and we will arrive at a contradiction in each case.
 
 \textbf{Case 1.}
 Assume that the sequence $\{ x_i \}$ accumulates on the regular part of $T^8/\Gamma$, namely, suppose $\lim r_{t_i}(x_i) > 0$ as $i \rightarrow \infty$.
 
 Let $K \subset ( T^8/\Gamma ) \setminus \fix(\Gamma)$ be a compact subset.
 Using the projection $\pi: M_{t_i} \rightarrow T^8 / \Gamma$ from Section \ref{subsection:approximate-spin7-structures}, we can consider $\pi ^{-1}(K) \subset M_{t_i}$ which is a compact set for all $i \in \mathbb{N}$.
 The restriction $\pi |_{\pi ^{-1}(K)}$ is a diffeomorphism that is not exactly an isometry, but $\pi^* g_t$ converges to the flat metric $g_0$ on $T^8/\Gamma$ in $C^2$ by the definition of $g_t$.
 
 For the bundles, $E_0|_K$ is naturally identified with a subset of $E_t|_{\pi ^{-1}(K)}$ by the construction of $E_t$.
 Under this identification, $A_t|_{\pi ^{-1}(K)} \rightarrow \theta|_K$ as $i \rightarrow \infty$.
 
 Therefore, we obtain from \cref{equation:contradiction-proof-C1alpha-bounded-sequence} that the sequence $a_i|_K$ is uniformly bounded in $C^{1,\alpha}$, measured with respect to the flat metric $g_0$ on $T^8/\Gamma$ and with respect to the flat connection $\theta$ on $E_0$.
 By the Arzela--Ascoli theorem, we can extract a subsequence of $a_i|_K$ converging in $C^{1, \alpha/2}(\Ad E_0|_K)$.
 Exhausting the set $( T^8 /\Gamma ) \setminus \fix(\Gamma)$ with compact sets $(K_i)$, we may use a diagonal sequence argument to extract a limit $a^* \in \Omega^1( (T^8 /\Gamma) \setminus \fix(\Gamma), \Ad E_0)$.
 This limit satisfies
 \begin{align}
 \label{equation:a-star-bound}
  |a^*| \leq
  c \, d_{\Omega_0}(\cdot, \fix ( \Gamma))^\beta
  \text{ and }
  L_\theta a^* =0. 
 \end{align}
 Furthermore, by passing to a subsequence, we find a limit $x^* = \lim_{i \rightarrow \infty} x_i $, viewed as a point in $T^8 /\Gamma$.
 By the assumption of this case, $x_i$ converges away from the singular set, namely, we have $x^* \in (T^8 /\Gamma ) \setminus \fix(\Gamma)$.
 Thus, from \cref{equation:a-i-nonzero-condition}, we obtain 
 \begin{align*}
  a^* 
  \left(
   x^*
  \right)
  >
  0.
 \end{align*}
 On the unit ball $B \subset \R^4$, we have $\int_B |x|^{\beta} \d x < \infty$ for $\beta>-2$ by changing to sphere coordinates.
 Because the singular set of $T^8/\Gamma$ has codimension four, this implies that $d_{\Omega_0}(\cdot, \fix ( \Gamma))^\beta$ is in $L^2(T^8/\Gamma)$.
 Thus, by \cref{equation:a-star-bound}, the limit $a^*$ is in $L^2$.
 Therefore, it is a well-defined distribution acting on $L^2$-sections of $\Lambda^1_{T^8/\Gamma} \otimes \Ad E_0 $.
 Thus, $a^*$ is a solution to the equation $L_\theta \, a^*=0$ in the distributional sense on all of $T^8/\Gamma$.
 Here, a section over the orbifold $T^8/\Gamma$ means a section over $T^8$ that is $\Gamma$-invariant.
 Therefore, by elliptic regularity (cf. \cite[Theorem 6.33]{Foll1}), we have that $a^* \in \Omega^1(T^8/\Gamma, \Ad E_0)$ is a smooth, non-zero element in the kernel of $L_\theta$.
 This is a contradiction to the assumption that $\theta$ is infinitesimally rigid.

 \textbf{Case 2.}
Assume that the sequence $\{ x_i \}$ accumulates on one of the ALE spaces, namely, suppose $\lim r_{t_i}(x_i)/t_i < \infty$ as $i \rightarrow \infty$.
 
 Without loss of generality, we can assume that $x_i \in \tilde{T}_{j,t_i}$ for one fixed $j \in I$ (cf. Section \ref{section:joyces-construction-of-spin7-mf}) and all $i \in \mathbb{N}$.
 Define $\tilde{a}_i := s_{j,\beta,t_i} a_i$ and choose $\tilde{x}_i \in \iota_{j,t}^{-1}(x_i)$.
 The $\tilde{a}_i$ are then $1$-forms over (an open subset of) one of the model spaces $\R^4 \times X$ or $X \times X$, 
 the map $s_{j,\beta,t_i}$ is defined in \eqref{equation:definition-s-beta-t} for $j \in \{1, \dots, 8, 73, \dots, 76\}$, and it is defined in \eqref{equation:def-s-beta-t-R4-times-X} for $j \in \{9, \dots, 72\}$.
 By \cref{proposition:s-beta-estimates,proposition:s-beta-estimates-X-times-X}, and the bound \eqref{equation:contradiction-proof-C1alpha-bounded-sequence} we have 
 \begin{align}
 \label{equation:ALE-case-a-i-tilde-estimate}
  \|{
   \tilde{a}_i
  }_{C^{1,\alpha}_\beta}
  \leq
  c.
 \end{align}
 We have \eqref{equation:a-i-nonzero-condition} for the glued metric on $M_{t_i}$.
 This implies
 \begin{align}
 \label{equation:ALE-case-a-i-tilde-nonzero-estimate}
  w(\tilde{x}_i)
  |\tilde{a}_i(\tilde{x}_i)|
  >
  \frac{1}{2}
 \end{align}
 for the product metric on the model space.
 Because the two metrics are close, but not necessarily equal, \eqref{equation:a-i-nonzero-condition} becomes a lower bound.
 
 \textbf{Case 2.A.}
 Suppose $j \in \{1, \dots, 8, 73, \dots, 76\}$. 
 Without loss of generality, we can assume that $\tilde{x}^*:=\lim_{i \rightarrow \infty} \tilde{x}_i$ exists.
 This is because 
 the $X$-component of $\tilde{x}_i \in \R^4 \times X$ is bounded by the assumption that $\lim r_{t_i}(x_i)/t_i < \infty$, thus, it has a convergent subsequence.
 The $\R^4$-component need not be bounded.
 However, because all operators and norms involved are translation invariant, we can translate $\tilde{a}_i$ in the $\R^4$-direction without affecting our estimates, and can therefore assume the $\R^4$-component of $\tilde{x}_i$ to be zero. 
 
 By the second point of \eqref{equation:L-a-i-to-zero} and the second estimate in \cref{proposition:s-beta-estimates,proposition:s-beta-estimates-X-times-X}, we have 
 \begin{align}
 \label{equation:ALE-case-L-a-i-tilde-goes-to-zero}
  \|{
   L_{p_2^*A_j} \tilde{a}_i
  }_{C^{0,\alpha}_{\beta-1}}
  \rightarrow 0.
 \end{align}
 Using a diagonal argument, we obtain a limit $a^* \in \Omega^1(\R^4 \times X, p_2^*\Ad E_j)$ satisfying
 \begin{align*}
  \|{a^*}_{C^{1,\alpha/2}_{\beta}} \leq
  c, \, 
  %\text{ and }
  L_{p_2^*A_j} a^* =0, 
  \text{ and }
  a^* 
  \left(
   \tilde{x}^*
  \right)
  >
  0, 
 \end{align*}
 where we used the bound \eqref{equation:ALE-case-a-i-tilde-estimate} for the first inequality,
 \eqref{equation:ALE-case-L-a-i-tilde-goes-to-zero} for the equality,
 and the inequality \eqref{equation:ALE-case-a-i-tilde-nonzero-estimate} for the last one.
 By \cref{lemma:kernel-of-pullback-instanton}, $a^*$ is the pullback of an element in the $L^2$-kernel of $\delta_{A_j}$ over $X$.
 Since $A_j$ is chosen to be rigid, this contradicts the above $a^* \left(  \tilde{x}^* \right) >0$. 
 
 \textbf{Case 2.B. }
 Suppose $j \in \{9, \dots, 72\}$.
 The sequence $\tilde{x}_i$ has a converging subsequence because of the assumption $\lim r_{t_i}(x_i)/t_i < \infty$.
 Thus we can assume without loss of generality that the limit $\tilde{x}^*:=\lim_{i \rightarrow \infty} \tilde{x}_i$ exists.
 Analogously to Case 2.A we obtain an element $a^* \in \Omega^1(X \times X, \mathfrak{g})$ satisfying
 \begin{align}
 \label{equation:linearisation-case-2-b-limit-properties}
  \|{a^*}_{C^{1,\alpha/2}_{\beta}} \leq
  c, \,
 % \text{ and }
  L a^* =0, 
  \text{ and }
  a^* 
  \left(
   \tilde{x}^*
  \right)
  >
  0, 
 \end{align}
 where $L$ is the linearisation of the $Spin(7)$-instanton equation at the product connection.
 This is a contradiction to \cref{proposition:trivial-connectionon-X-X-unobstructed-rigid}.
 
 \textbf{Case 3.}
Assume that the sequence $\{ x_i \}$ accumulates on one of the necks, namely, suppose $\lim r_{t_i}(x_i) = 0$ and $\lim r_{t_i}(x_i)/t_i = \infty$.
 
 Again, without loss of generality, we can assume that $x_i \in \tilde{T}_{j,t_i}$ for one fixed $j \in I$ and all $i \in \mathbb{N}$. 
 We define $\tilde{a}_i := s_{j,\beta,t_i} a_i$ and choose $\tilde{x}_i \in \iota_{j,t}^{-1}(x_i)$.
 In this case, we cannot assume that $\lim_{i \rightarrow \infty} \tilde{x}_i$ exists, because $\tilde{x}_i$ tends to infinity.
 We will therefore rescale it first.

 Fix a sequence $R_i \rightarrow \infty$ such that $R_i/(r_{t_i}(x_i)/t_i) \rightarrow 0$.
  
 \textbf{Case 3.A. }
 Suppose $j \in \{1, \dots, 8, 73, \dots, 76\}$.
 In this case, $\tilde{a}_i$ is defined over 
 $\R^4 \times \rho^{-1}(B^4_{t^{-1}\zeta})
 \subset
 \R^4 \times X$.
 Define the restriction
 \[
  \tilde{\tilde{a}}_i
  :=
  \tilde{a}_i|
  _{\R^4 \times \left(\rho^{-1} \left( B^4_{t^{-1}\zeta} \setminus B^4_{R_i} \right) \right)}.
 \]
 Since $(\id, \rho): \R^4 \times \left(\rho^{-1} \left( B^4_{t^{-1}\zeta} \setminus B^4_{R_i} \right) \right) \rightarrow \R^4 \times \left( B^4_{t^{-1}\zeta} \setminus B^4_{R_i} \right)$ is a diffeomorphism, we can view $\tilde{\tilde{a}}_i$ as a tensor on a subset of $\R^4 \times \C^2/\{ \pm 1\}$ rather than a tensor on a subset of $\R^4 \times X$.
 By pulling back under the quotient map $\R^4 \times \C^2 \rightarrow \R^4 \times \C^2/\{ \pm 1\}$, we obtain a tensor on
 $\R^4 \times \left( B^4_{t^{-1}\zeta} \setminus B^4_{R_i} \right)
 \subset
 \R^4 \times \C^2$, that we again denote by the same symbol, namely, $\tilde{\tilde{a}}_i$.
 We now define the rescaling by 
 \[
  \tilde{\tilde{\tilde{a}}}_i
  :=
  (\cdot r_{t_i}(x_i)/t_i)^*\tilde{\tilde{a}}_i
  \cdot
  (r_{t_i}(x_i)/t_i)^{-1-\beta}.  
 \]
 This is defined on $\R^4 \times \left( B^4_{\zeta/r_{t_i}(x_i)} \setminus B^4_{R_i/(r_{t_i}(x_i)/t_i)} \right)$ and satisfies
 \begin{align}
 \label{equation:linearisation-case-3-a-rescaling-analysis}
 \begin{split}
  |\tilde{\tilde{\tilde{a}}}_i \cdot r^{-\beta}|_{g_{\R^4} \oplus g_{\C^2}}
  &=
  |(\cdot r_{t_i}(x_i)/t_i)^*\tilde{\tilde{a}}_i
  \cdot
  (r_{t_i}(x_i)/t_i)^{-1-\beta}
  \cdot r^{-\beta}|_{g_{\R^4} \oplus g_{\C^2}}
  \\
  &=
  \left|
  (\cdot r_{t_i}(x_i)/t_i)^*
  \left[
  \tilde{\tilde{a}}_i
  \cdot
  r^{-\beta}
  \right]
  (r_{t_i}(x_i)/t_i)^{-1}
  \right|_{g_{\R^4} \oplus g_{\C^2}}
  \\
  &=
  \left|
  \tilde{\tilde{a}}_i
  \cdot
  r^{-\beta}
  \right|_{g_{\R^4} \oplus g_{\C^2}},
 \end{split}
 \end{align}
 where the last step follows from the behaviour of covariant tensors under rescaling.
 Because the sets $\R^4 \times \left( B^4_{\zeta/r_{t_i}(x_i)} \setminus B^4_{R_i/(r_{t_i}(x_i)/t_i)} \right)$ exhaust $\R^4 \times (\C^2 \setminus \{0\})$, we can again use a diagonal sequence argument to extract a limit $a^*$ on $\R^4 \times (\C^2 \setminus \{0\})$ of the sequence. 
 As in Case 2, we can assume without loss of generality that the sequence $\tilde{x}_i/(r_{t_i}(x_i)/t_i)$ converges to a point $x^* \in \R^4 \times \C^2$.
 Taking \eqref{equation:L-a-i-to-zero}, \eqref{equation:contradiction-proof-C1alpha-bounded-sequence}, and \eqref{equation:linearisation-case-3-a-rescaling-analysis} together, we obtain
 \begin{align}
 \label{equation:linearisation-3-a-limit-properties}
  |a^*| \leq
  c |\pi_2|^\beta, \,
%  \text{ and }
  L_{A_j^\infty} a^* =0, 
  \text{ and }
  a^* 
  \left(
   x^*
  \right)
  |\pi_2|(x^*)^{-\beta}
  >
  0.
 \end{align}
 Because $\beta> -3$, this $a^*$ defines a distribution on all of $\R^4 \times \C^2$, and it is smooth by elliptic regularity.
 We now check that $a^*$ is uniformly bounded. 
 This follows from the first estimate in \eqref{equation:linearisation-3-a-limit-properties} on $\R^4 \times (\C^2 \setminus \{ 0 \})$ as follows:
 for $z \in \R^4$, the first estimate in \eqref{equation:linearisation-3-a-limit-properties} gives an estimate for $a^*$ on the ball $B_2(z,0) \subset \R^4 \times \C^2$ in an $L^p$-norm for some $p \in (1, \infty)$ independent of $z$.
 By elliptic regularity, this gives $L^p_k$-estimates independent for all $k \in \mathbb{N}$, still independent of $z$.
 By the Sobolev embedding theorems, we therefore obtain an $L^\infty$-estimate for $a^*$ on $B_2(z,0)$.
 This estimate is independent of $z$, and therefore holds on all of $\R^4 \times B_2(0) \subset \R^4 \times \C^2$.
 Together with the estimate for $a^*$ on $\R^4 \times (\C^2 \setminus \{ 0 \})$  we find that $a^*$ is uniformly bounded.

 It follows from \eqref{equation:LL*-L*L-on-R4-times-X} and the second point of  \eqref{equation:linearisation-3-a-limit-properties} that $L^*_{A_j^\infty}L_{A_j^\infty}a^*=(\Delta_{\R^4} + \Delta_{\C^2})a^*=0$.
 By \cref{lemma:constant-in-R-n-direction} this implies that $a^*$ is constant in the $\R^4$-direction.
 Thus, restricted to a single fibre $\{z\} \times \C^2$ we have that $a^*$ consists of several harmonic real-valued functions that are bounded and decay at infinity.
 Then, by Liouville's theorem, we obtain $a^*=0$. However, this contradicts the third point of \eqref{equation:linearisation-3-a-limit-properties}.
 
 \textbf{Case 3.B. }
 Suppose $j \in \{9, \dots, 72\}$.
 In this case, $\tilde{a}_i$ is defined over 
 $\rho^{-1}(B^4_{t^{-1}\zeta}) \times \rho^{-1}(B^4_{t^{-1}\zeta})
 \subset
 X \times X$.
 Similar to Case 3.A,  we can extract a limit $a^*$ over $(\C^2 \setminus \{0\}) \times (\C^2 \setminus \{0\})$ which gives a contradiction to \cref{poposition:trivial-connection-rigid-on-c2-times-c2}.
\end{proof}

We also have the following:

\begin{proposition}
\label{proposition:glued-connection-unobstructed}
\label{proposition:glued-connection-unobstructed-augmented} % because the augmented section was deleted, use this double label here
 Assume that the flat connection $\theta$ on $T^8/\Gamma$ and the ASD intantons $A_i$ over $X$ from \cref{definition:compatible-gluing-data} are infinitesimally rigid.  Suppose furthermore that $\theta$ is irreducible and unobstructed.
Then, given $\beta \in (-2,0)$, there exists a constant $c > 0$ such that for small $t$ we have
 \begin{align}
  \|{
   \underline{b}
  }_{C^{1,\alpha}_{\beta,t}}
  \leq
  c
   \|{
    L_{A_t}^* (\underline{b})
   }_{C^{0,\alpha}_{\beta-1,t}}
 \end{align}
 for all $\underline{b} \in \Omega^0(M_t, \Ad E_t) \oplus \Omega^2_7(M_t, \Ad E_t)$.
\end{proposition}

\begin{proof}
The proof is analogous to the proof of \cref{proposition:estimate-for-inverse-of-spin7-linearisation}.
First, we have the following analogue of \cref{proposition:schauder-estimate-spin7-linearisation} that can be proved in the same way:
\begin{align}
\label{equation:linearisation-adjoint-schauder-estimate}
  \|{
   \underline{b}
  }_{C^{1,\alpha}_{\beta,t}}
  \leq
  c
  \left(
   \|{
    L_{A_t}^* \underline{b}
   }_{C^{0,\alpha}_{\beta-1,t}}
   +
   \|{
    \underline{b}
   }_{L^\infty_{\beta,t}}
  \right).
\end{align}
Now assume that the proposition does not hold.
As in the proof of \cref{proposition:estimate-for-inverse-of-spin7-linearisation}, we obtain sequences $t_i \rightarrow 0$, $x_i \in M_{t_i}$ and a sequence $\underline{b}_i \in \Omega^0(M_{t_i}, \Ad E_{t_i}) \oplus \Omega^2_7(M_{t_i}, \Ad E_{t_i})$ such that
\begin{align}
 \label{equation:L-b-i-to-zero}
  \|{
   \underline{b}_i
  }_{C^{1,\alpha}_{\beta,t_i}}
  \leq
  2c, \, 
  %\text{ and }
  \|{ L_{A_{t_i}}^* \underline{b}_i}_{C^{0,\alpha}_{\beta-1,t_i}} \rightarrow 0, 
  \text{ and }
  w_{t_i}(x_i)^{-\beta} | \underline{b}_i(x_i) | = 1.
\end{align}
Again, we consider the following three cases and arrive at a contradiction in each case:

\textbf{Case 1.}
Assume that the sequence $\{ x_i \}$ accumulates on the regular part of $T^8/\Gamma$, namely, suppose $\lim r_{t_i}(x_i) > 0$.

Using a diagonal argument as in Case 1 in the proof of \cref{proposition:estimate-for-inverse-of-spin7-linearisation}, we can extract a limit
$\underline{b}^* \in  \Omega^0(T^8/\Gamma, \Ad E_0) \oplus \Omega^2_7(T^8/\Gamma, \Ad E_0)$ satisfying
\begin{align}
 L^*_\theta(\underline{b}^*)=0
 \text{ and }
 \underline{b}^* \neq 0,
\end{align}
which contradicts the fact that $\theta$ is irreducible and unobstructed.

\textbf{Case 2.}
 Assume that the sequence $\{ x_i \}$ accumulates on one of the ALE spaces, namely, suppose $\lim r_{t_i}(x_i)/t_i < \infty$.
 
 Without loss of generality, we can assume that $x_i \in \tilde{T}_{j,t_i}$ for one fixed $j \in I$ (cf. Section \ref{section:joyces-construction-of-spin7-mf}) and all $i \in \mathbb{N}$.
 
\textbf{Case 2.A. }
Suppose $j \in \{1, \dots, 8, 73, \dots, 76\}$. 
As in Case 2.A in the proof of \cref{proposition:estimate-for-inverse-of-spin7-linearisation}, we can extract a limit
$\underline{b}^* \in  \Omega^0(\R^4 \times X, p_2^* \Ad E_j) \oplus \Omega^2_7(\R^4 \times X, p_2^* \Ad E_j)$ satisfying
\begin{align}
\label{equation:A_t-unobstructed-proof-case-2A-limit}
\|{
 \underline{b}^*
}_{C^{1,\alpha/2}_\beta}
\leq
c, \, 
%\text{ and }
L^*_{p_2^* A_j} \underline{b}^*=0, 
\text{ and }
\underline{b}^* \neq 0.
\end{align}
By \cref{lemma:constant-in-R-n-direction}, $\underline{b}^*$ is constant in the $\R^4$-direction, namely, the pullback of an element in $\Omega^0 \oplus \Omega^2_7(X, \Ad E_j)$ that we denote by the same symbol.
By \eqref{equation:LL*-L*L-on-R4-times-X}, $\underline{b}^*$ can be written as the sum of an element in the kernel of $\delta_{A_j}$ and an element in the kernel of $\delta_{A_j}^*$.
By the second point of \cref{proposition:ale-asd-kernel-cokernel}, the element in the kernel of $\delta_{A_j}^*$ is zero.
On the other hand, we assume that $A_j$ is infinitesimally rigid, therefore the component in the kernel of $\delta_{A_j}$ is also zero.
This contradicts $\underline{b}^* \neq 0$ from \eqref{equation:A_t-unobstructed-proof-case-2A-limit}.

\textbf{Case 2.B. }
Suppose $j \in \{9, \dots, 72\}$.
As in Case 2.B in the proof of \cref{proposition:estimate-for-inverse-of-spin7-linearisation}, we can extract a limit
$\underline{b}^* \in  \Omega^0(X \times X, \mathfrak{g}) \oplus \Omega^2_7(X \times X, \mathfrak{g})$ satisfying
\begin{align}
\label{equation:A_t-unobstructed-proof-case-2B-limit}
\|{
 \underline{b}^*
}_{C^{1,\alpha/2}_\beta}
\leq
c, \, 
%\text{ and }
\d \xi + \d^* b=0 , 
\text{ and }
\underline{b}^* \neq 0.
\end{align}
This is a contradiction to \cref{proposition:trivial-connectionon-X-X-unobstructed-rigid}.

\textbf{Case 3.}
Assume the sequence $(x_i)$ accumulates on one of the necks, namely, suppose $\lim r_{t_i}(x_i) = 0$ and $\lim r_{t_i}(x_i)/t_i = \infty$.

\textbf{Case 3.A.} 
As in Case 3.A. in the proof of \cref{proposition:estimate-for-inverse-of-spin7-linearisation}, we can extract a limit
$\underline{b}^*=(\xi,b) \in  \Omega^0(\R^4 \times (\R^4 \setminus \{0\}), \mathfrak{g}) \oplus \Omega^2_7(\R^4 \times (\R^4 \setminus \{0\}), \mathfrak{g})$ satisfying
\begin{align}
 \label{equation:linearisation-3-a-limit-properties2}
  |a^*| \leq
  c |\pi_2|^\beta, \,
%  \text{ and }
  L_{A_j^\infty}^* a^* =0, 
  \text{ and }
  a^* 
  \left(
   x^*
  \right)
  |\pi_2|(x^*)^{-\beta}
  >
  0.
 \end{align}
 Using \eqref{equation:LL*-L*L-on-R4-times-X} and Lemma \ref{lemma:constant-in-R-n-direction} we conclude like in the proof of \cref{proposition:estimate-for-inverse-of-spin7-linearisation} that $a^*$ is the pullback of a harmonic function with decay on $\R^4$, which contradicts the last point of \eqref{equation:linearisation-3-a-limit-properties2}.

\textbf{Case 3.B.}
As in Case 3.B. in the proof of \cref{proposition:estimate-for-inverse-of-spin7-linearisation}, we can extract a limit
$\underline{b}^*=(\xi,b) \in  \Omega^0((\R^4 \setminus \{0\}) \times (\R^4 \setminus \{0\}), \mathfrak{g}) \oplus \Omega^2_7((\R^4 \setminus \{0\}) \times (\R^4 \setminus \{0\}), \mathfrak{g})$ satisfying
\begin{align}
\label{equation:A_t-unobstructed-proof-case-3-limit}
\|{
 \underline{b}^*
}_{C^{1,\alpha/2}_\beta}
\leq
c, \, 
%\text{ and }
\d \xi + \d^* b=0, \, 
\text{ and }
\underline{b}^* \neq 0.
\end{align}
This contradicts \cref{poposition:trivial-connection-rigid-on-c2-times-c2}.
\end{proof}

\begin{remark}
\label{remark:reducible-theta}
We assumed in \cref{proposition:glued-connection-unobstructed} that the ambient connection $\theta$ is irreducible.
This assumption can be removed as in \cite[pp.68--69]{Walp3}, though one can no longer guarantee that the resulting $Spin(7)$-instanton is irreducible, and we will not use this generalisation for our examples in Section \ref{section:examples}.
\end{remark}

\subsection{The perturbation step}
\label{subsection:perturbation-step}

With all of this in place, we can prove the desired gluing result:

\begin{theorem}
 \label{theorem:main-gluing-result}
 Let $\theta$ be an unobstructed flat connection over $T^8/\Gamma$ and $A_j$ infinitesimally rigid ASD-instantons over Eguchi--Hanson space $X$ for $j \in \{73,74,75,76\}$ that are compatible in the sense of \cref{definition:compatible-gluing-data}.
 Let $A_t$ be the connection on the bundle $E_t$ over $M_t$ constructed in Section \ref{sec:6}.
 Then there exists $c>0$ such that for small $t$ there exists $a_t \in \Omega^1(M_t, \Ad E_t)$ such that $A_t+a_t$ is a $Spin(7)$-instanton on $E_t$.
 Moreover,
 $\|{a_t}_{C^{1,\alpha}_{-1,t}} \leq c t^{3/10}$ and $A_t+a_t$ is smooth.
 \label{th:analytic_construction}
\end{theorem}

\begin{proof}
 We aim to find a solution $a \in C^{1,\alpha}_{-1;t}(\Lambda^1(M, \Ad E_t))$ of the equation
 \begin{align}
 \label{equation:instanton-equation-for-deformation}
  \pi^2_7(F_{A_t+a})=0.
 \end{align}
 This equation decomposes into three summands.
 If we denote the term that is quadratic in $a$ by $Q_t$, i.e. $Q_t(a):= \pi^2_7([a,a])$, and denote the term that is constant in $a$ (that is the pre-gluing error) by $e_t$, i.e. $e_t:= \pi^2_7(F_{A_t})$, then it will suffice to solve the equation:
 \begin{align}
 \label{equation:instanton-equation-solved-for-linear-part}
  \tilde{L}_{A_t}(\underline{a}) = (0, -e_t - Q_t(a))
 \end{align}
 for $\underline{a}=(a, \xi) \in \Omega^1(M, \Ad E_t) \oplus \mathcal{H}^0_{A_t}$.
 In this case, $A_t+a$ will turn out to be a solution to \cref{equation:instanton-equation-for-deformation}.
 
 \Cref{equation:instanton-equation-solved-for-linear-part} is a \emph{non-linear} equation in $\underline{a}=(a, \xi)$.
 In order to solve it, we will repeatedly solve its linearisation and a limit of these solutions will be a solution to \cref{equation:instanton-equation-solved-for-linear-part}.
 
 To this end, define $\underline{a}_0 := 0$ and for $k > 0$ let $\underline{a}_k = (a_k, \xi_k) \in \Omega^1(M, \Ad E_t) \oplus (1-\chi_t) \cdot \mathcal{H}^0_\theta$ be such that 
 \begin{align}
 \label{equation:iterative-instanton-equation}
  \tilde{L}_{A_t}(\underline{a}_k) = (0, -e_t - Q_t(a_{k-1})),
 \end{align}
 where $\underline{a}_k$ exists because $\tilde{L}_{A_t}$ is surjective by \cref{proposition:glued-connection-unobstructed-augmented}.

 For the moment, denote by $c'$ the constant from \cref{proposition:proposition:estimate-for-inverse-of-spin7-linearisation-augmented}, namely, 
 $\|{
   \underline{a}
  }_{C^{1,\alpha}_{\beta,t}}
  \leq
  c'
   \|{
    \tilde{L}_{A_t} \underline{a}
   }_{C^{0,\alpha}_{\beta-1,t}}$,
 and denote by $c''$ the constant from \cref{proposition:pregluing-estimate}, that is, 
 $\|{ e_t }_{C^{0,\alpha}_{-2;t}} \leq c''t^{3/10}$.
 
 \textbf{Claim:}
 if $t$ is small, then $\|{ \underline{a}_k }_{C^{1,\alpha}_{-1;t}} \leq 2c'c'' t^{3/10}$ for all $k \geq 0$.
 
 %\textbf{Proof of claim:}
 \begin{proof} We prove the claim by induction.
 By definition, it is true for $k = 0$.
 
 If $k > 0$, then:
 \begin{align*}
  \|{ \underline{a}_k }_{C^{1,\alpha}_{-1;t}} &\leq
  c'
  \|{ \tilde{L}_{A_t} \underline{a}_k }_{C^{0,\alpha}_{-2;t}}
  \\
  &=
  c' 
  \|{ e_t + Q_t(\underline{a}_{k-1}) }_{C^{0,\alpha}_{-2;t}}
  \\
  &\leq
  c'
  \|{ e_t }_{C^{0,\alpha}_{-2;t}}
  +
  c'
  \|{ Q_t(\underline{a}_{k-1}) }_{C^{0,\alpha}_{-2;t}}
  \\
  &\leq
  c' c'' t^{3/10}
  +
  c'
  \|{ \underline{a}_{k-1} }_{C^{0,\alpha}_{-1;t}}^2
  \\
  &\leq
  2c' c'' t^{3/10}.
 \end{align*}
 In the first step, we used \cref{proposition:glued-connection-unobstructed-augmented}, in the last step we used the induction hypothesis together with the fact that $t$ is sufficiently small, in particular, it is small enough so that $c'(2c'c'' t^{3/10})^2 \leq c' c''t^{3/10}$.
 This proves the intermediate claim.
 \end{proof}

 Therefore, $\underline{a}_k$ is a bounded sequence, thus, it has a convergent subsequence by the Arzela--Ascoli theorem.
 Denote its limit by $\underline{a}=(a, \xi) \in \Omega^1(M, \Ad E_t) \oplus (1-\chi_t) \cdot \mathcal{H}^0_\theta$.
 By taking the limit of \cref{equation:iterative-instanton-equation}, we see that $\underline{a}$ satisfies \cref{equation:instanton-equation-solved-for-linear-part}.
 Thus, 
 \begin{align*}
  e_t+\left( L_{A_t}(a) \right) +Q_t(a)
  &=
  0, 
 \end{align*}
 and therefore $A_t+a$ is a $Spin(7)$-instanton.
 It remains to show smoothness of $A_t+a$.
 As $A_t$ is smooth by definition, we have to show that $a$ is smooth.
 This is done by \emph{elliptic bootstrapping}:
 assume that $a \in C^{k,\alpha}$ for some $k \geq 1$.
 Then, by general elliptic theory there exists a constant $M>0$ (depending on $t$ and $k$) such that
 \begin{align*}
    \|{
        a
    }_{C^{k+1,\alpha}}
    &\leq
    M
    \|{
        L_{A_t} a
    }_{C^{k,\alpha}}
    +
    \|{
        a
    }_{C^0}
    \\
    &=
    M
    \|{
        e_t+Q_t(a)
    }_{C^{k,\alpha}}
    +
    \|{
        a
    }_{C^0}.
 \end{align*}
 Because $a \in C^{k,\alpha}$ we have that the right-hand side is finite, which implies that the left-hand side is finite, so $a \in C^{k+1,\alpha}$.
 Induction shows that $a \in C^\infty$.
\end{proof}

\section{Examples}
\label{section:examples}

In Section \ref{subsection:approximate-solutions}, we constructed approximate $Spin(7)$-instantons on $M_t$, using compatible gluing data.
Roughly speaking, compatible gluing data consist of a flat connection $\theta$ on the orbifold $T^8/\Gamma$ and of rigid ASD instantons on Eguchi--Hanson space $X$, satisfying some compatibility conditions.
In Section \ref{section:construction}, we proved that, given such an approximate solution, one obtains a nearby exact solution to the $Spin(7)$-instanton equation.

We have not yet shown that there exist examples of compatible gluing data.
In this section, we will exhibit many examples of these data, leading to lots of new examples of $Spin(7)$-instantons.
It will turn out that given a rigid flat connection on the orbifold $T^8/\Gamma$, there is an easy way to extend it to compatible gluing data.
Therefore, we review flat connections on the orbifold $T^8/\Gamma$ and state how to produce compatible gluing data in the sense of  \cref{definition:compatible-gluing-data} out of suitable flat connections on the orbifold $T^8/\Gamma$ in Section \ref{subsection:flat-connections-on-orbifolds}.
In \Cref{subsection:SO3-examples}, we present explicit examples with structure group $SO(3)$, in \Cref{subsection:other-structure-groups}, we comment on the situation for all the other cases among $SO(n)$'s, providing examples with structure groups $SO(4)$, $SO(5)$, $SO(7)$ and $SO(8)$ and explaining why other cases are not possible

\subsection{Flat connections on orbifolds}
\label{subsection:flat-connections-on-orbifolds}

The following is an adaptation of \cite[Proposition 9.2]{Walp1} to the $Spin(7)$-setting:

\begin{lemma}
\label{lemma:rep-theory-criterion-for-rigidity}
 A flat connection $\theta$ on a $G$-bundle $E_0$ over a flat $Spin(7)$-orbifold $T^8/\Gamma$ corresponding to a representation $\rho: \pi_1^{\rm{orb}}(T^8/\Gamma) \rightarrow G$ is irreducible if and only if the induced representation of $\pi_1^{\rm{orb}}(T^8/\Gamma)$ on the Lie algebra $\mathfrak{g}$ of $G$ has no non-zero fixed vectors, where the action is given by the adjoint representation. Similarly, it is infinitesimally rigid if and only if the induced representation of $\pi_1^{\rm{orb}}(T^8/\Gamma)$ on $\R^8 \tensor \mathfrak{g}$ has no non-zero fixed vectors; and it is unobstructed if and only if the induced representation of $\pi_1^{\rm{orb}}(T^8/\Gamma)$ on $\Lambda^2_7(\R^8) \tensor \mathfrak{g}$ has no non-zero fixed vectors, where $\pi_1^{\rm{orb}}(T^8/\Gamma)$ acts on $\R^8$ and $\Lambda^2_7(\R^8)$ via the holonomy representation of the Levi-Civita connection. 
\end{lemma}

\begin{proof}
    By the Weitzenb\"{o}ck formula \cite[Equations 1.74, 1.75]{Papo}, we have 
	\[
		L^*_\theta L_\theta=\nabla^* \nabla,
		\;
		L_\theta L^*_\theta=\nabla^* \nabla. 
	\]
	Thus, elements in the kernel of $L_\theta$ and elements in the kernel of $L^*_\theta$ are parallel sections of $\Lambda^1(T^8/\Gamma)  \otimes \Ad E_0 $ and $\Lambda^0(T^8/\Gamma ) \otimes \Ad E_0  \oplus \Lambda^2_7(T^8/\Gamma)  \otimes \Ad E_0$, respectively.
	Hence, the claim then follows from the holonomy principle.
\end{proof}

As discussed in Section \ref{section:construction}, the crucial ingredients for our construction are unobstructed flat connections on a $G$-bundle over the $Spin(7)$-orbifold. Flat connections are in bijective correspondence with representations of the orbifold fundamental group on $G$, and \cref{lemma:rep-theory-criterion-for-rigidity} provides a simple criterion to verify whether a given flat connection is irreducible, infinitesimally rigid, or unobstructed.

Now recall that an \emph{orbifold} is locally modeled on $\mathbb{R}^n$ quotiented by finite group actions. In Section \ref{section:joyces-construction-of-spin7-mf}, we consider the action of $\Gamma$ on $T^8$, where $\Gamma$ is given in eq. (\ref{equation:alpha-beta-gamma-delta-def}). Since the action of $\Gamma$ is properly discontinuous, we obtain that $T^8/\Gamma$ is an orbifold (see for example Proposition 13.2.1 in \cite{Thurston}). Furthermore, we may consider the universal cover $\R^8$ of $T^8$. Recall that the \emph{fundamental group} $\pi_1^{\rm{orb}}$ is the group of deck transformations of the universal cover. Therefore, the orbifold fundamental group of $T^8/\Gamma$ is generated by the group of involutions $\Gamma=\lbrace\alpha,\beta,\gamma,\delta\rbrace$ together with the translations along the torus directions, which we denote by $\lbrace\tau_i\rbrace_{i=1,\dots,8}$. 

\begin{lemma}
The generators $\alpha,\beta,\gamma,\delta, \tau_1, \dots , \tau_8$ of $\pi_1^{\rm{orb}} (T^8 / \Gamma )$ satisfy the following relations: 
\begin{enumerate}
    \item[$1.$] $[\tau_i,\tau_j]=1$ for all $i,j\in\lbrace1,\dots,8\rbrace$. That is, the $\tau_i$ commute among themselves.\label{it:condition1}
    
    \item[$2.$] $\alpha^2=\beta^2=\gamma^2=\delta^2=1$. That is, $\lbrace\alpha,\beta,\gamma,\delta\rbrace$ are involutions. \label{it:condition2}
    
    \item[$3.$] $[\alpha,\beta]=1$, $[\alpha,\gamma]=\tau_2^{-1}\tau_1^{-1}$, $[\alpha,\delta]=\tau_3^{-1}$, $[\beta,\gamma]=\tau_6^{-1}\tau_5^{-1}$, $[\beta,\delta]=\tau_7^{-1}$ and $[\gamma,\delta]=\tau_1\,$. This shows $\lbrace\alpha,\beta,\gamma,\delta\rbrace$ do not commute among themselves. \label{it:condition3}
    
    \item[$4.$] $\alpha\tau_i=\tau_i^{-1}\alpha$ for $i\in\lbrace1,2,3,4\rbrace$, $\alpha\tau_i=\tau_i\alpha$ for $i\in\lbrace5,6,7,8\rbrace$. \label{it:condition4}
    
    \item[$5.$] $\beta\tau_i=\tau_i^{-1}\beta$ for $i\in\lbrace5,6,7,8\rbrace$, $\beta\tau_i=\tau_i\beta$ for $i\in\lbrace1,2,3,4\rbrace$. \label{it:condition5}
    
    \item[$6.$] $\gamma\tau_i=\tau_i^{-1}\gamma$ for $i\in\lbrace1,2,5,6\rbrace$, $\gamma\tau_i=\tau_i\gamma$ for $i\in\lbrace3,4,7,8\rbrace$. \label{it:condition6}
    
    \item[$7.$] $\delta\tau_i=\tau_i^{-1}\delta$ for $i\in\lbrace1,3,5,7\rbrace$, $\delta\tau_i=\tau_i\delta$ for $i\in\lbrace2,4,6,8\rbrace$. \label{it:condition7}
\end{enumerate}
\label{lemma:conditions}
\end{lemma}

The above relations in Lemma \ref{lemma:conditions} are obtained from the action of the group elements on $\mathbb{R}^8$, for example from
\begin{align*}
    \tau_1: & \left(x_1,x_2,x_3,x_4,x_5,x_6,x_7,x_8\right)\longmapsto\left(x_1+1,x_2,x_3,x_4,x_5,x_6,x_7,x_8\right) \, , \\
    \alpha: & \left(x_1,x_2,x_3,x_4,x_5,x_6,x_7,x_8\right)\longmapsto\left(-x_1,-x_2,-x_3,-x_4,x_5,x_6,x_7,x_8\right) \, ,
\end{align*}
it is immediate to verify that $\alpha\tau_1=\tau_1^{-1}\alpha$. We omit the proof.  (See also Ma \cite{Ma} for the corresponding $G_2$-orbifold case.)

Lemma \ref{lemma:conditions} gives conditions that representations of the orbifold fundamental group must satisfy.

\begin{remark}
\label{remark:all-are-involutions}
    Because of \cref{definition:compatible-gluing-data}, we are interested in representations for which $\lbrace\alpha,\beta\rbrace$ have a trivial representative. From the conditions 4 and 5 in Lemma \ref{lemma:conditions}, we immediately see that this implies that the representatives of all the $\lbrace\tau_i\rbrace_{i=1,\dots,8}$ must be involutions, so we will only be interested in subgroups of involutions of the structure group $G$. In fact, from the condition 1 in Lemma \ref{lemma:conditions} the representatives $\lbrace\rho(\tau_i)\rbrace_{i=1,\dots,8}$ generate a subgroup of commuting involutions, and the conditions 6 and 7 in Lemma \ref{lemma:conditions} imply they must also commute with $\rho(\gamma)$ and $\rho(\delta)$.

    Furthermore, the condition 3 in Lemma \ref{lemma:conditions} for $\rho(\alpha)=\rho(\beta)=1$ imposes additional constraints on the involutions representing the translations
    \begin{equation}
    \label{eq:conditions-alpha-beta-trivial}
        \rho(\tau_3)=\rho(\tau_7)=1, \qquad \rho(\tau_1)=\rho(\tau_2), \qquad \rho(\tau_5)=\rho(\tau_6) ,
    \end{equation}
    whereas there are no restrictions on $\rho(\tau_4)$ or $\rho(\tau_8)$.

    Finally, let us observe that if the representatives of $\gamma$ and $\delta$ also commute, the condition 3 in Lemma \ref{lemma:conditions} further implies that
    \begin{equation}
    \label{eq:conditions-trivial-alpha-delta-commuting}
        \rho(\tau_1)=\rho(\tau_2)=\rho(\tau_3)=\rho(\tau_7)=1, \qquad \rho(\tau_5)=\rho(\tau_6) ,
    \end{equation}
    with still no additional restrictions on $\rho(\tau_4)$ or $\rho(\tau_8)$. In this case, the image of the orbifold group generates a subgroup of commuting involutions of $G$ which is completely determined by the images of $\gamma,\delta,\tau_4,\tau_5,\tau_8$. The conditions 1 to 7 in Lemma \ref{lemma:conditions} are automatically satisfied as long as we impose the condition \eqref{eq:conditions-trivial-alpha-delta-commuting}. All the solutions we are going to present are of this type, as we will show in \Cref{proposition:no-so(n)-non-commutative-instantons} that without the commutativity assumption the associated connection is obstructed.
\end{remark}

\begin{remark}
\label{remark:for-us-H1-is-zero}
    When $\rho(\alpha)=\rho(\beta)=1$, the adjoint representation of $\alpha$ and $\beta$ acts trivially on $\mathfrak{g}$ whereas the holonomy representation does not fix any nonzero elements of $\mathbb{R}^8$. Therefore, from \cref{lemma:rep-theory-criterion-for-rigidity} we find that the corresponding flat connection is infinitesimally rigid.
\end{remark}

Now we prove a short lemma and proposition that will be useful to show that certain connections are obstructed.

\begin{lemma}
\label{lemma:holonomy-action-on-Lambda2}
    The holonomy action of $\alpha$ and $\beta$ fixes a 3-dimensional subspace $V$ of $\Lambda^2_7(\R^8)$. Within $V$, there is a 1-dimensional subspace $V_{-,-}$ where $\gamma$ and $\delta$ act both as $-1$, and two 1-dimensional subspaces $V_{\pm,\mp}$ where $\gamma$ acts as $\pm 1$ and $\delta$ acts as $\mp 1$, respecively. 
\end{lemma}

\begin{proof}
    Consider the following basis of $\Lambda^2_7(\R^8)$:
    \begin{align*}
        e_1=d x_{12}+d x_{34}+d x_{56}+d x_{78}, \qquad e_2=d x_{13}-d x_{24}+d x_{57}-d x_{68}, \\
        e_3=d x_{14}+d x_{23}+d x_{58}+d x_{67}, \qquad e_4=d x_{15}-d x_{26}-d x_{37}+d x_{48}, \\
        e_5=d x_{16}+d x_{25}+d x_{38}+d x_{47}, \qquad e_6=d x_{17}-d x_{28}+d x_{35}-d x_{46}, \\
        e_7=d x_{18}+d x_{27}-d x_{36}-d x_{45}.
    \end{align*}
    It is immediate to check that $\alpha$ and $\beta$ act as the identity on $V=\langle e_1,e_2,e_3\rangle$ and as minus the identity on the subspace $\langle e_4,e_5,e_6,e_7\rangle$. Now, from the holonomy action of $\gamma$ and $\delta$ we find that $V_{-,-}=\langle e_3\rangle$, $V_{+,-}=\langle e_1\rangle$ and $V_{-,+}=\langle e_2\rangle$
\end{proof}

\begin{proposition}
\label{proposition:showing-connections-are-obstructed}
    Let $\rho: \pi_1^{\rm{orb}}(T^8/\Gamma) \rightarrow G$ be a group homomorphism from the orbifold group to a gauge group $G$ such that $\rho(\alpha)=\rho(\beta)=1$. Suppose there exists an element $g\in\mathfrak{g}$ fixed by $\lbrace\rho(\tau_i)\rbrace_{i=1,\dots,8}$ and such that at least one of $\rho(\gamma)$, $\rho(\delta)$ acts on $g$ with a global minus sign, with the other element acting as $\pm 1$. Then, the associated connection is obstructed.
\end{proposition}

\begin{proof}
    By \cref{lemma:rep-theory-criterion-for-rigidity}, it is enough to show there exists an element of $\Lambda^2_7(\R^8) \tensor \mathfrak{g}$ fixed by the action of the orbifold group. To this end, we will use the subspaces $V,V_{-,-},V_{\pm,\mp}\subset\Lambda^2_7(\R^8)$ we introduced in \cref{lemma:holonomy-action-on-Lambda2}. First note the subspace $V\tensor\langle g \rangle$ is fixed by $\rho(\alpha)$, $\rho(\beta)$ and $\lbrace\rho(\tau_i)\rbrace_{i=1,\dots,8}$, so it is enough to find an element in this subspace that is also fixed by $\rho(\gamma)$ and $\rho(\delta)$.

    If both $\rho(\gamma)$ and $\rho(\delta)$ act with a minus sign on $g$, then the subspace $V_{-,-}\otimes \langle g\rangle$ is fixed by the action of $\rho(\gamma)$ and $\rho(\delta)$. On the other hand, if $\rho(\gamma)$ acts as $-1$ and $\rho(\delta)$ acts as $+1$ on $g$, the subspace $V_{-,+}\otimes \langle g\rangle$ is then fixed by the action of both elements. Finally, if $\rho(\gamma)$ acts trivially and $\rho(\delta)$ acts with a minus sign on $g$ then $V_{+,-}\otimes \langle g\rangle$ is a subspace fixed by $\rho(\gamma)$ and $\rho(\delta)$. Therefore, in all cases the associated connection is obstructed.
\end{proof}

In the rest of this section, we describe how to obtain compatible gluing data in the sense of \cref{definition:compatible-gluing-data} from flat connections. 

We denote by $[M]$ the real rank $2 \mu$ bundle underlying a complex vector bundle $M$ of rank $\mu$, and use the notation $s V := \underbrace{V \oplus \dots \oplus V}_{s \text{ times}}$ for any vector bundle $V$.

% The following lemma is a generalisation of the $SO(3)$ case from of \cite[p.133]{DoKr}:

\begin{lemma}
\label{lemma:reducible-so(n)-connections}
    Let $L$ be a complex vector bundle of rank one with connection $A$ over a manifold $X$.
    Consider $E := m[L^t] \oplus \underline{\R^k}$ for $m,k \in \mathbb{Z}_{\geq 0}, \, t \in \mathbb{Z}$.
    Then there exist $s_0, s_1, s_2 \in \mathbb{Z}_{\geq 0}$ such that
    \[
        \Ad E
        =
        s_0 \underline{\mathbb{R}} \oplus s_1 [L^t] \oplus s_2 [L^{2t}]
    \]
    and the connection on $\Ad E$ induced by $A$ restricts to each of the summands.
\end{lemma}

\begin{proof}
We first consider the complexification $E_{\mathbb{C}} = E \otimes \mathbb{C}$, i.e.
$$  E_{\mathbb{C}} \simeq m (L^t \oplus L^{-t} ) \oplus \underline{\C^k}. $$
For the $\mathfrak{so}(n,\mathbb{C})$-endomorphisms of $E_{\C}$,  we have $
\mathfrak{so}_{\mathbb{C}}(E_{\C})
        =
        \Lambda^2(E_{\C}) $. Since $ \Lambda^2( V \oplus W) = \Lambda^2(W) \oplus ( V \otimes W ) \oplus \Lambda^2(V)$ for any vector bundles $V$ and $W$ over $X$, we obtain 
\begin{align*}
      \Lambda^2(E_{\C})
        &=
        \Lambda^2( m (L^t \oplus L^{-t} ) \oplus \underline{\C^k} )
        \\
        &= \Lambda^2( \underline{\C^k} ) \oplus 
        ( m (L^t \oplus L^{-t} ) \otimes \underline{\C^k} ) 
        \oplus  \Lambda^2( m (L^t \oplus L^{-t} ) ) 
        \\
        &= \Lambda^2( \underline{\C^k} ) \oplus 
        ( m k (L^t \oplus L^{-t} ) ) 
        \oplus  \Lambda^2( m  L^{-t} ) \oplus m^2 \underline{\C} \oplus  \Lambda^2( m  L^{t} ). 
    \end{align*}
Here, $ \Lambda^2( m  L^{t} ) = (m-1) L^{2t}  \oplus \Lambda^2 (m-1)  L^{t} =\cdots = \frac{m (m-1)}{2} L^{2t} $. Similarly, 
$ \Lambda^2( m  L^{-t} ) = \frac{m (m-1)}{2} L^{- 2t} $, $\Lambda^2( \underline{\C^k} ) = \frac{k(k-1)}{2} \underline{\C} $. 
Hence, 
\[ 
\Lambda^2(E_{\C}) 
= 
\left( \frac{k(k-1)}{2} + m^2 \right) \underline{\C} \oplus 
        ( m k (L^t \oplus L^{-t} ) ) 
        \oplus \frac{m (m-1)}{2} ( L^{2t} \oplus L^{-2t} ) .\]
Since $ (\Ad E )_{\C}
        =
        \mathfrak{so}_{\mathbb{C}}(E_{\C}) =  \Lambda^2(E_{\C})$, by taking the real form of it, we obtain 
        \[ 
        \Ad E
        =
        s_0 \underline{\mathbb{R}} \oplus s_1 [L^t] \oplus s_2 [L^{2t}], \]
where $s_0 := \left( \frac{k(k-1)}{2} + m^2 \right) , s_1 := mk, s_2 := \frac{m (m-1)}{2} $. 
\end{proof}

\begin{proposition}
\label{proposition:existence-of-ALE-side-connection}
    Choose $n>1$ and let $\rho: \pi_1^{\rm{orb}}(T^8/\Gamma) \rightarrow SO(n)$ be a representation so that $\rho(\alpha)=\rho(\beta)=1$.
    Let $E_0$ be the corresponding $SO(n)$-bundle with flat connection $\theta$.
    Then, for each choice of four integers $\underline{\xi}=(\xi_{73}, \xi_{74}, \xi_{75}, \xi_{76})$, the connection $\theta$ can be extended to compatible gluing data in the sense of \cref{definition:compatible-gluing-data}.
    Furthermore, two different choices of $\underline{\xi}$ lead to different compatible gluing data, so that when the construction from \cref{theorem:main-gluing-result} is applied, the different $Spin(7)$-instantons are not gauge-equivalent.
\end{proposition}

\begin{proof}
    For $j \in \{73, 74\}$ we can canonically identify $\pi_1(T_j \setminus S_j) \cong \langle \gamma, \tau_3, \tau_4, \tau_7, \tau_8 \rangle$.
    The group element $\gamma \in \Gamma$ has order $2$, so, $\hol_\theta(\gamma)=: g_j \in SO(n)$ is the identity or has order $2$.

    Denote by $(L, A)$ the $U(1)$-bundle with ASD instanton on the Eguchi--Hanson space $X$ from \cref{proposition:eguchi-hanson-ASD-example}.
    Choose $m,k \geq 0$ with $2m+k=n$ such that $m[L] \oplus \underline{\R^k}$ has $\hol(\infty)=g_j$ for the connection induced by $A$.
    If $g_j \neq 1$, that means that $m \neq 0$, and we define $E_j := m[L^{\otimes (2\xi_j+1)}] \oplus \underline{\R^k}$.
    If $g_j = 1$, then $m=0$ and therefore $k \geq 2$, so we can define
    $E_j := [L^{\otimes(2 \xi_j)}] \oplus \underline{\R^{k-2}}$.
    Either way, if we let $A_j$ denote the connection on $E_j$ induced by $A$, it still has $\hol_{A_j}(\infty)=g_j$.
    
    We claim that ${A}_j$ is infinitesimally rigid.
    By \cref{lemma:reducible-so(n)-connections}, we have $\Ad E_j = s_0 \underline{\mathbb{R}} \oplus s_1 [L^t] \oplus s_2 [L^{2t}]$ with $s_0, s_1, s_2, t \in \mathbb{Z}_{\geq 0}$ and the induced connection reduces to each of the summands.
    We must check that the kernel of $\delta_{A_j}$ restricted to each of the summands of $\Ad (E_j)$ is zero.

    We first prove that $\Ker \delta_{A_j}|_{\Omega^1(\underline{\R})}=0$.
    This is equivalent to checking $\Ker \tilde{\delta}_{A_j}=0$ for the rescaled operator $\tilde{\delta}_{A_j}= \tilde{\delta} := (\d^*, \sqrt{2} \d^+)$.
    Let $a \in \Ker \tilde{\delta}$ which decays at infinity.
    By \cite[Proposition 2.41]{Walp3} we have that $a$ and its derivatives are in $L^2$.
    The operator $\tilde{\delta}$ satisfies the Weitzenb\"ock formula \cite[Equation 6.25]{FrUh}.
    The connection induced by $A_j$ on $\underline{\R}$ is the trivial connection, thus, the curvature term in the Weitzenb\"ock formula vanishes.
    Furthermore, Eguchi--Hanson space is Ricci-flat, thus,  $\tilde{\delta}^* \tilde{\delta}=\nabla^* \nabla$.
    Integration by parts, we obtain $\nabla a=0$.
    Because $a$ is assumed to decay at infinity, this implies $a=0$.
    Thus, $\Ker \delta_{A_j}|_{\Omega^1(\underline{\R})}=0$.

    It remains to show that $\Ker \delta_{A_j}|_{\Omega^1(L^t)}=0$ for any integer $t \neq 0$.
    The argument is similar to the above:
    we again apply the Weitzenb\"ock formula \cite[Equation 6.25]{FrUh}.
    We note that the curvature term in the Weitzenb\"ock formula contains a commutator.
    Since $L^t$ is a line bundle, this commutator is in an abelian Lie algebra, hence, the curvature term vanishes.  Then the argument is the same as in the case of $\Omega^1(\underline{\R})$ above.
    Namely, we obtain $\Ker \delta_{A_j}|_{\Omega^1(L^t)}=0$ as in the previous paragraph.

    Taking everything together, we obtain  $\Ker \delta_{A_j}=0$.

    The data from (ii) in \cref{definition:compatible-gluing-data} can be chosen arbitrarily.
    The construction for $j \in \{75, 76\}$ and $\hol_{\theta}(\delta) =: g_j$ is analogous.
    
    Different choices of $\underline{\xi}$ lead to different Pontrjagin classes by \cref{remark:pontrjagin-classes}, so lead to $Spin(7)$-instantons that are not gauge-equivalent.
\end{proof}

\begin{remark}
    \Cref{proposition:existence-of-ALE-side-connection} is not stated in the greatest generality possible.
    For $SO(n)$ with $n>3$ one can combine different tensor powers of $L$ on a single glued-in piece.
    For example, where \cref{proposition:existence-of-ALE-side-connection} constructs the rank $4$ vector bundle $[L^{\otimes (2 \xi)}] \oplus \underline{\R^2}$, one might as well use $[L^{\otimes (2 \xi)}] \oplus [L^{\otimes (2 \zeta)}]$ instead, for $\zeta \geq 0$.
\end{remark}

For illustration purposes, we briefly present how our construction can be used to recover some known instantons with structure group $G=SO(2)=U(1)$.

\begin{proposition}
\label{proposition:no-SO2-instantons}
    Consider a representation $\rho: \pi_1^{\rm{orb}}(T^8/\Gamma) \rightarrow SO(2)$. The associated connection is infinitesimally rigid and unobstructed, but it is always reducible.
\end{proposition}

\begin{proof}
    Since $SO(2)$ is abelian, \Cref{lemma:conditions} implies that the image of $\rho$ is a subgroup of commuting involutions of $SO(2)$ and it is therefore contained in $\lbrace\pm1\rbrace$. This shows the adjoint representation of $\pi_1^{\rm{orb}}(T^8/\Gamma)$ acts trivially on $\mathfrak{so}(2)$ and by \Cref{lemma:rep-theory-criterion-for-rigidity} the associated connection is reducible. On the other hand, the holonomy representation of $\alpha,\beta,\gamma,\delta$ leaves no non-zero elements of $\mathbb{R}^8$ and $\Lambda^2_7(\mathbb{R}^8)$ fixed and again by \Cref{lemma:rep-theory-criterion-for-rigidity} we conclude that the associated connection is infinitesimally rigid and unobstructed.
\end{proof}

\begin{remark}
    Setting $\rho(\alpha)=\rho(\beta)=1$, by \Cref{remark:all-are-involutions} we see a representation is completely determined by choosing $\gamma,\delta,\tau_4,\tau_5,\tau_8\in\lbrace\pm1\rbrace$. Since at least one of $\rho(\gamma)$ and $\rho(\delta)$ must be non-trivial, we obtain $(2^2-1)2^3=24$ non-equivalent flat connections.
    They can be extended to compatible gluing data by \cref{proposition:existence-of-ALE-side-connection}.
    \Cref{theorem:main-gluing-result} guarantees the existence of $Spin(7)$-instantons with structure group $SO(2)$ even though $\theta$ is reducible, thanks to \cref{remark:reducible-theta}.
    A priori, the resulting $Spin(7)$-instanton $\tilde{A}_t:=A_t+a_t$ could be reducible, which would imply that it is flat.
    There are different ways to see this cannot occur:
    one is that by \cref{remark:pontrjagin-classes} we have that $p_1(E_t) \neq 0$ if at least one of $\rho(\gamma)$ and $\rho(\delta)$ is non-trivial, so the Chern-Weil representative of $p_1(E_t)$ is non-zero, i.e. $\tilde{A}_t$ has non-zero curvature.
    Note that the existence of $\tilde{A}_t$ was already known on abstract grounds, because for each harmonic $\omega \in \Omega^2(M)$ there exists a line bundle with connection $\nabla$, such that $F_\nabla=\omega$.
    By \cite[Proposition 10.6.5]{Joyc5}, $0=\pi^2_7(\omega)=\pi^2_7(F_\nabla)$, and $\nabla$ is a $Spin(7)$-instanton.
    However, the precise description of $\nabla$ in terms of our gluing theorem is new.
\end{remark}

\subsection{Examples with structure group SO(3)}
\label{subsection:SO3-examples}

In this section, we describe how to construct irreducible, infinitesimally rigid, and unobstructed flat connections with structure group $G={SO}(3)$. We consider a representation $\rho: \pi_1^{\rm{orb}}(T^8/\Gamma) \rightarrow {SO}(3)$ with $\rho(\alpha)=\rho(\beta)=1$ and such that the representatives of $\gamma$ and $\delta$ commute, with at least one of them being non-trivial. By \cref{remark:all-are-involutions}, the representation is well-defined as long as the condition \eqref{eq:conditions-trivial-alpha-delta-commuting} is satisfied. As a result, $\rho$ is completely determined by the images of $\gamma,\delta,\tau_4,\tau_5,\tau_8$, which generate a subgroup of commuting involutions of $SO(3)$.

Recall involutions of $SO(3)$ are given by rotations of angle $\pi$, and two of them commute if and only if they are performed either along the same axis or along perpendicular axes. Therefore, the maximal number of commuting involutions in $SO(3)$ is three, and the image of the representation is a subgroup of $\mathbb{Z}_2\times\mathbb{Z}_2$. By performing a change of basis---which corresponds to a gauge transformation of the connection---we can assume the subgroup consists of rotations around the standard axes, which is precisely the Klein four-group $K=\langle a,b,c \rangle$, where:
\begin{equation*}
a=\begin{pmatrix}
1 & 0 & 0 \\
0 & -1 & 0 \\
0 & 0 & -1
\end{pmatrix}, \quad
b=\begin{pmatrix}
-1 & 0 & 0 \\
0 & 1 & 0 \\
0 & 0 & -1
\end{pmatrix}, \quad
c=\begin{pmatrix}
-1 & 0 & 0 \\
0 & -1 & 0 \\
0 & 0 & 1
\end{pmatrix}.
\end{equation*}
Note each element of $\mathfrak{so}(3)$ is the infinitesimal generator of rotations on a certain plane and it is therefore only fixed by the adjoint action of rotations on the same plane. The adjoint action of a rotation of angle $\pi$ around an axis which is perpendicular to that of the infinitesimal rotation results in a global change of sign. 
Taking the generators $u_1,u_2,u_3 \in \mathfrak{so}(3)$ defined by
\begin{equation}
\label{eq:basis-of-so3-lie-algebra}
    u_1 = \begin{pmatrix}  0 & 0 & 0 \\ 
0 & 0 & -1 \\ 
0 & 1 & 0 \\ \end{pmatrix} , 
 u_2 = \begin{pmatrix}  0 & 0 & 1 \\ 
0 & 0 & 0 \\ 
-1 & 0 & 0 \\ \end{pmatrix} , 
 u_3 = \begin{pmatrix}  0 & -1 & 0 \\ 
1 & 0 & 0 \\ 
0 & 0 & 0 \\ \end{pmatrix},
\end{equation} 
one easily checks
\begin{align*}
%\label{equation:klein-action-on-so3}
\begin{split}
a (u_1) = u_1 , \quad b (u_1) = - u_1 , \quad c (u_1) = - u_1 , \\
a (u_2) = - u_2 , \quad b (u_2) = u_2 , \quad c (u_2) = - u_2 ,\\
a (u_3) = - u_3 , \quad b (u_3 ) = - u_3 , \quad c (u_3) = u_3   .
\end{split}
\end{align*}

Now we perform a detailed study of which connections are irreducible and unobstructed.
\begin{lemma}
\label{lemma:so3-irreducible}
    The connection associated with the representation $\rho$ is irreducible if and only if there exist two elements within $\gamma,\delta,\tau_4,\tau_5,\tau_8$ with different non-trivial representatives.
\end{lemma}
\begin{proof}
    By \cref{lemma:rep-theory-criterion-for-rigidity}, we need to study which elements of $\mathfrak{g}=\mathfrak{so}(3)$ are fixed by the adjoint representation of $\rho$.
    
    If the image of $\gamma,\delta,\tau_4,\tau_5,\tau_8$ consists of a single rotation, then the elements of $\mathfrak{so}(3)$ infinitesimally generating that same rotation will be fixed by the adjoint representation and the connection will not be irreducible.
    
    Conversely, if the image of $\gamma,\delta,\tau_4,\tau_5,\tau_8$ includes rotations around two different axes then no non-trivial elements of $\mathfrak{so}(3)$ can be fixed by the adjoint action and the connection will be irreducible.
\end{proof}

\begin{lemma}
\label{lemma:so3-unobstructed}
    The connection associated with the representation $\rho$ is unobstructed if and only if one of the following is true:
    \begin{itemize}
        \item Two elements within $\tau_4,\tau_5,\tau_8$ have different non-trivial representatives.
        \item The image of $\tau_4,\tau_5,\tau_8$ has a single non-trivial representative and agrees with the image of $\gamma,\delta$.
    \end{itemize}
\end{lemma}
\begin{proof}
    Suppose first that there exist at least two elements within $\tau_4,\tau_5,\tau_8$ with different non-trivial representatives. The holonomy representation of $\tau_i$ acts trivially on $\Lambda^2_7(\R^8)$ for all $i=1,\dots,8$, so by \cref{lemma:rep-theory-criterion-for-rigidity} it is enough to show that no non-trivial elements of $\mathfrak{so}(3)$ are fixed by the adjoint action of all $\rho(\tau_i)$. This is immediate arguing analogously to the proof of \cref{lemma:so3-irreducible}.

    Suppose now that $\tau_4,\tau_5,\tau_8$ have precisely one non-trivial representative. In this case, there exists one element $g\in\mathfrak{so}(3)$ such that the action of $\lbrace\tau_i\rbrace_{i=1,\dots,8}$ on $\Lambda^2_7(\R^8)\otimes\mathfrak{so}(3)$ fixes precisely $\Lambda^2_7(\R^8)\otimes \langle g\rangle$. One then has to check whether $\alpha,\beta,\gamma,\delta$ fix any of these elements.
    
    If the image of $\gamma,\delta$ agrees with the image of $\tau_4,\tau_5,\tau_8$, then $\alpha,\beta,\gamma,\delta$ act as the identity on $g$. However, $0$ is the only common fixed vector of $\alpha,\beta,\gamma,\delta$ in $\Lambda^2(\R^8)$, and so in particular in $\Lambda^2_7(\R^8)$.
    Therefore, there are no non-zero elements in $\Lambda^2_7(\R^8)\otimes\mathfrak{so}(3)$ fixed by $\alpha,\beta,\gamma,$ and $\delta$.
    By \cref{lemma:rep-theory-criterion-for-rigidity}, the representation is unobstructed.

    Assume the image of $\gamma,\delta$ does not agree with the image of $\tau_4,\tau_5,\tau_8$. This means that at least one of $\rho(\gamma)$ or $\rho(\delta)$ acts as $-1$ on $g$ via the adjoint action, whereas the other element acts either trivially or as $-1$. By \cref{proposition:showing-connections-are-obstructed}, we conclude that the connection is obstructed.
    
    It only remains to study the case where the image of $\tau_4,\tau_5,\tau_8$ is trivial. Since the image of $\gamma,\delta$ is non-trivial, there exists some $g\in\mathfrak{so}(3)$ for which the adjoint action of either $\rho(\gamma)$ or $\rho(\delta)$ acts with a minus sign. Once again, \cref{proposition:showing-connections-are-obstructed} shows that the connection is obstructed.
\end{proof}

Combining \cref{remark:for-us-H1-is-zero}, \cref{lemma:so3-irreducible} and \cref{lemma:so3-unobstructed} we obtain the following proposition.

\begin{proposition}
    Let $\rho: \pi_1^{\rm{orb}}(T^8/\Gamma) \rightarrow SO(3)$ be a representation of the orbifold fundamental group with image a $\mathbb{Z}_2\times\mathbb{Z}_2$ subgroup of commuting involutions of $SO(3)$, satisfying the condition \eqref{eq:conditions-trivial-alpha-delta-commuting}, with $\rho(\alpha)=\rho(\beta)=1$, at least one of $\rho(\gamma)$ or $\rho(\delta)$ non-trivial, and with $\rho(\tau_4),\rho(\tau_5),\rho(\tau_8)$ generating the whole image of $\rho$. Then, the flat connection associated with this representation is irreducible, infinitesimally rigid, and unobstructed.
\end{proposition}

\begin{remark}
    \label{remark:105-examples}
    By a change of basis of $SO(3)$ we can always consider the Klein four-group $K$ as the image of $\rho$ and gauge transformations map the generators $a,b,c$ to each other. 
    The number of connections is the number of ways of assigning $1,a,b,c$ to $\rho(\gamma),\rho(\delta),\rho(\tau_4),\rho(\tau_5),\rho(\tau_8)$ in such a way that at least two out of $\rho(\tau_4),\rho(\tau_5),\rho(\tau_8)$ are different and not equal to $1$, we do not have $\rho(\gamma)=\rho(\delta)=1$, where then some will be gauge equivalent.
    
    Up to gauge equivalence, there are 105 possibilities to do this.
    Here is how to arrive at this number by counting representations $\rho: \pi_1^{orb}(T^8/\Gamma) \rightarrow K$:
    we get the number of representations satisfying the two criteria by taking the number of all representations minus the number of representations violating one of the criteria, and adding back in the representations which violate more than one criterion, because we have multiply counted them.
    \begin{align*}
        &|\{ \rho \}|
        \\
        -&
        |\{ \rho: \rho(\gamma)=\rho(\delta)=1\}|
        \\
        -&
        |\{ \rho: \rho(\tau_4), \rho(\tau_5),\rho(\tau_8) \in \{1,a\} \}|
        \\
        -&
        |\{ \rho: \rho(\tau_4), \rho(\tau_5),\rho(\tau_8) \in \{1,b\} \}|
        \\
        -&
        |\{ \rho: \rho(\tau_4), \rho(\tau_5),\rho(\tau_8) \in \{1,c\} \}|
        \\
        +&
        |\{ \rho: \rho(\gamma)=\rho(\delta)=1,
        \rho(\tau_4), \rho(\tau_5),\rho(\tau_8) \in \{1,a\} \}|  
        \\
        +&
        |\{ \rho: \rho(\gamma)=\rho(\delta)=1,
        \rho(\tau_4), \rho(\tau_5),\rho(\tau_8) \in \{1,b\} \}|
        \\
        +&
        |\{ \rho: \rho(\gamma)=\rho(\delta)=1,
        \rho(\tau_4), \rho(\tau_5),\rho(\tau_8) \in \{1,c\} \}| 
        \\
        +&
        2 \cdot |\{ \rho: \rho(\tau_4)=\rho(\tau_5)=\rho(\tau_8)=1 \}|
        \\
        -& 2 \cdot |\{ \rho: \rho \equiv 1\}|
        \\
        =&\,
        4^5-4^3-4^2 \cdot 2^3-4^2 \cdot 2^3-4^2 \cdot 2^3+2^3+2^3+2^3+2 \cdot 4^2-2 
        = 630
    \end{align*}
    Before the first equality, we subtracted $2 \cdot |\{ \rho: \rho \equiv 1\}|$ because we counted the trivial representation $1-4+5=2$ times and should have counted it zero times, because it is not an allowed solution.
    Each of these representations takes at least two non-identity values in $K$, so their stabiliser in
    \[
        \text{Aut} (K)
        =
        \{
            f: K \rightarrow K:
            f(1)=1, \text{ $f$ bijective}
        \}
        \cong
        S_3
    \]
    is trivial.
    Thus, these $630$ representations are made up of $105$ orbits of representations under the action of $\text{Aut}(K)$.
    That is, up to gauge equivalence, we have $105$ flat connections to which our construction can be applied.
    An explicit list of $105$ representations which gives rise to flat connections that are not gauge equivalent is given in Appendix \ref{section:list-of-so3-representations-on-t8-gamma}.
    By \cref{proposition:existence-of-ALE-side-connection}, this gives rise to $105$ new four-parameter families of $Spin(7)$-instantons.
\end{remark}

\begin{remark}
    If the representatives of $\gamma$ and $\delta$ are not assumed to commute, further representations of the orbifold group can be obtained. Unfortunately, they can not be used for our construction and we show this in Appendix \ref{appendix:so3noncommutativeconnections}.
\end{remark}

\begin{remark}
    Our construction cannot produce irreducible $Spin(7)$-instantons with structure group $G=SU(2)$.
    Here is why:
    for the resulting connection to be irreducible, we would need either (a) the flat connection $\theta$ to be irreducible, or we would need (b) at least one of the glued-in connections to have full holonomy $SU(2)$ and still be rigid.
    
    Irreducible, flat $SU(2)$-connections on $T^8/\Gamma$ do not exist.
    The conditions $\alpha^2=\beta^2=\gamma^2=\delta^2=1$ force $\rho(\alpha),\rho(\beta),\rho(\gamma),\rho(\delta) \in \{  \pm 1\}$.
    The conditions $[\tau_i, \tau_j]=1$ mean that $\tau_j \in Z(\tau_i)$, and if $\tau_i \neq 1$, then its centralizer satisfies $Z(\tau_i) \simeq U(1) \subset SU(2)$.
    I.e. $\rho$ takes values in $U(1) \subset SU(2)$ and the corresponding connection is reducible.
    (One can even argue that the image of $\rho$ is contained in $\{\pm 1\}$, but this is not needed here.)

    Note however that all these connections are unobstructed. The condition $\rho(\alpha),\rho(\beta),\rho(\gamma),\rho(\delta) \in \{  \pm 1\}$ implies that the adjoint action of these elements is trivial on $\mathfrak{su}(2)$. Since 0 is the only element of $\Lambda^2_7(\mathbb{R}^8)$ fixed by the holonomy action of $\alpha,\beta,\gamma,\delta$, we find there are no non-trivial elements of $\Lambda^2_7(\mathbb{R}^8)\tensor\mathfrak{su}(2)$ fixed by $\rho$ and the associated connection is unobstructed.

    On the other hand, non-flat instantons with structure group $SU(2)$ on Eguchi--Hanson space always come in a positive dimensional moduli space, see \cite[Equation 5.10]{Bian}.
\end{remark}

\subsection{Other cases among SO(n)'s}
\label{subsection:other-structure-groups}

We now comment on other choices of $G=SO(n)$ with $n\geq 4$. Analogously to the case $G={SO}(3)$, we consider a representation $\rho: \pi_1^{\rm{orb}}(T^8/\Gamma) \rightarrow {SO}(n)$ with $\rho(\alpha)=\rho(\beta)=1$ and such that $\rho(\gamma)$ and $\rho(\delta)$ commute---with at least one of them non-trivial. Imposing the condition \eqref{eq:conditions-trivial-alpha-delta-commuting}, we have that $\rho$ is completely determined by the images of $\gamma,\delta,\tau_4,\tau_5,\tau_8$ and we know from \Cref{remark:all-are-involutions} that the image of the representation is a subgroup of commuting involutions of $SO(n)$.

\begin{remark}
\label{remark:involutions-diagonal-matrices}
    The maximal subgroup of commuting involutions of $SO(n)$ is of the form $\mathbb{Z}_2\times\overset{n-1}{\cdots}\times\mathbb{Z}_2$ and consists of all diagonal matrices with an even number of $-1$ entries and $1$ in the remaining entries. To see this, note that a set of $SO(n)$ matrices commutes with itself if and only if they are simultaneously diagonalisable, which means that (in a convenient basis) a subgroup of commuting involutions is represented by diagonal matrices. The involution condition restricts the diagonal entries to $\pm 1$, and the determinant being $+1$ forces the number of $-1$ entries to be even.
\end{remark}

Recall that our method to produce $Spin(7)$-instantons requires the existence of flat connections which are simultaneously irreducible and unobstructed. 

\begin{remark}
    The commutativity assumption for $\rho(\gamma)$ and $\rho(\delta)$ is necessary, otherwise the associated flat connection is obstructed. See \Cref{appendix:non-commutative} for a proof of this fact. 
\end{remark}

Inspecting \Cref{lemma:so3-irreducible} and \Cref{lemma:so3-unobstructed} reveals a necessary condition for the existence of irreducible unobstructed connections, which we now extend to all $G=SO(n)$ with $n\geq 3$:

\begin{proposition}
\label{proposition:no-SOn-instantons}
    Consider a representation $\rho: \pi_1^{\rm{orb}}(T^8/\Gamma) \rightarrow SO(n)$ with $n\geq 3$ such that $\rho(\alpha)=\rho(\beta)=1$, the representatives of $\gamma$ and $\delta$ commute and such that $\lbrace\rho(\tau_i)\rbrace_{i=1,\dots,8}$ leaves a plane invariant. The associated connection cannot be simultaneously irreducible and unobstructed.
\end{proposition}

\begin{proof}
    Consider an element $g\in\mathfrak{so}(n)$  corresponding to an infinitesimal generator of rotations on the plane left invariant by $\lbrace\rho(\tau_i)\rbrace_{i=1,\dots,8}$. By an analogous argument to the $SO(3)$ case, the adjoint action of $\lbrace\rho(\tau_i)\rbrace_{i=1,\dots,8}$ fixes $g$.

    If the adjoint action of both $\gamma$ and $\delta$ is trivial on $g$, then $g\in\mathfrak{so}(n)$ is fixed by the adjoint representation of $\rho$ and by \cref{lemma:rep-theory-criterion-for-rigidity} the associated connection is not irreducible.

    On the other hand, suppose that the associated connection is irreducible. This means that at least one of $\rho(\gamma),\rho(\delta)$ must act non-trivially on $g$ and not preserve the plane. This means the plane contains eigenvectors with eigenvalues $1$ and $-1$ for that involution. It is then immediate to check that the adjoint action of the involution must send $g$ to $-g$, and using \cref{proposition:showing-connections-are-obstructed} we conclude that the connection is obstructed.
\end{proof}

Since \Cref{remark:all-are-involutions} says that $\lbrace\rho(\tau_i)\rbrace_{i=1,\dots,8}$ is generated by $\rho(\tau_4),\rho(\tau_5),\rho(\tau_8)$, we have found that a necessary condition for our construction to apply to $SO(n)$ is the existence of three commuting involutions not leaving a plane fixed.

\begin{proposition}
    Our method cannot be used to produce $Spin(7)$-instantons with structure group $SO(n)$ for $n=6$ or $n\geq 9$.
\end{proposition}

\begin{proof}
    It is immediate to check that, by dimensional reasons, three commuting involutions always leave a plane fixed in these cases.
\end{proof}

The necessary condition in \Cref{proposition:no-SOn-instantons} turns out to be sufficient:

\begin{proposition}
\label{proposition:correct-SOn-instantons}
    Consider a representation $\rho: \pi_1^{\rm{orb}}(T^8/\Gamma) \rightarrow SO(n)$ with $n\geq 3$ such that $\rho(\alpha)=\rho(\beta)=1$ and the representatives of $\gamma$ and $\delta$ commute, with one of them at least non-trivial. The associated flat connection is irreducible, infinitesimally rigid and unobstructed if and only if $\lbrace\rho(\tau_i)\rbrace_{i=1,\dots,8}$ does not leave a plane invariant.
\end{proposition}

\begin{proof}
    The infinitesimally rigid condition was explained in \Cref{remark:for-us-H1-is-zero}. \Cref{proposition:no-SOn-instantons} gives one of the implications, so it only remains to assume that $\lbrace\rho(\tau_i)\rbrace_{i=1,\dots,8}$ does not leave a plane invariant. 

    The elements of $\mathfrak{g}=\mathfrak{so}(n)$ are infinitesimal generators of rotations on planes, and since no planes are fixed by $\lbrace\rho(\tau_i)\rbrace_{i=1,\dots,8}$ its adjoint action also fixes no elements of $\mathfrak{so}(n)$. By \Cref{lemma:rep-theory-criterion-for-rigidity} the associated connection will be irreducible.

    On the other hand, recall the holonomy representation of $\tau_i$ on $\Lambda^2_7(\R^8)$ is trivial for all $i=1,\dots,8$. Since we just argued that the adjoint action does not fix any elements of $\mathfrak{so}(n)$, we conclude that no elements in $\Lambda^2_7(\R^8)\otimes\mathfrak{so}(n)$ are fixed by the induced representation of $\lbrace\tau_i\rbrace_{i=1,\dots,8}$. \Cref{lemma:rep-theory-criterion-for-rigidity} implies the associated connection is unobstructed, concluding the proof.
\end{proof}

We now study the remaining $SO(n)$ cases. Although it is possible to identify the conditions for irreducibility and unobstructedness independently case by case as we did in \Cref{lemma:so3-irreducible} and \Cref{lemma:so3-unobstructed}, we simply rely on \Cref{proposition:correct-SOn-instantons} to identify the connections amenable to our construction.

\paragraph{SO(4).} Any two non-trivial commuting involutions of $SO(4)$ generate a Klein subgroup of $\mathbb{Z}_2\times\mathbb{Z}_2\times\mathbb{Z}_2$. However, there are two types of Klein subgroups in $SO(4)$. The first type is given by subgroups which include the central inversion $-1$. Such a subgroup actually leaves two orthogonal planes invariant. An example of this is the subgroup generated by diagonal matrices with entries $(-1,-1,1,1)$ and $(1,1,-1,-1)$.

The second type of subgroups is characterized by the fact that any two non-trivial involutions act as reflections on planes that intersect at a line, implying that no planes are fixed by such a subgroup. For example, the group generated by $a=(1,-1,-1,1)$ and $b=(-1,1,-1,1)$ is of this type. We thus conclude:

\begin{proposition}
    A representation $\rho: \pi_1^{\rm{orb}}(T^8/\Gamma) \rightarrow SO(4)$ with the assumptions above is valid for our construction if the image of $\lbrace\tau_i\rbrace_{i=1,\dots,8}$ contains a Klein subgroup without the central inversion $-1$.
\end{proposition}

    After taking into account possible gauge equivalences, there are $882$ valid representations. A possible argument goes as follows: we want to count all possible representations up to gauge equivalence, so we give a set of rules to gauge transform any representation to a canonical configuration and we count how many of these configurations we have. Given a representation $\rho$, we choose a basis of $SO(4)$ such that the image of $\rho$ are diagonal involutions.
    
    First we change basis to gauge fix $\rho(\tau_4),\rho(\tau_5),\rho(\tau_8)$, in that order. Our gauge fixing criteria is to impose that the first Klein subgroup not including $-1$ we encounter is generated by $a$ and $b$ as above. In this way, if $\rho(\tau_4)=1$ or $\rho(\tau_4)=-1$ we must change basis so that $\rho(\tau_5)=a$, $\rho(\tau_8)=b$. Otherwise, we change basis so that $\rho(\tau_4)=a$ and we can have that $\rho(\tau_5)$ is equal to $1,-1,a,-a$ (meaning that we must change basis so that $\rho(\tau_8)=b$) or we change basis so that $\rho(\tau_5)=b$. Exploring all possibilities in this way we obtain $14$ different canonical configurations for $\rho(\tau_4),\rho(\tau_5),\rho(\tau_8)$.

    Now that the basis is fixed, so are $\rho(\gamma),\rho(\delta)$ and we just need to count how many possible values they can take. All possible pairs of $(\rho(\gamma),\rho(\delta))$ are allowed except for $(1,1)$, that means there are $8^2-1=63$ different configurations. Combining the number of canonical configurations for $\rho(\tau_4),\rho(\tau_5),\rho(\tau_8)$ and $\rho(\gamma),\rho(\delta)$ we obtain $14\cdot 63=882$ different possible representations which are not gauge equivalent. The explicit list can be found in \Cref{section:list-of-so4-representations-on-t8-gamma}.
    By \cref{proposition:existence-of-ALE-side-connection}, we obtain $882$ additional four-parameter families of $Spin(7)$-instantons.

\paragraph{SO(5).} For $SO(5)$, all Klein subgroups leave planes invariant so we focus on $\mathbb{Z}_2\times\mathbb{Z}_2\times\mathbb{Z}_2$ subgroups of commuting involutions instead. There are two types: the first one is characterized by acting trivially on a 1-dimensional subspace, and correspond to the commuting involutions of a 4-dimensional subspace. They are generated by 3 involutions which are non-trivial along 3 different planes that intersect on the same line. These subgroups preserve no planes. An example is the subgroup generated by diagonal matrices with entries
    \begin{equation*}
        (-1,-1,1,1,1), \quad (-1,1,-1,1,1), \quad (-1,1,1,-1,1).
    \end{equation*}
The second type acts non-trivially on the full 5-dimensional space. They are again generated by 3 involutions which are non-trivial along 3 different planes, and for dimensional reasons one of the planes is orthogonal to the other two. As a result, each of these subgroups preserves a 2-dimensional subspace. An example of such a subgroup is generated by diagonal matrices with entries
    \begin{equation*}
        (-1,-1,1,1,1), \quad (-1,1,-1,1,1), \quad (1,1,1,-1,-1).
    \end{equation*}
\begin{proposition}
    A representation $\rho: \pi_1^{\rm{orb}}(T^8/\Gamma) \rightarrow SO(5)$ with the assumptions above is valid for our construction if the image of $\lbrace\tau_i\rbrace_{i=1,\dots,8}$ is a $\mathbb{Z}_2\times\mathbb{Z}_2\times\mathbb{Z}_2$ subgroup acting trivially on a 1-dimensional subspace.
\end{proposition}

Up to gauge equivalence, there are $1785$ such representations. We count them as in the $SO(4)$ case: the first step is to change basis to gauge fix $\rho(\tau_4),\rho(\tau_5),\rho(\tau_8)$ and the $\mathbb{Z}_2\times\mathbb{Z}_2\times\mathbb{Z}_2$ subgroup they generate. Taking into account that this subgroup can be generated by 3 involutions acting non-trivially on different planes, or by 2 such involutions and another one acting non-trivially on a 4-dimensional subspace, we obtain 7 possible configurations.

Now, each element in the pair $(\rho(\gamma),\rho(\delta))$ can be equal to any of the $16$ elements of $(\mathbb{Z}_2)^4$, with the only restriction of $(1,1)$ not being allowed. This gives $16^2-1=255$ configurations, which combined with the previous ones result in $7\cdot 255=1785$ non-equivalent representations. Thus, we find $1785$ new four-parameter families of $Spin(7)$-instantons thanks to \cref{proposition:existence-of-ALE-side-connection}.

\paragraph{SO(7).} Up to gauge equivalence, there is a single subgroup of $SO(7)$ generated by three commuting involutions that leaves no planes fixed:
\begin{proposition}
    A representation $\rho: \pi_1^{\rm{orb}}(T^8/\Gamma) \rightarrow SO(7)$ with the assumptions above is valid for our construction if $\rho(\tau_4),\rho(\tau_5),\rho(\tau_8)$ are gauge equivalent to the diagonal matrices with entries
    \begin{equation*}
        (-1,-1,-1,-1,1,1,1), \quad (-1,-1,1,1,-1,-1,1), \quad (1,-1,-1,1,1,-1,-1).
    \end{equation*}
\end{proposition}
We can construct $4095$ of these representations, up to gauge equivalence. First, we gauge fix $\rho(\tau_4),\rho(\tau_5),\rho(\tau_8)$ to the form above, which uniquely fixes each coordinate, and we then count $64$ possible choices for each of $\rho(\gamma)$ and $\rho(\delta)$---but we do not allow both to be trivial simultaneously. This gives $64^2-1=4095$ inequivalent connections, and by \Cref{proposition:existence-of-ALE-side-connection} they produce $4095$ new four-parameter families of $Spin(7)$-instantons.

\paragraph{SO(8).} Exactly as in the $SO(7)$ case, there is---up to gauge equivalence---a single subgroup of $SO(8)$ generated by three commuting involutions that leaves no planes fixed:
\begin{proposition}
    A representation $\rho: \pi_1^{\rm{orb}}(T^8/\Gamma) \rightarrow SO(8)$ with the assumptions above is valid for our construction if $\rho(\tau_4),\rho(\tau_5),\rho(\tau_8)$ are gauge equivalent to the diagonal matrices with entries
    \begin{equation*}
        (-1,-1,-1,-1,1,1,1,1), \quad (-1,-1,1,1,-1,-1,1,1), \quad (1,-1,-1,1,1,-1,-1,1).
    \end{equation*}
\end{proposition}
Counting up to gauge equivalence as in the previous cases, we obtain $128^2-1=16383$ possibilities. By \cref{proposition:existence-of-ALE-side-connection} we obtain $16383$ new four-parameter families of $Spin(7)$-instantons.

\appendix

\section{ALE spaces and anti-self-dual instantons}
\label{sec:IALE}

We begin by fixing some notation for hyperkähler Asymptotically Locally Euclidean (ALE) spaces, using definitions from \cite[Section 7.2]{Joyc5}.

\begin{definition}
    \label{definition:ALE}
    Let $G$ be a finite subgroup of $Sp(1)$, and let $(\hat{\omega}_1, \hat{\omega}_2, \hat{\omega}_3, \hat{g})$ be the Euclidean hyperkähler structure and $r: \mathbb{H}/G \rightarrow [0,\infty)$ the ratdius function of $\mathbb{H}/G$.
    We say that a hyperkähler $4$-manifold $(X, \omega_1, \omega_2, \omega_3, g)$ is \emph{Asymptotically Locally Euclidean}, or \emph{ALE}, and asymptotic to $\mathbb{H}/G$, if there exists a compact subset $S \subset X$ and a map $\rho: X \setminus S \rightarrow \mathbb{H}/G$ that is a diffeomorphism between $X \setminus S$ and $\{x \in \mathbb{H}/G: r(x)>R\}$ for some $R>0$ such that
    \begin{align*}
        \hat{\nabla}^k(\rho_* (g)-\hat{g})
        =
        O(r^{-4-k})
        \quad
        \text{ and }
        \quad
        \hat{\nabla} (\pi_*(\omega_j)-\hat{\omega}_j)
        =
        O(r^{-4-k})
    \end{align*}
    as $r \rightarrow \infty$, for $j=1,2,3$ and $k \geq 0$, where $\hat{\nabla}$ is the Levi-Civita connection of $\hat{g}$.
\end{definition}

Hyperkähler ALE spaces were classified by Kronheimer in \cite{Kro1,Kro2}.
We only use the simplest of Hyperkähler ALE spaces, namely, the Eguchi--Hanson space $X$ asymptotic to $\C^2/\{\pi 1\}$.
We need the following properties of $X$:

\begin{proposition}
    \label{proposition:eguchi-hanson}
    There exist a Hyperkähler $4$-manifold $(X, \omega_1, \omega_2, \omega_3)$ and a map $\rho:X \rightarrow \mathbb{H}/\{\pm 1\}$ such that for all $t>0$, the Hyperkähler triple $t^2 \omega_1, t^2 \omega_2, t^2 \omega_3 \in \Omega^2(X)$ is a Hyperkähler ALE manifold asymptotic to $\mathbb{H}/\{\pm 1\}$ in the sense of \cref{definition:ALE} with respect to the map $\rho_t (x)=t \cdot \rho(x)$.
    Furthermore,
    \begin{align}
        \label{equation:rescaled-EH-estimate}
        \hat{\nabla}(\pi_*(t^2 \omega_j)-\hat{\omega}_j)
        =
        t^4 O(r^{-4-k}).
    \end{align}
\end{proposition}

Existence of Eguchi--Hanson space is classical, and the estimate \cref{equation:rescaled-EH-estimate} is an easy consequence.
The consequence of \cref{equation:rescaled-EH-estimate} for glued geometric structures can be found in \cite[Section 2.1]{Walp3} and \cite[Proposition 13.4.3]{Joyc5}. 

The ASD instantons on ALE spaces are well understood due to the intensive studies by Nakajima \cite{Naka2}, Kronheimer--Nakajima \cite{KrNa}, Gocho--Nakajima \cite{GoNa} and others. 
Walpuski gives an excellent overview on them in \cite[Section 2.2]{Walp3}. 
For the reader's convenience, we list some facts which we use in this paper. The following are taken from \cite[Section 2.2]{Walp3}. 

Let $\Gamma$ be a finite subgroup of $Sp(1)$, let $X$ be an ALE space asymptotic to $\C^2/\Gamma$ with projection $\rho:X \rightarrow \C^2/\Gamma$, and let $E$ be a $G$-bundle over $X$.

\begin{definition}
 A \emph{framing at infinity} of $E$ is a bundle isomorphism $\Phi:E_\infty|_U \rightarrow \rho_* E|_U$ where $E_\infty$ is a $G$-bundle over $(\C^2 \setminus \{0\})/\Gamma$ and $U$ is the complement of a compact neighbourhood of the singular point in $\C^2/\Gamma$.
\end{definition}

Let $\theta$ be a flat connection on a $G$-bundle $E_\infty$ over $(\C^2 \setminus \{0\})/\Gamma$.

\begin{definition}
 Let $\Phi: E_\infty|_U \rightarrow \rho_* E|_U$ be a framing at infinity of $E$.
 Then a connection $A \in \mathscr{A}(E)$ is called \emph{asymptotic to $\theta$ at rate $\delta$ with respect to $\Phi$} if
 \begin{align}
 \label{equation:framing-at-infinity-def}
  \nabla^k(\Phi^*A-\theta)(x)=O(|\rho(x)|^{\delta-k})
 \end{align}
 for all $k \geq 0$.
 Here $\nabla$ is the covariant derivative associated with $\theta$.
\end{definition}

We exhibit next two propositions from \cite{Walp3}.

\begin{proposition}[Proposition 2.36 in \cite{Walp3}]
\label{proposition:framed-asd-decay}
 Let $A \in \mathscr{A}(E)$ be a Yang--Mills connection on $E$ with finite energy, then there is a $G$-bundle $E_\infty$ over $(\C^2 \setminus \{0\})/\Gamma$  together with a flat connection $\theta$ and a framing $\Phi: E_\infty|_U \rightarrow \rho_* E|_U$ such that \cref{equation:framing-at-infinity-def} holds with $\delta = -3$.
\end{proposition}

\begin{proposition}[Proposition 2.41 in \cite{Walp3}]
\label{proposition:ale-asd-kernel-cokernel}
 Let $A \in \mathscr{A}(E)$ be a finite energy ASD instanton on $E$.
 Then the following holds:
 \begin{enumerate}
  \item
  If $a \in \Ker \delta_A$ decays to zero at infinity, i.e., $\displaystyle{\lim_{ r \rightarrow \infty} \sup_{\rho(x)=r} |a|(x) =0}$, then $\nabla_A^k a= \mathcal{O}(|\rho|^{-3-k})$ for all $k \geq 0$.
  
  \item
  If $(\xi, \omega) \in \Ker \delta_A^*$ decays to zero at infinity, then $(\xi, \omega)=0$.
 \end{enumerate}
\end{proposition}

We use the following explicit example of an ASD instanton on Eguchi--Hanson space by Gocho--Nakajima:

\begin{proposition}[Section 2 in \cite{GoNa}, see also Lemma 7.1 in \cite{KrNa}]
    \label{proposition:eguchi-hanson-ASD-example}
    There exists a $U(1)$-bundle on Eguchi--Hanson space carrying a connection that
    \begin{enumerate}
        \item 
        has finite energy,

        \item 
        is an infinitesimally rigid ASD instanton,

        \item 
        is asymptotic to the flat connection defined by the representation $\rho: \mathbb{Z}_2 \rightarrow U(1)$ given by $\rho(-1)=-1$.
    \end{enumerate}
\end{proposition}

\section{Flat SO(3)-connections via non-commutative representations}
\label{appendix:so3noncommutativeconnections}

In \Cref{subsection:SO3-examples}, we have described flat orbifold connections with structure group $SO(3)$ coming from a commutative representation. In this Appendix, we give a geometric description of representations that do not commute and explain why the associated flat connections cannot be used in our construction.

Note that if we assume that the image of the representation is an abelian subgroup, it is not hard to show that the image must be a subgroup of commuting involutions of $SO(3)$. The representation is then completely determined by the conditions \eqref{eq:conditions-trivial-alpha-delta-commuting} and the images of $\alpha,\beta,\gamma,\delta,\tau_4,\tau_5,\tau_8$, which must be rotations of angle $\pi$ along orthogonal axes. 

We now drop the assumption of commutativity. Conditions 4 and 5 in \Cref{lemma:conditions} imply that, for every $i\in\lbrace1,\dots,8\rbrace$, $\rho(\tau_i)$ is equal to its inverse after conjugation by a rotation. Since a rotation and its conjugate have the same angle, this implies $\rho(\tau_i)$ and its inverse are rotations with the same angle. This is only possible when the angle is $\pi$ (or $0$) and therefore $\rho(\tau_i)$ is an involution. Thus, the image of $\lbrace\tau_i\rbrace_{i=1,\dots,8}$ must still be a subgroup of commuting involutions of $SO(3)$.

Suppose now that the images of two elements $\lbrace\alpha_1,\alpha_2\rbrace\in\lbrace\alpha,\beta,\gamma,\delta\rbrace$ do not commute. Note that by condition 3 in \Cref{lemma:conditions} we can not choose $\lbrace\alpha_1,\alpha_2\rbrace\in\lbrace\alpha,\beta\rbrace$. Moreover, combining this condition with the fact that the images of all the elements are involutions we conclude that $\rho(\alpha_1)\rho(\alpha_2)$ must square to a rotation of angle $\pi$. This can only happen if $\rho(\alpha_1)$ and $\rho(\alpha_2)$ are rotations of angle $\pi$ whose product is a rotation of angle $\pi/2$ or $3\pi/2$. Given two rotations of angles $\theta_1$ and $\theta_2$ around axes $\hat{v}_1$ and $\hat{v}_2$, their product is a rotation of angle $\theta$ around an axis $\hat{v}$ given by the formulas
\begin{align*}
    \cos{\frac{\theta}{2}}=&\cos{\frac{\theta_1}{2}}\cos{\frac{\theta_2}{2}}-\sin{\frac{\theta_1}{2}}\sin{\frac{\theta_2}{2}}\hat{v}_1\cdot\hat{v}_2 \, , \\
    \sin{\frac{\theta}{2}} \hat{v}=&\sin{\frac{\theta_1}{2}}\cos{\frac{\theta_2}{2}}\hat{v}_1+\cos{\frac{\theta_1}{2}}\sin{\frac{\theta_2}{2}}\hat{v}_2+\sin{\frac{\theta_1}{2}}\sin{\frac{\theta_2}{2}}\hat{v}_1\times\hat{v}_2 \, .    
\end{align*}
These formulas are easily obtained using quaternions to represent the rotations. It is then not hard to see that $\rho(\alpha_1)\rho(\alpha_2)$ is a rotation of angle $\pi/2$ or $3\pi/2$ when $\rho(\alpha_1)$ and $\rho(\alpha_2)$ are at an angle of $\pi/4$ or $3\pi/4$. Furthermore, $\rho(\alpha_1)\rho(\alpha_2)$ is a rotation around the axis perpendicular to the axes of both $\rho(\alpha_1)$ and $\rho(\alpha_2)$.

\Cref{lemma:conditions} forces the involutions $\rho(\tau_i)$ to commute with $\rho(\alpha_1)$ and $\rho(\alpha_2)$. As a result, each $\rho(\tau_i)$ is equal to either 1 or a rotation of angle $\pi$ around the axis of rotation perpendicular to those of $\rho(\alpha_1)$ and $\rho(\alpha_2)$. In fact, \Cref{lemma:conditions} imposes further constraints on the $\rho(\tau_i)$ as well as on the $\rho(\alpha),\rho(\beta),\rho(\gamma),\rho(\delta)$, and the following proposition shows that these representations can not be used to produce $Spin(7)$-instantons in our construction.

\begin{proposition}
    Given a non-commutative representation $\rho: \pi_1^{\rm{orb}}(T^8/\Gamma) \rightarrow SO(3)$ with $\rho(\alpha)=\rho(\beta)=1$, the associated connection is always obstructed.
\end{proposition}
\begin{proof}
    We require $\rho(\alpha)=\rho(\beta)=1$, so the representation must satisfy \eqref{eq:conditions-alpha-beta-trivial}. By our previous discussion, if $\rho(\gamma)$ and $\rho(\delta)$ are non-trivial and non-commuting, their commutator must be a rotation of angle $\pi$ along an axis orthogonal to the axes of rotation of $\rho(\gamma)$ and $\rho(\delta)$. We denote such rotation by $r$, and we must have each $\rho(\tau_i)$ equal to either $1$ or $r$. In fact, from \Cref{lemma:conditions} we have $[\rho(\gamma),\rho(\delta)]=\rho(\tau_1)$, so $\rho(\tau_1)=r$ whereas we can take $\rho(\tau_4)$, $\rho(\tau_5)$ and $\rho(\tau_8)$ equal to either $1$ or $r$. 
    
    The image of the $\tau_i$ is just $\lbrace 1,r\rbrace$, so there exists an element $g\in\mathfrak{so}(3)$ fixed by the action of the $\rho(\tau_i)$ for which $\rho(\gamma)$ and $\rho(\delta)$ act with a  global minus sign. By \Cref{proposition:showing-connections-are-obstructed}, the associated connection is always obstructed.
\end{proof}

\section{The non-commutative case}
\label{appendix:non-commutative}

The following proposition shows that, without the commutativity assumption for the representatives of $\gamma$ and $\delta$, the associated flat connection is obstructed and cannot produce new $Spin(7)$-instantons via our method. 
An alternative proof in the case $G=SO(3)$ can be found in \Cref{appendix:so3noncommutativeconnections}.

\begin{proposition}
\label{proposition:no-so(n)-non-commutative-instantons}
    Consider a representation $\rho: \pi_1^{\rm{orb}}(T^8/\Gamma) \rightarrow SO(n)$ with $n\geq 3$ such that $\rho(\alpha)=\rho(\beta)=1$ and such that the representatives of $\gamma$ and $\delta$ do not commute. 
    Then the associated connection is obstructed.
\end{proposition}

\begin{proof}
    By \Cref{remark:all-are-involutions}, requiring $\rho(\alpha)=\rho(\beta)=1$ implies that $\lbrace\rho(\tau_i)\rbrace_{i=1,\dots,8}$ are commuting involutions that commute with both the involutions $\rho(\gamma)$ and $\rho(\delta)$. Since $\rho(\gamma)$ and $\rho(\delta)$ do not commute, by \Cref{lemma:conditions} $\rho(\tau_1)$ can not be equal to the identity and we have that
    \begin{equation}
    \label{eq:non-commutation-relation}
        [\rho(\gamma),\rho(\delta)]=\rho(\tau_1) \implies  \rho(\gamma)\rho(\delta)=\rho(\tau_1)\rho(\delta)\rho(\gamma) \, .
    \end{equation}
    Following \Cref{remark:involutions-diagonal-matrices}, consider a basis of $\mathbb{R}^n$ where $\lbrace\rho(\tau_i)\rbrace_{i=1,\dots,8}$ and $\rho(\gamma)$ are given by diagonal matrices with $\pm1$ entries, and such that $\rho(\tau_1)$ has $-1$ in the first $p$ diagonal entries and $1$ in the remaining $n-p$, where $0<p\leq n$. Reordering the first $p$ basis elements if necessary, we can further assume that $\rho(\gamma)$ has $-1$ in the first $q$ diagonal entries and $1$ in the following $p-q$, where $0\leq q\leq p$.

    Since $\rho(\delta)$ commutes with $\rho(\tau_1)$, it must be the case that $\rho(\delta)$ in the basis we are considering is a block diagonal matrix with two blocks $M_1$ and $M_2$ of sizes $p\times p$ and $(n-p)\times(n-p)$, matching the entries with $-1$ or $+1$ in $\rho(\tau_1)$. Note $M_1$ and $M_2$ are orthogonal matrices (with the same determinant) and they are also involutions. This implies they are symmetric matrices.
    \begin{equation}
        \rho(\delta)=\left(\begin{array}{c|c}
  M_1 & 0 \\
\hline
  0 & M_2
\end{array}\right) \, .
    \end{equation}
    We focus now on the subspace $\mathbb{R}^p$ of $\mathbb{R}^n$ given by the first $p$ coordinates. Equation \eqref{eq:non-commutation-relation} can be rewritten as $\rho(\gamma)\vert_{\mathbb{R}^p}\cdot M_1= -M_1\cdot \rho(\gamma)\vert_{\mathbb{R}^p}$, where $\rho(\gamma)\vert_{\mathbb{R}^p}$ is a block diagonal matrix with a $q\times q$ block equal to minus the identity and a $(p-q)\times(p-q)$ block equal to the identity. Decomposing $M_1$ as a symmetric block matrix matching the block structure of $\rho(\gamma)\vert_{\mathbb{R}^p}$, this equation implies that the diagonal blocks of $M_1$ must vanish. Furthermore, since $M_1$ is not singular it must be the case that the off-diagonal blocks are square matrices, this means $q=p-q$ (which forces $p=2q$) and we can write
    \begin{equation}
        M_1=\left(\begin{array}{c|c}
  0 & B \\
\hline
  B^T & 0
\end{array}\right) \, ,
    \end{equation}
    where $B$ is a $q\times q$ orthogonal matrix (not necessarily involutive). Consider the following element $g\in\mathfrak{so}(n)$, which we write as a block matrix splitting $n$ into $(q,q,n-2q)$
    \begin{equation}
        g=\left(\begin{array}{c|c|c}
  0 & -B & 0 \\
\hline
  B^T & 0 & 0 \\
\hline
  0 & 0 & 0
\end{array}\right) \, ,
    \end{equation}
    We now study the adjoint action of the elements of $\pi_1^{\rm{orb}}(T^8/\Gamma)$ on $g$. By direct computation, we see $\rho(\delta)$ acts on $g$ with a global minus sign. Since $\lbrace\rho(\tau_i)\rbrace_{i=1,\dots,8}$ and $\rho(\gamma)$ are diagonal matrices, we can restrict to the subspace $\mathbb{R}^p=\mathbb{R}^{2q}$ where $g$ is non-trivial to determine the action. Using the block decomposition of $\rho(\gamma)\vert_{\mathbb{R}^p}$ and $\rho(\tau_1)\vert_{\mathbb{R}^p}$, we find by direct computation that $\rho(\gamma)$ acts on $g$ with a global minus sign whereas $\rho(\tau_1)$ acts as the identity.

    Now, each $\rho(\tau_i)\vert_{\mathbb{R}^p}$ is a diagonal matrix that commutes with $M_1$. Equivalently, $M_1$ is fixed under conjugation by $\rho(\tau_i)\vert_{\mathbb{R}^p}$ and it is immediate to see that so is $g\vert_{\mathbb{R}^p}$. Therefore, $\rho(\tau_i)$ acts trivially on $g$ via the adjoint action.

    The element $g$ satisfies the hypothesis of \Cref{proposition:showing-connections-are-obstructed} and we conclude that the associated connection is obstructed.
\end{proof}

\section{List of SO(3)-representations on $T^8/\Gamma$}
\label{section:list-of-so3-representations-on-t8-gamma}

The following is a list of $105$ representations $\rho: \pi_1^{orb}(T^8/\Gamma)$ explained in \cref{remark:105-examples}.
We adopt the notation $(x_1,x_2,x_3,x_4,x_5)$ for the uniquely defined representation $\rho: \pi_1^{orb}(T^8/\Gamma) \rightarrow SO(3)$ satisfying $\rho(\gamma)=x_1, \rho(\delta)=x_2, \rho(\tau_4)=x_3, \rho(\tau_5)=x_4, \rho(\tau_8)=x_5$:

{\tiny 
$\{$(1, a, 1, a, b), (1, a, 1, b, a), (1, a, a, 1, b), (1, a, a, a, b), (1, a, a, b, 1), (1, a, a, b, a), (1, a, a, b, b), (1, a, b, 1, a), (1, a, b, a, 1), (1, a, b, a, a), (1, a, b, a, b), (1, a, b, b, a), (1, c, 1, a, b), (1, c, a, 1, b), (1, c, a, a, b), (1, c, a, b, 1), (1, c, a, b, a), (1, c, a, b, b), (1, c, a, b, c), (1, c, a, c, b), (1, c, c, a, b), (a, 1, a, b, a), (a, 1, a, b, b), (a, 1, b, a, 1), (a, 1, b, a, a), (a, 1, b, a, b), (a, 1, b, b, a), (a, a, a, b, a), (a, a, a, b, b), (a, a, b, a, 1), (a, a, b, a, a), (a, a, b, a, b), (a, a, b, b, a), (a, b, 1, b, a), (a, b, a, 1, b), (a, b, a, a, b), (a, b, a, b, 1), (a, b, a, b, a), (a, b, a, b, b), (a, b, b, 1, a), (a, b, b, a, 1), (a, b, b, a, a), (a, b, b, a, b), (a, b, b, b, a), (c, 1, 1, a, b), (c, 1, 1, b, c), (c, 1, 1, c, b), (c, 1, a, 1, b), (c, 1, a, a, b), (c, 1, a, b, 1), (c, 1, a, b, a), (c, 1, a, b, b), (c, 1, a, b, c), (c, 1, a, c, b), (c, 1, b, 1, c), (c, 1, c, 1, b), (c, 1, c, a, b), (c, 1, c, b, 1), (c, 1, c, c, b), (c, a, 1, a, b), (c, a, 1, b, a), (c, a, 1, b, c), (c, a, 1, c, b), (c, a, a, 1, b), (c, a, a, a, b), (c, a, a, b, 1), (c, a, a, b, a), (c, a, a, b, b), (c, a, a, b, c), (c, a, a, c, b), (c, a, b, 1, a), (c, a, b, 1, c), (c, a, b, a, 1), (c, a, b, a, a), (c, a, b, a, b), (c, a, b, a, c), (c, a, b, b, a), (c, a, b, b, c), (c, a, b, c, 1), (c, a, b, c, a), (c, a, b, c, b), (c, a, b, c, c), (c, a, c, 1, b), (c, a, c, a, b), (c, a, c, b, 1), (c, a, c, b, a), (c, a, c, b, b), (c, a, c, b, c), (c, a, c, c, b), (c, b, 1, c, b), (c, c, 1, a, b), (c, c, 1, b, c), (c, c, 1, c, b), (c, c, a, 1, b), (c, c, a, a, b), (c, c, a, b, 1), (c, c, a, b, a), (c, c, a, b, b), (c, c, a, b, c), (c, c, a, c, b), (c, c, b, 1, c), (c, c, c, 1, b), (c, c, c, a, b), (c, c, c, b, 1), (c, c, c, c, b) $\}$\par}

\section{List of SO(4)-representations on $T^8/\Gamma$}
\label{section:list-of-so4-representations-on-t8-gamma}

%The following is a list of $882$ representations $\rho: \pi_1^{\rm{orb}}(T^8/\Gamma)$ explained in \cref{remark:instanton-count-so4}.
The following is a list of $882$ representations $\rho: \pi_1^{\rm{orb}}(T^8/\Gamma)$ explained in the $SO(4)$ examples in \Cref{subsection:other-structure-groups}. Here $c$ is the diagonal matrix with entries $(-1,-1,1,1)$.
We adopt the notation $(x_1,x_2,x_3,x_4,x_5)$ for the uniquely defined representation $\rho: \pi_1^{\rm{orb}}(T^8/\Gamma) \rightarrow SO(4)$ satisfying $\rho(\gamma)=x_1, \rho(\delta)=x_2, \rho(\tau_4)=x_3, \rho(\tau_5)=x_4, \rho(\tau_8)=x_5$:

{\tiny 
$\{$(1, a, 1, a, b), (1, a, 1, b, a), (1, a, 1, -a, -b), (1, a, 1, -b, -a), (1, a, a, 1, b), (1, a, a, a, b), (1, a, a, b, 1), (1, a, a, b, a), (1, a, a, b, b), (1, a, a, b, -1), (1, a, a, b, -a), (1, a, a, b, -b), (1, a, a, -1, b), (1, a, a, -a, b), (1, a, b, 1, a), (1, a, b, a, 1), (1, a, b, a, a), (1, a, b, a, b), (1, a, b, a, -1), (1, a, b, a, -a), (1, a, b, a, -b), (1, a, b, b, a), (1, a, b, -1, a), (1, a, b, -a, a), (1, a, b, -a, -b), (1, a, b, -b, a), (1, a, b, -b, -a), (1, a, -1, a, b), (1, a, -1, b, a), (1, a, -1, -a, -b), (1, a, -1, -b, -a), (1, a, -a, 1, -b), (1, a, -a, a, b), (1, a, -a, b, a), (1, a, -a, b, -b), (1, a, -a, -1, -b), (1, a, -a, -a, -b), (1, a, -a, -b, 1), (1, a, -a, -b, -1), (1, a, -a, -b, -a), (1, a, -a, -b, -b), (1, a, -b, 1, -a), (1, a, -b, -1, -a), (1, a, -b, -a, 1), (1, a, -b, -a, -1), (1, a, -b, -a, -a), (1, a, -b, -a, -b), (1, a, -b, -b, -a), (1, c, 1, a, b), (1, c, a, a, b), (1, c, a, b, a), (1, c, a, b, c), (1, c, a, b, -1), (1, c, a, b, -a), (1, c, a, b, -c), (1, c, a, c, -b), (1, c, a, -1, b), (1, c, a, -a, b), (1, c, a, -a, -b), (1, c, a, -b, b), (1, c, a, -b, c), (1, c, a, -b, -a), (1, c, a, -c, b), (1, c, c, a, b), (1, c, c, -a, b), (1, c, -1, a, b), (1, c, -a, 1, -b), (1, c, -a, b, -c), (1, c, -a, -b, 1), (1, c, -a, -b, -b), (1, c, -c, a, b), (1, c, -c, a, -b), (1, -1, 1, a, b), (1, -1, a, 1, b), (1, -1, a, a, b), (1, -1, a, b, 1), (1, -1, a, b, a), (1, -1, a, b, b), (1, -1, a, b, -1), (1, -1, a, b, -a), (1, -1, a, b, -b), (1, -1, a, -1, b), (1, -1, a, -a, b), (1, -1, c, a, b), (1, -1, c, -a, b), (1, -1, -1, b, a), (1, -c, 1, a, b), (1, -c, a, a, b), (1, -c, a, b, 1), (1, -c, a, b, a), (1, -c, a, b, b), (1, -c, a, b, -1), (1, -c, a, c, b), (1, -c, a, -1, b), (1, -c, a, -c, -b), (1, -c, -1, a, b), (1, -c, -a, b, -b), (1, -c, -b, 1, -a), (a, 1, 1, b, a), (a, 1, 1, -a, -b), (a, 1, 1, -b, -a), (a, 1, a, 1, b), (a, 1, a, a, b), (a, 1, a, b, 1), (a, 1, a, b, a), (a, 1, a, b, b), (a, 1, a, b, -1), (a, 1, a, b, -a), (a, 1, a, b, -b), (a, 1, a, -1, b), (a, 1, a, -a, b), (a, 1, b, 1, a), (a, 1, b, a, 1), (a, 1, b, a, a), (a, 1, b, a, b), (a, 1, b, a, -1), (a, 1, b, a, -a), (a, 1, b, a, -b), (a, 1, b, b, a), (a, 1, b, -1, a), (a, 1, b, -a, a), (a, 1, b, -a, -b), (a, 1, b, -b, a), (a, 1, b, -b, -a), (a, 1, -1, a, b), (a, 1, -1, b, a), (a, 1, -1, -a, -b), (a, 1, -1, -b, -a), (a, 1, -a, 1, -b), (a, 1, -a, a, b), (a, 1, -a, b, a), (a, 1, -a, b, -b), (a, 1, -a, -1, -b), (a, 1, -a, -a, -b), (a, 1, -a, -b, 1), (a, 1, -a, -b, -1), (a, 1, -a, -b, -a), (a, 1, -a, -b, -b), (a, 1, -b, 1, -a), (a, 1, -b, -1, -a), (a, 1, -b, -a, 1), (a, 1, -b, -a, -1), (a, 1, -b, -a, -a), (a, 1, -b, -a, -b), (a, 1, -b, -b, -a), (a, a, 1, a, b), (a, a, 1, b, a), (a, a, 1, -a, -b), (a, a, 1, -b, -a), (a, a, a, 1, b), (a, a, a, a, b), (a, a, a, b, 1), (a, a, a, b, a), (a, a, a, b, b), (a, a, a, b, -1), (a, a, a, b, -a), (a, a, a, b, -b), (a, a, a, -1, b), (a, a, a, -a, b), (a, a, b, 1, a), (a, a, b, a, 1), (a, a, b, a, a), (a, a, b, a, b), (a, a, b, a, -1), (a, a, b, a, -a), (a, a, b, a, -b), (a, a, b, b, a), (a, a, b, -1, a), (a, a, b, -a, a), (a, a, b, -a, -b), (a, a, b, -b, a), (a, a, b, -b, -a), (a, a, -1, a, b), (a, a, -1, b, a), (a, a, -1, -a, -b), (a, a, -1, -b, -a), (a, a, -a, 1, -b), (a, a, -a, a, b), (a, a, -a, b, a), (a, a, -a, b, -b), (a, a, -a, -1, -b), (a, a, -a, -a, -b), (a, a, -a, -b, 1), (a, a, -a, -b, -1), (a, a, -a, -b, -a), (a, a, -a, -b, -b), (a, a, -b, 1, -a), (a, a, -b, -1, -a), (a, a, -b, -a, 1), (a, a, -b, -a, -1), (a, a, -b, -a, -a), (a, a, -b, -a, -b), (a, a, -b, -b, -a), (a, b, 1, a, b), (a, b, 1, b, a), (a, b, 1, -a, -b), (a, b, 1, -b, -a), (a, b, a, 1, b), (a, b, a, a, b), (a, b, a, b, 1), (a, b, a, b, a), (a, b, a, b, b), (a, b, a, b, -1), (a, b, a, b, -a), (a, b, a, b, -b), (a, b, a, -1, b), (a, b, a, -a, b), (a, b, a, -a, -b), (a, b, a, -b, b), (a, b, a, -b, -a), (a, b, b, 1, a), (a, b, b, a, 1), (a, b, b, a, a), (a, b, b, a, b), (a, b, b, a, -1), (a, b, b, a, -a), (a, b, b, a, -b), (a, b, b, b, a), (a, b, b, -1, a), (a, b, b, -a, a), (a, b, b, -a, -b), (a, b, b, -b, a), (a, b, b, -b, -a), (a, b, -1, a, b), (a, b, -1, b, a), (a, b, -1, -a, -b), (a, b, -1, -b, -a), (a, b, -a, 1, -b), (a, b, -a, a, b), (a, b, -a, a, -b), (a, b, -a, b, a), (a, b, -a, b, -b), (a, b, -a, -1, -b), (a, b, -a, -a, -b), (a, b, -a, -b, 1), (a, b, -a, -b, a), (a, b, -a, -b, b), (a, b, -a, -b, -1), (a, b, -a, -b, -a), (a, b, -a, -b, -b), (a, b, -b, 1, -a), (a, b, -b, a, b), (a, b, -b, a, -a), (a, b, -b, b, a), (a, b, -b, b, -a), (a, b, -b, -1, -a), (a, b, -b, -a, 1), (a, b, -b, -a, a), (a, b, -b, -a, b), (a, b, -b, -a, -1), (a, b, -b, -a, -a), (a, b, -b, -a, -b), (a, b, -b, -b, -a), (a, -1, 1, a, b), (a, -1, 1, b, a), (a, -1, 1, -a, -b), (a, -1, 1, -b, -a), (a, -1, a, 1, b), (a, -1, a, a, b), (a, -1, a, b, 1), (a, -1, a, b, a), (a, -1, a, b, b), (a, -1, a, b, -1), (a, -1, a, b, -a), (a, -1, a, b, -b), (a, -1, a, -1, b), (a, -1, a, -a, b), (a, -1, b, 1, a), (a, -1, b, a, 1), (a, -1, b, a, a), (a, -1, b, a, b), (a, -1, b, a, -1), (a, -1, b, a, -a), (a, -1, b, a, -b), (a, -1, b, b, a), (a, -1, b, -1, a), (a, -1, b, -a, a), (a, -1, b, -a, -b), (a, -1, b, -b, a), (a, -1, b, -b, -a), (a, -1, -1, a, b), (a, -1, -1, b, a), (a, -1, -1, -a, -b), (a, -1, -1, -b, -a), (a, -1, -a, 1, -b), (a, -1, -a, a, b), (a, -1, -a, b, a), (a, -1, -a, b, -b), (a, -1, -a, -1, -b), (a, -1, -a, -a, -b), (a, -1, -a, -b, 1), (a, -1, -a, -b, -1), (a, -1, -a, -b, -a), (a, -1, -a, -b, -b), (a, -1, -b, 1, -a), (a, -1, -b, -1, -a), (a, -1, -b, -a, 1), (a, -1, -b, -a, -1), (a, -1, -b, -a, -a), (a, -1, -b, -a, -b), (a, -1, -b, -b, -a), (a, -a, 1, a, b), (a, -a, 1, b, a), (a, -a, 1, -a, -b), (a, -a, 1, -b, -a), (a, -a, a, 1, b), (a, -a, a, a, b), (a, -a, a, b, 1), (a, -a, a, b, a), (a, -a, a, b, b), (a, -a, a, b, -1), (a, -a, a, b, -a), (a, -a, a, b, -b), (a, -a, a, -1, b), (a, -a, a, -a, b), (a, -a, b, 1, a), (a, -a, b, a, 1), (a, -a, b, a, a), (a, -a, b, a, b), (a, -a, b, a, -1), (a, -a, b, a, -a), (a, -a, b, a, -b), (a, -a, b, b, a), (a, -a, b, -1, a), (a, -a, b, -a, a), (a, -a, b, -a, -b), (a, -a, b, -b, a), (a, -a, b, -b, -a), (a, -a, -1, a, b), (a, -a, -1, b, a), (a, -a, -1, -a, -b), (a, -a, -1, -b, -a), (a, -a, -a, 1, -b), (a, -a, -a, a, b), (a, -a, -a, b, a), (a, -a, -a, b, -b), (a, -a, -a, -1, -b), (a, -a, -a, -a, -b), (a, -a, -a, -b, 1), (a, -a, -a, -b, -1), (a, -a, -a, -b, -a), (a, -a, -a, -b, -b), (a, -a, -b, 1, -a), (a, -a, -b, -1, -a), (a, -a, -b, -a, 1), (a, -a, -b, -a, -1), (a, -a, -b, -a, -a), (a, -a, -b, -a, -b), (a, -a, -b, -b, -a), (a, -b, 1, a, b), (a, -b, 1, b, a), (a, -b, 1, -a, -b), (a, -b, 1, -b, -a), (a, -b, a, 1, b), (a, -b, a, a, b), (a, -b, a, b, 1), (a, -b, a, b, a), (a, -b, a, b, b), (a, -b, a, b, -1), (a, -b, a, -1, b), (a, -b, b, 1, a), (a, -b, b, a, 1), (a, -b, b, a, a), (a, -b, b, a, b), (a, -b, b, a, -1), (a, -b, b, b, a), (a, -b, b, -1, a), (a, -b, -1, a, b), (a, -b, -1, b, a), (a, -b, -1, -a, -b), (a, -b, -1, -b, -a), (a, -b, -a, 1, -b), (a, -b, -a, -1, -b), (a, -b, -a, -a, -b), (a, -b, -a, -b, 1), (a, -b, -a, -b, -1), (a, -b, -a, -b, -a), (a, -b, -a, -b, -b), (a, -b, -b, 1, -a), (a, -b, -b, -1, -a), (a, -b, -b, -a, 1), (a, -b, -b, -a, -1), (a, -b, -b, -a, -a), (a, -b, -b, -a, -b), (a, -b, -b, -b, -a), (b, 1, 1, b, a), (c, 1, 1, a, b), (c, 1, a, 1, b), (c, 1, a, b, 1), (c, 1, a, b, b), (c, 1, a, b, c), (c, 1, a, b, -1), (c, 1, a, b, -a), (c, 1, a, b, -b), (c, 1, a, c, b), (c, 1, a, -1, b), (c, 1, a, -a, -b), (c, 1, a, -b, c), (c, 1, a, -b, -c), (c, 1, a, -c, b), (c, 1, a, -c, -b), (c, 1, c, a, b), (c, 1, c, a, -b), (c, 1, -1, a, b), (c, 1, -a, b, a), (c, 1, -a, -a, -b), (c, 1, -a, -b, -a), (c, 1, -c, a, b), (c, 1, -c, -a, b), (c, a, 1, b, a), (c, a, 1, b, c), (c, a, 1, -a, -b), (c, a, 1, -b, -c), (c, a, a, 1, b), (c, a, a, a, b), (c, a, a, b, 1), (c, a, a, b, a), (c, a, a, b, b), (c, a, a, b, -1), (c, a, a, b, -b), (c, a, a, b, -c), (c, a, a, c, b), (c, a, a, -1, b), (c, a, a, -a, -b), (c, a, a, -b, b), (c, a, a, -b, c), (c, a, a, -b, -c), (c, a, a, -c, -b), (c, a, b, 1, a), (c, a, b, 1, c), (c, a, b, a, 1), (c, a, b, a, a), (c, a, b, a, b), (c, a, b, a, c), (c, a, b, a, -1), (c, a, b, a, -a), (c, a, b, a, -c), (c, a, b, b, c), (c, a, b, c, 1), (c, a, b, c, a), (c, a, b, c, b), (c, a, b, c, c), (c, a, b, c, -1), (c, a, b, c, -a), (c, a, b, c, -c), (c, a, b, -1, a), (c, a, b, -1, c), (c, a, b, -a, c), (c, a, b, -a, -c), (c, a, b, -b, c), (c, a, b, -b, -c), (c, a, b, -c, c), (c, a, b, -c, -b), (c, a, c, 1, b), (c, a, c, a, b), (c, a, c, b, 1), (c, a, c, b, b), (c, a, c, b, c), (c, a, c, b, -1), (c, a, c, b, -b), (c, a, c, b, -c), (c, a, c, c, b), (c, a, c, -1, b), (c, a, c, -a, b), (c, a, c, -b, a), (c, a, c, -b, -a), (c, a, c, -b, -c), (c, a, c, -c, b), (c, a, -1, a, b), (c, a, -1, b, a), (c, a, -1, b, c), (c, a, -1, -a, -b), (c, a, -1, -b, -a), (c, a, -1, -b, -c), (c, a, -1, -c, -b), (c, a, -a, 1, -b), (c, a, -a, a, b), (c, a, -a, a, -b), (c, a, -a, b, c), (c, a, -a, b, -c), (c, a, -a, c, -b), (c, a, -a, -1, -b), (c, a, -a, -a, -b), (c, a, -a, -b, 1), (c, a, -a, -b, a), (c, a, -a, -b, c), (c, a, -a, -b, -1), (c, a, -a, -b, -a), (c, a, -a, -b, -c), (c, a, -a, -c, b), (c, a, -a, -c, -b), (c, a, -b, 1, -c), (c, a, -b, a, c), (c, a, -b, a, -c), (c, a, -b, b, c), (c, a, -b, b, -a), (c, a, -b, b, -c), (c, a, -b, c, a), (c, a, -b, c, b), (c, a, -b, c, -c), (c, a, -b, -1, -c), (c, a, -b, -a, 1), (c, a, -b, -a, a), (c, a, -b, -a, c), (c, a, -b, -a, -1), (c, a, -b, -a, -a), (c, a, -b, -a, -c), (c, a, -b, -b, -c), (c, a, -b, -c, 1), (c, a, -b, -c, a), (c, a, -b, -c, c), (c, a, -b, -c, -a), (c, a, -b, -c, -b), (c, a, -b, -c, -c), (c, a, -c, 1, -b), (c, a, -c, a, b), (c, a, -c, b, a), (c, a, -c, b, -a), (c, a, -c, b, -b), (c, a, -c, c, b), (c, a, -c, -1, -b), (c, a, -c, -a, b), (c, a, -c, -a, -b), (c, a, -c, -b, a), (c, a, -c, -b, c), (c, a, -c, -b, -1), (c, a, -c, -b, -a), (c, a, -c, -b, -b), (c, a, -c, -b, -c), (c, c, 1, a, b), (c, c, a, a, b), (c, c, a, b, a), (c, c, a, b, c), (c, c, a, b, -1), (c, c, a, b, -a), (c, c, a, b, -c), (c, c, a, c, -b), (c, c, a, -1, b), (c, c, a, -a, b), (c, c, a, -a, -b), (c, c, a, -b, b), (c, c, a, -b, -a), (c, c, a, -c, b), (c, c, c, a, b), (c, c, c, -a, b), (c, c, -1, a, b), (c, c, -a, 1, -b), (c, c, -a, b, -c), (c, c, -a, -b, 1), (c, c, -a, -b, -b), (c, c, -c, a, b), (c, c, -c, a, -b), (c, -1, 1, a, b), (c, -1, a, a, b), (c, -1, a, b, a), (c, -1, a, b, c), (c, -1, a, b, -1), (c, -1, a, b, -a), (c, -1, a, b, -c), (c, -1, a, c, b), (c, -1, a, c, -b), (c, -1, a, -1, b), (c, -1, a, -a, b), (c, -1, a, -b, b), (c, -1, a, -b, -a), (c, -1, a, -c, b), (c, -1, a, -c, -b), (c, -1, c, a, b), (c, -1, c, -a, b), (c, -1, -1, a, b), (c, -1, -a, 1, -b), (c, -1, -a, b, -c), (c, -1, -a, -b, 1), (c, -1, -a, -b, -b), (c, -1, -c, a, b), (c, -1, -c, a, -b), (c, -a, 1, b, a), (c, -a, 1, b, c), (c, -a, 1, -a, -b), (c, -a, 1, -b, -c), (c, -a, 1, -c, -b), (c, -a, a, b, b), (c, -a, a, c, -b), (c, -a, a, -b, b), (c, -a, b, 1, a), (c, -a, b, 1, c), (c, -a, b, a, b), (c, -a, b, b, a), (c, -a, b, c, 1), (c, -a, b, c, a), (c, -a, b, c, c), (c, -a, b, c, -1), (c, -a, b, -1, a), (c, -a, b, -1, c), (c, -a, b, -a, -b), (c, -a, b, -c, -b), (c, -a, c, 1, b), (c, -a, c, b, 1), (c, -a, c, b, b), (c, -a, c, b, c), (c, -a, c, b, -1), (c, -a, c, b, -b), (c, -a, c, -1, b), (c, -a, c, -b, a), (c, -a, c, -c, b), (c, -a, -1, b, c), (c, -a, -1, c, b), (c, -a, -1, -b, -c), (c, -a, -1, -c, -b), (c, -a, -a, a, -b), (c, -a, -a, c, b), (c, -a, -a, -b, c), (c, -a, -a, -c, -b), (c, -a, -b, a, b), (c, -a, -b, b, -a), (c, -a, -b, -1, -c), (c, -a, -b, -a, b), (c, -a, -b, -b, -a), (c, -a, -b, -b, -c), (c, -a, -b, -c, a), (c, -a, -b, -c, -1), (c, -a, -b, -c, -b), (c, -a, -b, -c, -c), (c, -a, -c, 1, -b), (c, -a, -c, b, -b), (c, -a, -c, -a, b), (c, -a, -c, -b, 1), (c, -a, -c, -b, c), (c, -a, -c, -b, -b), (c, -a, -c, -b, -c), (c, -a, -c, -c, -b), (c, -c, 1, a, b), (c, -c, a, a, b), (c, -c, a, b, a), (c, -c, a, b, c), (c, -c, a, b, -1), (c, -c, a, b, -a), (c, -c, a, b, -c), (c, -c, a, c, b), (c, -c, a, c, -b), (c, -c, a, -1, b), (c, -c, a, -a, b), (c, -c, a, -b, b), (c, -c, a, -b, -a), (c, -c, a, -c, b), (c, -c, a, -c, -b), (c, -c, c, a, b), (c, -c, c, -a, b), (c, -c, -1, a, b), (c, -c, -a, 1, -b), (c, -c, -a, b, -c), (c, -c, -a, -b, 1), (c, -c, -a, -b, -b), (c, -c, -c, a, b), (c, -c, -c, a, -b), (-1, 1, 1, a, b), (-1, 1, a, 1, b), (-1, 1, a, a, b), (-1, 1, a, b, 1), (-1, 1, a, b, a), (-1, 1, a, b, b), (-1, 1, a, b, -1), (-1, 1, a, b, -a), (-1, 1, a, b, -b), (-1, 1, a, -1, b), (-1, 1, a, -a, b), (-1, 1, c, a, b), (-1, 1, c, -a, b), (-1, 1, -1, b, a), (-1, a, 1, b, a), (-1, a, 1, -a, -b), (-1, a, 1, -b, -a), (-1, a, a, 1, b), (-1, a, a, a, b), (-1, a, a, b, 1), (-1, a, a, b, a), (-1, a, a, b, b), (-1, a, a, b, -1), (-1, a, a, b, -a), (-1, a, a, b, -b), (-1, a, a, -1, b), (-1, a, a, -a, b), (-1, a, b, 1, a), (-1, a, b, a, 1), (-1, a, b, a, a), (-1, a, b, a, b), (-1, a, b, a, -1), (-1, a, b, a, -a), (-1, a, b, a, -b), (-1, a, b, b, a), (-1, a, b, -1, a), (-1, a, b, -a, a), (-1, a, b, -a, -b), (-1, a, b, -b, a), (-1, a, b, -b, -a), (-1, a, -1, a, b), (-1, a, -1, b, a), (-1, a, -1, -a, -b), (-1, a, -1, -b, -a), (-1, a, -a, 1, -b), (-1, a, -a, a, b), (-1, a, -a, b, a), (-1, a, -a, b, -b), (-1, a, -a, -1, -b), (-1, a, -a, -a, -b), (-1, a, -a, -b, 1), (-1, a, -a, -b, -1), (-1, a, -a, -b, -a), (-1, a, -a, -b, -b), (-1, a, -b, 1, -a), (-1, a, -b, -1, -a), (-1, a, -b, -a, 1), (-1, a, -b, -a, -1), (-1, a, -b, -a, -a), (-1, a, -b, -a, -b), (-1, a, -b, -b, -a), (-1, b, 1, b, a), (-1, c, 1, a, b), (-1, c, a, 1, b), (-1, c, a, b, 1), (-1, c, a, b, b), (-1, c, a, b, c), (-1, c, a, b, -1), (-1, c, a, b, -b), (-1, c, a, c, b), (-1, c, a, c, -b), (-1, c, a, -1, b), (-1, c, a, -a, b), (-1, c, a, -a, -b), (-1, c, a, -b, b), (-1, c, a, -b, c), (-1, c, a, -b, -c), (-1, c, a, -c, -b), (-1, c, c, a, -b), (-1, c, -1, a, b), (-1, c, -a, b, a), (-1, c, -a, -a, -b), (-1, c, -a, -b, -a), (-1, c, -c, a, b), (-1, c, -c, -a, b), (-1, -1, 1, b, a), (-1, -1, a, a, b), (-1, -1, a, b, 1), (-1, -1, a, b, a), (-1, -1, a, b, b), (-1, -1, a, b, -1), (-1, -1, a, b, -a), (-1, -1, a, b, -b), (-1, -1, a, -1, b), (-1, -1, a, -a, b), (-1, -1, b, 1, a), (-1, -1, c, a, b), (-1, -1, c, -a, b), (-1, -1, -1, -c, -b), (-1, -c, 1, a, b), (-1, -c, a, 1, b), (-1, -c, a, a, b), (-1, -c, a, b, 1), (-1, -c, a, b, a), (-1, -c, a, b, -1), (-1, -c, a, c, -b), (-1, -c, a, -1, b), (-1, -c, a, -b, -a), (-1, -c, -1, -a, -b), (-1, -c, -a, b, c), (-1, -c, -a, -b, -b), (-1, -c, -c, -a, b), (-a, 1, 1, c, b), (-c, 1, a, 1, b), (-c, 1, a, b, 1), (-c, 1, a, b, a), (-c, 1, a, b, b), (-c, 1, a, b, -1), (-c, 1, a, b, -b), (-c, 1, a, -1, b), (-c, 1, a, -b, c), (-c, 1, -1, a, b), (-c, 1, -a, a, b), (-c, 1, -a, -a, -b), (-c, 1, -a, -c, -b), (-c, a, 1, a, b), (-c, a, 1, c, b), (-c, a, 1, -a, -b), (-c, a, 1, -c, -b), (-c, a, a, b, 1), (-c, a, a, b, b), (-c, a, a, b, -1), (-c, a, a, b, -a), (-c, a, a, -1, b), (-c, a, a, -b, -a), (-c, a, b, 1, a), (-c, a, b, a, a), (-c, a, b, a, b), (-c, a, b, a, -a), (-c, a, b, b, a), (-c, a, b, c, -1), (-c, a, b, -1, a), (-c, a, b, -b, a), (-c, a, c, b, 1), (-c, a, c, c, b), (-c, a, c, -c, b), (-c, a, -1, b, a), (-c, a, -1, -a, -b), (-c, a, -1, -b, -a), (-c, a, -1, -c, -b), (-c, a, -a, b, -b), (-c, a, -a, -1, -b), (-c, a, -a, -a, -b), (-c, a, -a, -b, 1), (-c, a, -a, -b, -1), (-c, a, -a, -b, -a), (-c, a, -a, -b, -b), (-c, a, -b, 1, -a), (-c, a, -b, b, -a), (-c, a, -b, -1, -a), (-c, a, -b, -a, a), (-c, a, -b, -a, -1), (-c, a, -b, -a, -a), (-c, a, -b, -b, -a), (-c, a, -b, -c, b), (-c, a, -c, a, b), (-c, a, -c, -b, a), (-c, c, 1, -a, -b), (-c, c, a, 1, b), (-c, c, a, b, 1), (-c, c, a, b, b), (-c, c, a, b, -1), (-c, c, a, b, -c), (-c, c, a, -1, b), (-c, c, a, -b, b), (-c, c, -1, a, b), (-c, c, -a, a, -b), (-c, c, -a, -a, -b), (-c, c, -a, -b, -a), (-c, -1, 1, -a, -b), (-c, -1, a, 1, b), (-c, -1, a, b, 1), (-c, -1, a, b, b), (-c, -1, a, b, -1), (-c, -1, a, b, -c), (-c, -1, a, -1, b), (-c, -1, a, -b, b), (-c, -1, -1, a, b), (-c, -1, -a, a, -b), (-c, -1, -a, -a, -b), (-c, -1, -a, -b, -a), (-c, -a, 1, b, a), (-c, -a, 1, -b, -a), (-c, -a, 1, -c, -b), (-c, -a, a, 1, b), (-c, -a, a, b, -a), (-c, -a, b, 1, c), (-c, -a, b, a, 1), (-c, -a, b, a, b), (-c, -a, b, c, 1), (-c, -a, b, c, -a), (-c, -a, c, b, -1), (-c, -a, c, -1, b), (-c, -a, -1, -a, -b), (-c, -a, -a, 1, -b), (-c, -a, -a, -a, -b), (-c, -a, -a, -b, -a), (-c, -a, -b, a, b), (-c, -a, -b, -a, 1), (-c, -a, -b, -a, -1), (-c, -a, -b, -b, -c), (-c, -a, -b, -c, -b), (-c, -a, -c, a, -b), (-c, -a, -c, -c, -b), (-c, -c, 1, a, b), (-c, -c, a, 1, b), (-c, -c, a, a, b), (-c, -c, a, b, 1), (-c, -c, a, b, a), (-c, -c, a, b, b), (-c, -c, a, b, -1), (-c, -c, a, b, -c), (-c, -c, a, -1, b), (-c, -c, a, -b, b), (-c, -c, -1, a, b), (-c, -c, -a, c, -b), (-c, -c, -a, -c, b)$\}$\par}

%%%%%%%%%%%%%%%%%%%%%%%%%%%%%%%%%%%%%%%%%%%%%%%%%%%%%%%%%%%%%%%%%%%%%%%%%%%%%%%%%%%%%%%

\vskip 8pt

\noindent{\small\sc Department of Physics, Astronomy and Mathematics, University of Hertfordshire, 
College Lane, Hatfield, 
AL10 9AB, United Kingdom

\noindent E-mail: {\tt m.galdeano@herts.ac.uk}}

\vskip 8pt

\noindent{\small\sc Department of Mathematics, Imperial College London, Exhibition Rd, London SW7 2BX, United Kingdom;

\vskip 2pt 

\noindent I-X Centre for AI In Science, Imperial College London, White City Campus, 84 Wood Lane, London W12 0BZ, United Kingdom

\noindent E-mail: {\tt daniel.platt.berlin@gmail.com}}

\vskip 8pt

\noindent{\small\sc Beijing Institute of Mathematical Sciences and Applications (BIMSA), No. 544, Hefangkou Village, Huaibei Town, Huairou District, Beijing 101408, China

\noindent E-mail: {\tt ytanaka@bimsa.cn}}

\vskip 8pt

\noindent{\small\sc Institute for Advanced Study, Princeton, NJ 08540, U.S.A. 

\noindent E-mail: {\tt luyawang@ias.edu}}

%%%%%%%%%%%%%%%%%%%%%%%%%%%%%%%%%%%%%%%%%%%%%%%%%%%%%%%%%%%%%%%%%%%%%%%%%%%%%%%%%%%%%%%%%%%%%%%%%%%%%%

\end{document}